\documentclass[12pt]{amsart}
\usepackage[margin=1in]{geometry}
\usepackage{amsfonts,amssymb,stmaryrd,amscd,amsmath,latexsym,amsbsy,amsthm}
\usepackage{hyperref}
\usepackage{mathtools}
\usepackage{enumitem}
\usepackage{multicol}

\newtheorem{theorem}{Theorem}[section]
\newtheorem{lemma}[theorem]{Lemma}
\newtheorem{proposition}[theorem]{Proposition}
\newtheorem{corollary}[theorem]{Corollary}
\theoremstyle{definition}
\newtheorem{definition}[theorem]{Definition}
\newtheorem{example}[theorem]{Example}

\newtheorem{remark}[theorem]{Remark}

\newcommand{\g}{\mathfrak{g}}

\newcommand{\h}{\mathfrak{h}}

\newcommand{\C}{{\mathbb C}}

\newcommand{\ben}{\begin{enumerate}}
\newcommand{\een}{\end{enumerate}}

\newcommand{\nc}{\newcommand}
\nc{\on}{\operatorname}
\nc{\wh}{\widehat}
\nc{\wt}{\widetilde}
\nc{\sw}{{\mathfrak s}{\mathfrak l}}
\nc{\ghat}{\wh{\g}}
\nc{\hhat}{\wh{\h}}
\nc{\mc}{\mathcal}
\nc{\Bun}{\on{Bun}}
\nc{\ol}{\overline}
\nc{\OO}{\mathcal O}
\nc{\pone}{{\mathbb P}^1}
\nc{\pa}{\partial}
\nc{\Pic}{\on{Pic}}
\nc{\ga}{\gamma}
\nc{\orr}{\underline}
\nc{\mbb}{\mathbb}
\nc{\mbf}{\mathbf}
\nc{\V}{{\mc V}}

\newcommand{\norm}[1]{\left\lVert#1\right\rVert}
\newcommand{\bmu}{{\mbox{\boldmath$\mu$}}}

\theoremstyle{plain}

\newtheorem*{sol}{Solution}

\theoremstyle{definition}

\theoremstyle{remark}

\usepackage{color}

\newcommand{\solu}[1]{\begin{sol}{\bf (\ref{#1})}}

\newcommand{\ho}{\hookrightarrow}
\nc{\Bunc}{\on{Bun}^{\circ}}

\def\g{\mathfrak{g}}

\def\h{\mathfrak{h}}

\def\M{\mathcal{M}}

\def\LG{{}^L\hspace*{-0.4mm}G}

\nc{\ppart}{(\!(t)\!)}
\nc{\zpart}{(\!(z)\!)}
\nc{\hl}{h_{\ell}}
\nc{\hr}{h_{r}}
\nc{\mb}{\mathbf}

\numberwithin{equation}{section}

\begin{document}

\title[Analytic Langlands correspondence for $PGL_2$ on $\Bbb P^1$]{Analytic Langlands correspondence for $PGL_2$ on $\Bbb P^1$ with parabolic structures over local fields}

\author{Pavel Etingof}

\address{Department of Mathematics, MIT, Cambridge, MA 02139, USA}

\author{Edward Frenkel}

\address{Department of Mathematics, University of California,
  Berkeley, CA 94720, USA}

\author{David Kazhdan}

\address{Einstein Institute of Mathematics, Edmond J. Safra Campus,
  Givaat Ram The Hebrew University of Jerusalem, Jerusalem, 91904,
  Israel}
  
\maketitle

\begin{abstract} We continue to develop the analytic Langlands program for curves over local fields initiated in \cite{EFK,EFK2} following a suggestion of Langlands and a work of Teschner. Namely, we study the Hecke operators introduced in \cite{EFK2} in the case of $\Bbb P^1$ over a local field with parabolic structures at finitely many points for the group $PGL_2$. We establish most of the conjectures of \cite{EFK,EFK2} in this case.
\end{abstract} 

\tableofcontents

\section{Introduction} 

In our previous papers \cite{EFK,EFK2},
motivated by a suggestion of Langlands (\cite{La}) and a work of Teschner (\cite{Te}), we proposed an ``analytic Langlands program" for curves defined over local fields. In particular, in \cite{EFK2} we constructed analogues of the Hecke operators for the moduli space of stable
$G$-bundles on a smooth projective curve $X$ over a local field $F$ with parabolic
structures at finitely many points. We conjectured that they define
commuting compact normal operators on the Hilbert space of
half-densities on this moduli space. In the case $F=\C$, we also
conjectured that their joint spectrum is essentially in bijection with
the set of $\LG$-opers on $X$ with real monodromy. Moreover, we conjectured an explicit formula relating the
eigenvalues of the Hecke operators and the global differential
operators studied in \cite{EFK}.

The main goal of this paper is to prove the conjectures of \cite{EFK,EFK2} for $G=PGL_2$ in genus 0 with parabolic points. In particular, we establish a spectral decomposition for Hecke operators acting on the Hilbert space $L^2(Bun^\circ(F))$ of square integrable complex half-densities on the analytic manifold $Bun^\circ(F)$ of isomorphism classes of stable quasiparabolic $PGL_2$-bundles on $\Bbb P^1$ with $m+2$ marked points defined over a local field $F$. We also study the corresponding eigenfunctions and eigenvalues. Thus this paper implements the analytic Langlands correspondence for $G=PGL_2$ in genus 0.

The content of the paper is as follows. 

After setting up preliminaries for a general curve $X$ and $G=PGL_2$ (Section \ref{preli}), we focus on the case $X=\Bbb P^1$ with parabolic points $t_0,...,t_{m+1}$. 
We first define birational parametrizations of the moduli spaces $Bun^\circ_0$ 
and $Bun^\circ_1$ of bundles of degree $0$ and $1$ (Subsection \ref{birpar}), and observe that $Bun^\circ$ has a natural action of the group $(\Bbb Z/2)^{m+2}$, whose generators $S_i,i=0,...,m+1$ switch $Bun^\circ_0$ with $Bun^\circ_1$. This yields an action of $\Bbb V:=(\Bbb Z/2)^{m+1}$ on each component $Bun^\circ_j$, $j=0,1$ by $S_i\mapsto S_iS_{m+1}$. 

This allows us to explicitly describe the Hecke correspondence (Subsection \ref{heckecorrr}) and derive an explicit formula for the Hecke operator $H_x$ in this case (Subsection \ref{Hecke operators}). More precisely, we identify $L^2(Bun^\circ_0(F))$ with $L^2(Bun^\circ_1(F))$ using the map $S_{m+1}$, which allows us to view $H_x$ as an (initially, densely defined) operator on the Hilbert space $\mathcal H:=L^2(Bun^\circ_0(F))$, and we write a formula for this operator in terms of the birational parametrization of $Bun^\circ_0$. We then use this formula and basic representation theory of $PGL_2(F)$ to prove the compactness of $H_x$ 
(Subsections \ref{Boundedness of Hecke operators}, \ref{Compactness of Hecke operators}). Then we compute the asymptotics of the Hecke operators when $x$ approaches one of the parabolic points, and show that the leading 
terms of this asymptotics are given by the action of the generators 
$S_i$ of the group $(\Bbb Z/2)^{m+2}$ on $L^2(Bun^\circ_0(F))$, with $S_{m+1}\mapsto 1$ (Subsection \ref{Asymptotics of Hecke operators}). This implies that the common kernel of all $H_x$ vanishes, hence we have a spectral decomposition of $L^2(Bun^\circ_0(F))$ into their finite dimensional joint eigenspaces (Subsection \ref{The spectral theorem}). 

Since the Hecke operators commute, their product $H_{x_1}...H_{x_n}$ is symmetric in $x_1,...,x_n$, but the formula for this product arising from the definition of $H_x$ is not manifestly symmetric. Using the Cauchy-Jacobi interpolation formula for rational functions, we give a manifestly symmetric formula for this product in genus $0$. This formula can then be extended to the case when $(x_1,...,x_n)\in S^nX(F)$ but individual coordinates $x_i$ are not necessarily defined over $F$ (Subsection \ref{A formula for the product of Hecke operators}).  

In the archimedian case ($F=\Bbb R,\Bbb C$) we reprove by an explicit computation (for $X=\Bbb P^1$) the result from \cite{EFK2} showing that Hecke operators $H_x$ commute with quantum Hitchin (i.e., Gaudin) hamiltonians (Subsection \ref{Schwartz space}) and satisfy a second order ODE with respect to $x$ -- an operator version of the oper equation (Subsection 
\ref{Differential equations for Hecke operators}). This implies that each 
eigenvalue of $H_x$ is a solution of an $SL_2$-oper with respect to $x$ (Subsection \ref{The differential equation for eigenvalues}). This gives rise  to natural commuting normal extensions  of Gaudin hamiltonians, yielding their joint spectral decomposition (namely, the same decomposition as for Hecke operators), which confirms Conjecture 1.5 of \cite{EFK} and its Corollary 1.6. 

From this decomposition we deduce that for $F=\Bbb C$
the spectrum $\Sigma$ of Hecke operators is simple. Moreover, there is a natural injective map from $\Sigma$ to a subset of the set $\mathcal{R}$ of $SL_2$-opers with real
monodromy, as conjectured in \cite{EFK}, Conjecture 1.8,(1) (Subsection \ref{The complex case}). Conjecturally, this map is bijective (i.e., $\Sigma\cong \mathcal{R}$), and we prove this for 4 and 5 points. Moreover, we express the eigenvalues of the Hecke operators as bilinear combinations of the solutions of the second order differential equation representing the corresponding oper, which proves Conjecture 1.11 of \cite{EFK2} in this case.\footnote{\label{foot1}More precisely, Conjecture 1.11 of \cite{EFK2} concerns the Hecke operators on $L^2(Bun^\circ(F))=L^2(Bun^\circ_0(F))\oplus L^2(Bun^\circ_1(F))$, which have the form $\begin{pmatrix} 0 & H_x\\ H_x & 0\end{pmatrix}$. The eigenvalues of these operators thus come in pairs $(\beta_k(x),-\beta_k(x))$, where $\beta_k(x)$ are the eigenvalues of $H_x$ acting on $L^2(Bun^\circ_0(F))$ (using the identification $S_{m+1}: L^2(Bun^\circ_0(F))\to L^2(Bun^\circ_1(F))$). So each real oper (which occurs in the spectrum) defines two eigenvalues of Hecke operators on $L^2(Bun^\circ(F))$ differing by sign, exactly as stated in \cite{EFK2}, Conjecture 1.11.}
 
To describe the spectrum of Hecke operators for $F=\Bbb R$ (Subsection \ref{The real case}), we introduce the notion of a {\bf balancing} of an  
$SL_2$ local system on $\Bbb C\Bbb P^1\setminus \lbrace t_0,...,t_{m+1}\rbrace$. An  $SL_2$ local system admits at most two balancings, and when it does, 
then generically only one. The space of local systems that admit a balancing is a middle-dimensional subvariety of the variety of all local systems, which we identify with the space of solutions of the {\bf T-system of type $A_1$}. 
 
 Let $\mathcal B$ be  the set of balanced local systems that come from oper connections.  It is equipped with a natural, at most 2-to-1 map to the space of opers (1-to-1 for generic positions of parabolic points), whose image is expected to be discrete. In Subsection \ref{The real case} we show 
that the spectrum of Hecke operators in genus zero for $F=\Bbb R$ can be realized as a subset of $\mathcal B$.    

While in general we do not expect eigenfunctions or eigenvalues of Hecke operators to be explicitly computable in terms of special functions of hypergeometric type,\footnote{E.g., in the case of $4$ points they express via Lam\'e functions with parameter $-1/2$, which don't have an explicit integral representation for generic eigenvalue (see \cite{De,CC}).} in some special cases this is possible. As an example, we compute the largest eigenvalue of the Hecke operator in the case when the configuration of parabolic points admits a cyclic symmetry, over $F=\Bbb R$ and $\Bbb C$. In these cases solutions of the 
oper equation are expressible via the classical hypergeometric function, which gives a hypergeometric expression for this eigenvalue (Subsection \ref{hyge}). 

We then proceed to study in detail the simplest nontrivial special case $m=2$, when we have $4$ parabolic points,  over a general local field $F$ (Section \ref{Four parabolic points}). Then the varieties of stable bundles of degree $0$ and $1$ are $Bun_0^\circ=Bun_1^\circ=\Bbb P^1\setminus \lbrace 0,t,1,\infty\rbrace$, the Hitchin hamiltonian is the Lam\'e operator with parameter $-1/2$, and eigenfunctions and eigenvalues of Hecke operators can be written in terms of Lam\'e functions with this parameter. In this case the Schwartz kernel $K(x,y,z)$ 
of the suitably normalized Hecke operator $\Bbb H_x$ is an explicit locally $L^1$-function, given by the formula from \cite{K} (Subsection \ref{The Hecke correspondence and Hecke operators}). We also  
reprove the compactness of Hecke operators by direct analysis of this kernel (Subsection \ref{Boundedness and compactness}).\footnote{In the case $F=\Bbb C$, $X=\Bbb P^1$ with $4$ parabolic points and $G=PGL_2$ the existence of the natural normal extension for the Hitchin hamiltonian (=Lam\'e operator) was shown in \cite{EFK}, but Hecke operators, which are the main new objects in the present paper, provide a much simpler proof.}

It turns out that the kernel $K(x,y,z)$ is symmetric not just under the switch of $y$ and $z$ (which is equivalent to the fact that the operator $\Bbb H_x$ is symmetric) but has the full $\Bbb S_3$-symmetry. This is a special feature of the 4-point case which is related to the Okamoto symmetries of the Painlev\'e VI equation. As a result, eigenvalues of $\Bbb H_x$ as functions of $x$ and eigenvectors of $\Bbb H_x$ are actually the same functions. In the case $F=\Bbb C$ they are single-valued  solutions of the 
Lam\'e equation real analytic outside the four singular points, as considered in \cite{Be}, and for $F=\Bbb R$ they are solutions satisfying appropriate gluing conditions at the real parabolic points. Moreover, in the 
case $F=\Bbb C$ we show that the spectrum of the Hecke operators coincides with the full set of Lam\'e opers with real monodromy. We also describe the spectrum in the case $F=\Bbb R$ in terms of a suitable Sturm-Liouville problem for the Lam\'e operator (Subsections \ref{archim}, \ref{archim1}). Finally, in Subsections \ref{subleading}, \ref{ruij} we compute the subleading term 
of the asymptotics of the Hecke operator $H_x$ and explain how it is connected to the work of Ruijsenaars \cite{Ru}. 

In Section \ref{nonar} we study the example of $4$ points over a non-archimedian local field. Namely, in Subsection \ref{nonar1} we give a proof of the statement from \cite{K} that the eigenvalues of Hecke operators are algebraic numbers. It is expected that this holds not just in this example but for a general group and general curve. 
In Subsection \ref{nonar2} we compute the ``first batch" of eigenvalues of Hecke operators, and in 
Subection \ref{nonar3} we show that (except for $5$ special eigenvalues) they are the same 
as eigenvalues of the usual Hecke operators over the finite (residue) field. This agrees with predictions in \cite{K}. 

In Section \ref{sing} we study the behavior of eigenfunctions near their singularities in the archimedian case, using the quantum Gaudin system (we expect the same type of singularities over a non-archimedian field, even though there is no obvious analogue of the Gaudin system in that case). These singularities occur on the so-called {\bf wobbly divisor}, which is the divisor of bundles that admit a nonzero nilpotent Higgs field. Namely, in Subsection \ref{sing1} we compute
the local behavior of solutions of the quantum Gaudin system near a generic point of the wobbly divisor. This allows us to describe the local behavior of eigenfunctions and monodromy (Subsection \ref{sing2}). 
Moreover, in the case of $5$ points we prove that eigenfunctions of Hecke 
operators are continuous (but not differentiable) with square root singularities near the wobbly divisor (which has normal crossings), 
and give a geometric description of the Schwartz space (Subsection \ref{The case of 5 points}). This settles all the main conjectures from \cite{EFK},\cite{EFK2} over the complex field in the case of $5$ points.\footnote{For $N>5$ points, it still remains to be proved that all single-valued eigenfunctions of Gaudin operators (as distributions) are in $L^2$, and thus every real oper defines a point in the spectrum of Hecke operators.}

Finally, in the appendix we collect auxiliary results. 

{\bf Acknowledgements.} We are grateful to M. Kontsevich for sharing his ideas on Hecke operators and to A. Braverman, V. Drinfeld, D. Gaiotto, D. 
Gaitsgory and E. Witten for useful discussions. We thank T. Pantev for his help with the proof of Prop. \ref{dpan}. P. E.'s work was partially supported by the NSF grant DMS - 1916120. The project has received funding from ERC under grant agreement No. 669655.

\section{Preliminaries} \label{preli}

\subsection{Measures on analytic manifolds over local fields} 
Let $F$ be a local field with absolute value $x\mapsto \norm{x}$ (i.e., the Haar measure on $F$ multiplies by $\norm{\lambda}$ under rescaling by $\lambda\in F$). 
For instance, for $F=\Bbb R$ we have $\norm{x}=|x|$, for $F=\Bbb C$ 
we have $\norm{x}=|x|^2$, and for $F=\Bbb Q_p$ we have $\norm{x}=p^{-v(x)}$, where $v(x)$ is the $p$-adic valuation of $x$. 

Let $\norm{dx}$ denote the Haar measure on $F$ normalized so that 
$$
\int_{1\le \norm{x}<R}\frac{\norm{dx}}{\norm{x}}=\log R
$$
where $R\ge 1$ and $R\in \norm{F^\times}$. 
(This normalization differs from the usual one by a factor of $\frac{1}{2}$ for $F=\Bbb R$, $\frac{1}{\pi}$ for $F=\Bbb C$, and $\frac{\log q}{1-q^{-1}}$ for a non-archimedian local field with residue field $\Bbb F_q$). Then for a top degree differential form $\omega$ on an open set in $F^n$ or, more generally, an analytic $F$-manifold $Y$, we can define the corresponding measure (or density) $\norm{\omega}$, see \cite{We}.\footnote{Indeed, recall that the bundle of densities (or measures) on $Y$ is the $\Bbb R$-bundle whose sections transform under local coordinate changes $\lbrace x_i\rbrace\mapsto \lbrace x_i'\rbrace$ 
according to the rule of change of variable in the integral: 
$$
\mu'=\norm{\det\left(\tfrac{\partial x_i}{\partial x_j'}\right)}\mu.
$$ 
On the other hand, a top differential form on $M$ changes according to the rule 
 $\omega'=\det(\tfrac{\partial x_i}{\partial x_j'})\omega$. Thus any top form $\omega$ defines a measure $\norm{\omega}$.}
 This agrees with the notation 
$\norm{dx}$ for the Haar measure on $F$. 

Given an analytic $F$-line bundle $L$ on an analytic $F$-manifold $Y$, we will denote by 
$\norm{L}$ the associated complex line bundle (using the character $\norm{\cdot}: F^\times \to \Bbb R_{>0}\subset \Bbb C^\times$). For example, $\norm{K_Y}$ is the bundle of densities on $Y$. Note that since the structure group of $\norm{L}$ is contained in $\Bbb R_{>0}$, it makes sense to consider the complex power  $\norm{L}^s$ of this line bundle for any $s\in \Bbb C$. 

Denote by $L^2(Y)$ the space of measurable half-densities on $Y$ (sections of $\norm{K_Y}^{\frac{1}{2}}$) such that $\int_Y f\overline{f}<\infty$, modulo ones that vanish outside of a set of measure zero. This is a Hilbert space with inner product 
$$
(f,g)=\int_Y f\overline g.
$$
Note that if $U\subset Y$ is a subset whose complement has measure zero 
then we have a natural isometry 
$L^2(U)\cong L^2(Y)$. Thus if $Y=\bold Y(F)$ for a smooth irreducible algebraic variety $\bold Y$ over $F$ and $\Gamma$ is the group of birational automorphisms of $\bold Y$ defined over $F$ then we have a natural unitary representation $\rho: \Gamma\to {\rm Aut}L^2(Y)$. 

\subsection{Moduli spaces of stable bundles} 
The Langlands correspondence over the field $\Bbb F_q(X)$ of rational functions on a curve $X$ over $\Bbb F_q$ is formulated in terms of complex-valued functions on $\Bbb F_q$-points of the moduli stack ${\rm Bun}={\rm Bun}_G(X)$, which includes isomorphism classes of {\bf all} $G$-bundles 
on $X$ (including unstable bundles with arbitrarily large automorphism groups); this is needed because the Hecke operators arising in this correspondence involve a summation over all Hecke modifications of a given bundle at a point 
$x\in X(\Bbb F_q)$, and each modification gives a nonzero contribution. Thus one cannot define the action of Hecke operators on the space of functions on the set of stable bundles, since a Hecke modification of a stable 
bundle could be unstable.

On the other hand, in our setting, when $\Bbb F_q$ is replaced by a local 
field $F$, summation is replaced by integration. Thus if the locus of stable bundles is open and dense (which happens under the conditions given in the next paragraph) then non-stable bundles constitute ``a set of measure zero". So we can restrict ourselves to the space of square integrable functions (or, more precisely, half-densities) defined only on stable bundles, which (at least in the case when $G=PGL_n$) form a smooth quasiprojective variety. This is convenient for doing harmonic analysis. 

{\it From now on let $X$ be a smooth projective curve with distinct marked points $t_0,...,t_{N-1}$ defined over a field $F$ of characteristic $\ne 2$, where $N\ge 1$ for genus $1$ and $N\ge 3$ for genus $0$, and let $G=PGL_2$} (this guarantees that the locus of stable bundles is open and dense). 
 \footnote{Our analysis can be generalized to the case when 
 $(t_0,...,t_{N-1})\in S^NX(F)$ but $t_i$ are not necessarily defined over $F$.} 
 
Recall that a $G$-bundle on $X$ is a $GL_2$-bundle (i.e., a rank $2$ vector bundle) up to tensoring with line bundles. Let ${\rm Bun}_G(X,t_0,...,t_{N-1})$ be the moduli stack of principal $G$-bundles on $X$ with parabolic structures at $t_0,...,t_{N-1}$ (i.e., elements $y_i$ in the fibers at $t_i$ of the associated $\Bbb 
P^1$-bundle, which we will call {\bf parabolic lines}). Such bundles are called {\bf quasiparabolic}.\footnote{Note that this is slightly different from the notion of a {\bf parabolic bundle} where all parabolic points are equipped with parabolic weights.}
 
 \begin{definition}(\cite{S,MS}) The {\bf parabolic slope} of a rank 2 vector bundle $E$ on $X$ with parabolic structures $y_0,...,y_{N-1}$ at $t_0,...,t_{N-1}$ is the number
$$
\mu(E):=\tfrac{1}{2}\deg(E)+\tfrac{N}{4}.
$$
If $L\subset E$ is a line subbundle of $E$ then the parabolic slope of $L$ is 
$$
\mu(L):=\deg(L)+\tfrac{N_L}{2},
$$
where $N_L$ is the number of those $i$ for which $L_{t_i}=y_i$. 
The quasiparabolic bundle $E$ is called {\bf stable}, respectively {\bf semistable},\footnote{More precisely, this is the notion of (semi)stability for weights $0,1/2$ (see \cite{S,MS}).} 
if for any line subbundle $L\subset E$, one has 
$\mu(L)<\mu(E)$, respectively $\mu(L)\le \mu(E)$.\footnote{If $N=0$ (no parabolic points) then these notions coincide with the usual notions of slope 
and (semi)stability for vector bundles. Also, these notions coincide with 
the usual notions of (semi)stability in GIT, for weights $0,\frac{1}{2}$ at parabolic points, see \cite{Mu}, Section 12, and also the original papers \cite{S,MS}.}  
\end{definition} 
 
 Let ${\rm Bun}_G^\circ(X,t_0,...,t_{N-1})\subset {\rm Bun}_G(X,t_0,...,t_{N-1})$ be the open substack of {\bf stable} quasiparabolic bundles (\cite{S,MS}). 
 Every stable quasiparabolic bundle has a trivial automorphism group, 
 so ${\rm Bun}_G^\circ(X,t_0,...,t_{N-1})$ can be viewed as a scheme, 
 and moreover it is known to be a smooth quasiprojective variety of dimension $3(g-1)+N$. 
 We will denote this variety by $Bun_G^\circ(X,t_0,...,t_{N-1})$ or shortly by $Bun^\circ$ when no confusion is possible.\footnote{Namely, as explained in \cite{S,MS}, 
 $Bun_G^\circ(X,t_0,...,t_{N-1})$ is a smooth open subset of the normal projective variety 
 $Bun_G(X,t_0,...,t_{N-1})$ -- the (coarse) moduli space of {\bf semistable bundles}. This space, however, may be singular and not contained in the stack ${\rm Bun}_G(X,t_0,...,t_{N-1})$, since different semistable bundles with the same associated graded under a Harder-Narasimhan filtration may correspond to the same point of $Bun_G(X,t_0,...,t_{N-1})$. Also a semistable bundle may have a nontrivial group of automorphisms.} 
In our case, when $G=PGL_2$,  the variety $Bun^\circ_G(X,t_0,...,t_{N-1})$ is the union of two connected components $Bun_G^\circ(X,t_0,...,t_{N-1})_0$ and $Bun_G^\circ(X,t_0,...,t_{N-1})_1$, the moduli spaces of bundles of even and odd degrees, respectively. Moreover, the varieties $Bun_G^\circ(X,t_0,...,t_{N-1})_0$ and $Bun_G^\circ(X,t_0,...,t_{N-1})_1$ are naturally identified with the moduli spaces of stable rank $2$ quasiparabolic vector bundles on $X$ of degree $0$ and $1$, respectively, modulo tensoring 
with line bundles of degree $0$. 

\begin{remark} If $N=3,g=0$ then $Bun_G^\circ(X,t_0,...,t_{N-1})_0$ and $ Bun_G^\circ(X,t_0,...,t_{N-1})_1$ each consist of one point. 
Therefore, if $g=0$, we will usually assume that $N\ge 4$. 
\end{remark} 

\subsection{Hecke modifications and the Hecke correspondence}\label{hmhc} Assume that 
$X(F)\ne \emptyset$. Let $HM_{x,s}(E)$ be the {\bf Hecke modification} of 
a rank $2$ vector bundle $E$ on $X$ at the point $x\in X(F)$
along the line $s\in E_x$. Namely, regular sections of $HM_{x,s}(E)$ are rational sections of $E$ with no poles outside of $x$ and a possible first order pole at $x$ with residue belonging to $s$ (more precisely, this residue is well defined only up to scaling, but it still makes sense to say that it belongs to $s$). So we have a short exact sequence 
of coherent sheaves
$$
0\to E\to HM_{x,s}(E)\to \delta_x^s\to 0
$$
where $\delta_x^s$ is the skyscraper sheaf supported at $x$ whose space of sections over 
an open set containing $x$ is the line $s\otimes T_xX$. This gives rise to 
a natural short exact sequence of vector spaces 
\begin{equation}\label{natuiso} 
0\to E_x/s\to HM_{x,s}(E)_x\to s\otimes T_xX\to 0,
\end{equation} 
where the second non-trivial map takes the residue of a section. 

Note that if $E$ is stable then $HM_{x,s}(E)$ need not be stable (nor even semistable) in general. However, it is easy to see that if $E$ is stable and $x$ is fixed then for generic $s$ the bundle $HM_{x,s}(E)$ is stable. Moreover, if $Z$ denotes the (smooth) variety of triples $(x,E,s)$ such that $x\in  X$, $E\in Bun_G^\circ(X,t_0,...,t_{N-1})$, $s\subset E_x$, and the bundle $HM_{x,s}(E)$ is stable, then the assignment 
$(x,E,s)\mapsto HM_{x,s}(E)$ defines a regular map $Z\to Bun_G^\circ(X,t_0,...,t_{N-1})$. The variety $Z$ is called the {\bf universal Hecke correspondence} for stable bundles.

The notion of a Hecke modification is also defined on $PGL_2$-bundles with parabolic structures. Namely, if $x\ne t_i$, then there is no change at 
$t_i$, and if $x=t_i$ then we define the fixed line $y_i'$ in $HM_{x,s}(E)_x$ to be $E_x/s$ (regardless of $y_i$). In particular, if $s=y_i$ then $y_i'=E_x/y_i$.\footnote{It is easy to show that $HM_{t_i,y}$ is the limit of $HM_{x,s}$ as $x\to t_i$, $s\to y$, as long as $y\ne y_i$. This is not so, however, if $y=y_i$, in which case this limit does not exist. This follows, for instance, from Proposition \ref{heckcor}(i) below.} 

This construction gives rise to the {\bf Hecke correspondence} $Z$ for stable bundles. Namely, let $$Z\subset Bun_G^\circ(X,t_0,...,t_{N-1})\times 
Bun_G^\circ(X,t_0,...,t_{N-1})\times (X\setminus \lbrace{t_0,...,t_{N-1}\rbrace})$$ 
be the set of triples $(E,F,x)$ such that $F$ is obtained from $E$ by a Hecke modification at $x$ along some $s\in E_x$, and 
$q_1,q_2: Z\to Bun_G^\circ(X,t_0,...,t_{N-1})$ and $q_3: Z\to X \setminus 
\lbrace{t_0,...,t_{N-1}\rbrace}$ be the natural projections. We define the {\bf Hecke correspondence} at $x\in X \setminus \lbrace{t_0,...,t_{N-1}\rbrace}$ by $Z_x:=q_3^{-1}(x)$. Note that the maps $q_1,q_2:Z_x\to Bun_G^\circ(X,t_0,...,t_{N-1})$ are $\Bbb P^1$-bundles restricted to a dense 
open subset of the total space, so $Z, Z_x$ are irreducible varieties of dimensions $3g+N-1$ and $3g+N-2$. 

The following lemma is easy. 

\begin{lemma} \label{comm} There are unique isomorphisms 
$$
HM_{x,E_x/s}\circ HM_{x,s}(E)\cong E
$$
and 
$$
 HM_{x,s}\circ HM_{x',s'}(E)\cong HM_{x',s'}\circ HM_{x,s}(E)
 $$
  if $x\ne x'$, restricting to the identity outside $x,x'$. 
\end{lemma} 

It follows from Lemma \ref{comm} that $Z$ is symmetric under the swap of the two copies of $Bun_G^\circ(X,t_0,...,t_{N-1})$. 

For any $i=1,...,N$ we denote by 
 $S_i: Bun_G^\circ(X,t_0,...,t_{N-1})\to Bun_G^\circ(X,t_0,...,t_{N-1})$ the (a priori) rational morphism 
given by $S_i:=HM_{t_i,y_i}$.  

\begin{proposition}\label{invo} 
(i) The map $S_i$ extends to 
an involutive automorphism of the variety $Bun_G^\circ(X,t_0,...,t_{N-1})$ that 
exchanges $Bun_G^\circ(X,t_0,...,t_{N-1})_0$ 
and $Bun_G^\circ(X,t_0,...,t_{N-1})_1$.\footnote{The involutions $S_i$ extend to 
$Bun_G(X,t_0,...,t_{N-1})$, since they preserve semistability and $S$-equivalence. 
Thus, if $N>0$ 
then $Bun_G(X,t_0,...,t_{N-1})_0\cong Bun_G(X,t_0,...,t_{N-1})_1$. 
 However, if $N=0$ (no parabolic points) then $Bun_G(X)_0$ may be non-isomorphic to $Bun_G(X)_1$. For example, for genus $2$, by a theorem of Narasimhan and Ramanan (\cite{NR}), $Bun_G(X)_0\cong \Bbb P^3$, while $Bun_G(X)_1$ 
is the intersection of two general quadrics in $\Bbb P^5$ (\cite{C}), so these 3-folds (both smooth in this case) have different middle cohomology: $b_3(Bun_G(X)_1)=4\ne b_3(Bun_G(X)_0)=0$.}

(ii) $S_iS_j=S_jS_i$. 
\end{proposition} 

Thus the automorphisms $S_i$ define an action of $(\Bbb Z/2)^N$ 
on $Bun^\circ_G(X,t_0,...,t_{N-1})$, which gives rise to an action 
of the subgroup $\Bbb V:=(\Bbb Z/2)^N_0\cong (\Bbb Z/2)^{N-1}$ (vectors with zero 
sum of coordinates) 
on each of the two components of $Bun^\circ_G(X,t_0,...,t_{N-1})$.\footnote{This action is faithful for $N\ge 5$, but for $N=4$ there is a kernel generated by $\prod_{i=0}^3S_i$.}  

\begin{proof} Let us first show that the morphisms $S_i$ are regular  (not just rational). 
For this it is sufficient to check that $S_i$ preserves stability. 

Let $E':=S_i(E)$. Then $\mu(E')=\mu(E)+\frac{1}{2}$, where $\mu$ denotes the parabolic slope. Let $L$ be a line subbundle of $E$. If $L_{t_i}\ne y_i$ then $L$ defines a subbundle $L'$ of the same degree in $E'$, but 
now $L_{t_i}'=y_i'$.  
On the other hand, if $L_{t_i}=y_i$ then the corresponding subbundle $L'\subset E'$ 
has degree $\deg(L)+1$ but $L'_{t_i}\ne y_i'$. So in both cases $\mu(L')=\mu(L)+\frac{1}{2}$.
So if $\mu(L)< \mu(E)$ then $\mu(L')<\mu(E')$.  

The equalities $S_i^2={\rm Id}$ and $S_iS_j=S_jS_i$ follow from Lemma 
\ref{comm}. 
\end{proof} 

\begin{remark}\label{invo1} If the number $N$ of parabolic points is odd then every semistable rank $2$ bundle is stable (since $2\mu(E)$ is not an integer while $2\mu(L)$ is an integer for every line subbundle $L\subset E$). Thus, in this case 
the varieties $Bun_G^\circ(X,t_0,...,t_{N-1})_0$ and $Bun_G^\circ(X,t_0,...,t_{N-1})_1$ are smooth projective varieties, isomorphic to each other via $S_i$. 
Likewise, if $N=0$ then every rank $2$ bundle of odd degree is stable, so $Bun_G^\circ(X,t_0,...,t_{N-1})_1$ is a smooth projective variety. 
\end{remark} 

\subsection{Higgs fields, Hitchin system, nilpotent cone, very stable bundles}\label{hit} 

In this subsection we recall basics about Hitchin systems for quasiparabolic bundles for $G=PGL_2$ (much of it is also recalled in \cite{EFK,EFK2}, but we repeat it here for reader's convenience).  

Let $E$ be a quasiparabolic $PGL_2$-bundle 
on a smooth irreducible projective curve $X$ of genus $g$ with parabolic points $t_0,...,t_{N-1}$. 

\begin{definition} A {\bf Higgs field} for $E$ is 
an element $\phi\in H^0(X,{\rm ad}(E)\otimes K_X\otimes \bigotimes_{i=0}^{N-1} {O}(t_i))$ such that for all $i=0,...,N-1$, the residue of $\phi$ at $t_i$ is nilpotent and preserves the parabolic structure at $t_i$ (i.e., acts by zero on the corresponding line). A {\bf quasiparabolic Higgs bundle}
is a pair $(E,\phi)$ of a quasiparabolic bundle and a Higgs field for this bundle. 
\end{definition} 

Thus, if $E$ is stable then a Higgs field for $E$ is just a cotangent vector $\phi$ at $E$ to the moduli space $Bun^\circ$ of stable quasiparabolic bundles on $X$ with parabolic points $t_0,...,t_{N-1}$, and $T^*Bun^\circ$ is the variety of quasiparabolic Higgs bundles $(E,\phi)$ such that $E$ is stable. 

Let $\mathcal{B}:=H^0(X,K_X^{\otimes 2}\otimes \bigotimes_{i=0}^{N-1} {O}(t_i))$. 
This is a vector space of the same dimension $d=3g-3+N$ as $Bun^\circ$, and it is called {\bf the Hitchin base}. The {\bf Hitchin map} $\det: T^*Bun^\circ\to \mathcal{B}$ 
is defined by the formula $(E,\phi)\mapsto \det \phi$ (note that since the residue of $\phi$ is nilpotent at $t_i$, $\det\phi$ has at most first order poles). It is well known that the Hitchin map is flat and generically a Lagrangian fibration, so it defines an algebraic integrable system called the {\bf Hitchin system} (\cite{Hi,BD}). 
Namely, a choice of linear coordinates on $\mathcal{B}$ defines 
a collection of algebraically independent Poisson commuting regular functions $H_1,...,H_d$ on $T^*Bun^\circ$ (quadratic on fibers), which are called the {\bf Hitchin hamiltonians}.   

The {\bf nilpotent cone} $\mathcal N\subset T^*Bun^\circ$ is the zero-fiber of the Hitchin map, i.e., the subvariety of $(E,\phi)$ such that $\phi$ is nilpotent (that is, 
$\phi^2=0$). Since the Hitchin map is flat and generically a Lagrangian fibration, $\mathcal{N}$ is a Lagrangian 
subvariety of $T^*Bun^\circ$.

\begin{remark} The variety $T^*Bun^\circ$ is an open subset in the {\bf Hitchin moduli space} $\mathcal{M}_H$ of Higgs pairs $(E,\phi)$ which are stable in the sense of geometric invariant theory. If $E$ is stable then so is $(E,\phi)$ for any $\phi$, but $(E,\phi)$ may be stable for unstable $E$, so the inclusion $T^*Bun^\circ\hookrightarrow \mathcal{M}_H$ is strict. The Hitchin system and nilpotent cone are actually defined for the whole Hitchin moduli space, but here we restrict ourselves only to the open subset $T^*Bun^\circ \subset \mathcal{M}_H$. 
\end{remark} 

Beilinson and Drinfeld (\cite{BD}) quantized the Hitchin integrable system and defined the {\bf quantum Hitchin system} $(\widehat H_1,...,\widehat H_d)$
where $\widehat H_i$ are the {\bf quantum 
Hitchin hamiltonians} -- commuting 
twisted differential operators on $Bun^\circ$ whose symbols are $H_i$. 
The twisting here is by half-forms on $Bun^\circ$. 

We may therefore consider the system of differential equations
$$
\widehat H_i\psi=\mu_i\psi
$$
with respect to a (multivalued) half-form $\psi$ on $Bun^\circ$, 
where $\mu_i$ are scalars. 
It follows that this system defines a holonomic twisted $D$-module $\Delta_\mu$ on $Bun^\circ$ whose singular support is contained 
in the nilpotent cone $\mathcal{N}$. 
These $D$-modules were studied in \cite{BD} and play a key role in the geometric Langlands correspondence, as well as in our previous works \cite{EFK,EFK2} and this paper.  

Let $D$ be the projection of $\mathcal{N}$ to $Bun^\circ$. This is a divisor in $Bun^\circ$ which is called the {\bf wobbly divisor} (\cite{DP}). The complement of $D$ is thus the locus of bundles $E$ such that every nilpotent Higgs field for $E$ is zero. 
Such bundles are called {\bf very stable}. So we will denote the locus of such bundles by $Bun^{\rm vs}=Bun^\circ\setminus D\subset Bun^\circ$. 
 
Since the singular support of the quantum Hitchin $D$-module $\Delta_\mu$ is contained in $\mathcal{N}$, this $D$-module is $O$-coherent (i.e., smooth) on $Bun^{\rm vs}$. Thus it defines a de Rham local system (vector bundle with a flat connection) $\mathcal V_\mu$ on $Bun^{\rm vs}$. Moreover, the rank of this local system equals the degree of the Hitchin map restricted to the generic fiber of the cotangent bundle of ${ Bun^\circ}$, i.e., the product of degrees of $H_i$ restricted to this fiber, which is $2^d=2^{3g-3+N}$. 

The same definitions and results apply verbatim to $GL_2$-bundles (or, equivalently, rank $2$ vector bundles). 

\begin{lemma}\label{ves} Let $E$ be a very stable quasiparabolic rank $2$ vector bundle on a curve $X$ of genus $g$ with $N$ parabolic points, and $L\subset E$ a line subbundle. Then 
$\mu(E)-\mu(L)\ge \frac{N}{4}+\frac{g-1}{2}$. 
\end{lemma} 

\begin{proof} Suppose $L$ contains $k$ of the parabolic lines. 
Let $M=E/L$. Then 
$$
2(\mu(E)-\mu(L))=\deg E+\frac{N}{2}-2\deg L-k=
\deg M-\deg L-k+\frac{N}{2}.
$$
Let $\sigma\in H^0(L\otimes M^*\otimes K_X\otimes \bigotimes_{i=0}^{N-1} O(t_i))$ be such that $\sigma$ acts by zero on the parabolic line $s_i$ for all $i\in [0,N-1]$. 
The last condition  is vacuous if $s_i=L_{t_i}$, so the dimension of the space $V$ of such elements $\sigma$ satisfies the inequality 
$$
\dim V\ge H^0(L\otimes M^*\otimes K_X\otimes \bigotimes_{i=0}^{N-1} O(t_i))-N+k\ge \deg L-\deg M+k+g-1,
$$
where in the last inequality we used the Riemann-Roch theorem. But every $\sigma\in V$ defines a nilpotent Higgs field $\phi_\sigma$, so for a very stable $E$ 
we must have $V=0$. Thus 
$$
\deg M-\deg L-k-g+1\ge 0, 
$$ 
which yields
$$
2(\mu(E)-\mu(L))\ge \frac{N}{2}+g-1, 
$$
as claimed.
\end{proof} 

\begin{corollary}\label{ves0}
A very stable bundle $E$ remains stable after $< \frac{N-2+2g}{2}$ Hecke modifications at non-parabolic points, and remains semistable under 
$\le \frac{N-2+2g}{2}$ such modifications. 
\end{corollary} 

\begin{proof} A Hecke modification increases in $\mu(E)$ by $\frac{1}{2}$ 
and either increases $\mu(L)$ by $1$ or keeps it unchanged. 
Thus $\mu(E)-\mu(L)$ either increases or decreases by $\frac{1}{2}$. 
Hence the statement follows from Lemma \ref{ves}. 
\end{proof} 

\begin{corollary}\label{ves1} Let $H_x$ be the Hecke operator on $L^2(Bun^\circ(F))$ at a point $x\in X(F)$ defined in \cite{EFK2}. Let $\psi$ be a smooth 
compactly supported half-density on $Bun^\circ(F)$ with support in $Bun^{\rm vs}(F)$. 
Then for any positive integer $r< \frac{N-2+2g}{2}$ and 
any non-parabolic points $x_1,...,x_r\in X(F)$, the half-density 
$H_{x_1}...H_{x_r}\psi$ is smooth and compactly supported on $Bun^\circ(F)$. 
\end{corollary} 

\begin{proof} Corollary \ref{ves0} implies that the integral defining $H_{x_1}...H_{x_r}\psi$ is over a compact set and has no singularities, which implies the statement. 
\end{proof} 

\subsection{The wobbly divisor in genus zero} \label{sing2}
In the case of $X=\Bbb P^1$ with $N$ parabolic points the wobbly divisor can be described explicitly. Namely, for generic $E\in Bun_0^\circ$, a line subbundle $O(1-r)\subset E$ contains at most 
$2r-1$ parabolic lines, and it turns out that $D$ is exactly the locus on which 
this condition is violated for some $r$. 

In more detail, for a subset $S\subset [0,N-1]$ of even cardinality $2r$, 
let $D_S\subset Bun_0^\circ$ be the locus of bundles $E$ which contain 
a subbundle $O(1-r)$ containing the parabolic lines $y_i$ at $t_i$, $i\in S$. 

Also let $\Bbb S_N$ be the symmetric group 
and recall that $\Bbb V$ denotes its $N-1$-dimensional reflection representation over $\Bbb 
F_2$, i.e, the space of functions $f: [0,N-1]\to \Bbb F_2$ with sum of all values zero. Let us identify 
$\Bbb V$ with the set of subsets of $[0,N-1]$ of even cardinality by mapping $f\in \Bbb V$ to its support. 
Consider the Weyl group $\Bbb W:=W(D_N):=\Bbb S_N\ltimes \Bbb V$. This group acts naturally on the 
set $W(D_N)/W(A_{N-1})=\Bbb S_N\ltimes \Bbb V/\Bbb S_N=\Bbb V$, where $\Bbb V$ acts on itself by translations and $\Bbb S_N$ acts on $\Bbb V$ by permutations.

The group $\Bbb W$ also acts on the set of components of $D$, 
with $\Bbb S_N$ acting as the geometric Galois group permuting the parabolic points 
and $\Bbb V$ acting by the maps $S_i$, and it is easy to see that for all $g\in \Bbb W$ we have 
$g(D_S)=D_{g(S)}$.  

\begin{proposition}\label{dpan}(\cite{DP2}) (i) For each $S\subset [0,N-1]$, $D_S$ is an irreducible divisor in $Bun_0^\circ$, and these are distinct for $N\ge 5$.  

(ii) The components of the wobbly divisor $D$ are exactly the $D_S$. Thus for $N\ge 5$, 
$D$ has $2^{N-1}$ irreducible components permuted transitively by the group $\Bbb W$. 
\end{proposition} 

\begin{proof} (i) If $S=\lbrace i,j\rbrace$ then $D_S$ is the locus 
of bundles isomorphic to $O\oplus O$ which contain a trivial line subbundle $O$ 
containing $y_i$ and $y_j$. It follows that $D_S$ is just the closure of the locus 
where $y_i=y_j$ and there are no other equalities between 
$y_k$. Thus for $|S|=2$, $D_S$ is an irreducible divisor in $Bun^\circ_0$. Since $D_S$ are transitively 
permuted by $\Bbb V$, this holds for any $S$. Moreover, it follows from the formulas for $S_i$ (see Proposition \ref{Si} below) that for $N\ge 5$ the stabilizer of $D_S$ in $\Bbb V$ is trivial, i.e., all $D_S$ are distinct. 

(ii) Let $E\in D_S$, $|S|=2r$. So we have a line subbundle $O(1-r)\subset E$ containing 
the parabolic lines $y_i$ at $t_i, i\in S$. Thus we have a short exact sequence 
$$
0\to O(1-r)\to E\to O(r-1)\to 0.
$$
Recall that $K=O(-2)$. So the line bundle $O(r-1)^*\otimes O(1-r)\otimes K\otimes \bigotimes_{i\in S} O(t_i)$ has degree zero, hence is trivial. Thus there exists a unique up to scaling nonzero 
$$
\phi \in  H^0(\Bbb P^1,O(r-1)^*\otimes O(1-r)\otimes K\otimes \bigotimes_{i\in S} O(t_i))\subset 
H^0(\Bbb P^1,{\rm ad}(E)\otimes K\otimes \bigotimes_{i\in S} O(t_i)).
$$
Moreover, since $O(1-r)$ contains $y_i$, $\phi$ is 
a nilpotent Higgs field for $E$. Thus $D_S\subset D$. 

Conversely, let $E$ be not very stable and $\phi$ be a nonzero nilpotent Higgs field for $E$. 
We have $E=O(k)\oplus O(-k)$ for some $k\ge 0$, so the adjoint $\mathfrak{sl}_2$-bundle has the triangular decomposition ${\rm ad}(E)=O(k)\oplus O\oplus O(-k)$. Thus if $\phi$ is regular 
then $k\ge 1$, so $E\in D_\emptyset$. So assume that $\phi$ is not regular and has minimal possible number of poles occurring at the points $t_i, i\in S$ for some $\emptyset \ne S\subset [0,N-1]$. 
Then $\phi$ is nonvanishing (otherwise the number of poles of $\phi$ can be reduced by renomalizing it by a rational function). Thus $L:={\rm Ker}\phi$ is a line subbundle of $E$ containing the parabolic lines $y_i$ at $t_i,i\in S$. So we have a short exact sequence 
$$
0\to L\to E\to E/L\to 0
$$
and 
$$
\phi\in H^0(\Bbb P^1,(E/L)^*\otimes L\otimes K\otimes \bigotimes_{i\in S} O(t_i)).
$$
Thus the degree of the bundle $(E/L)^*\otimes L\otimes K\otimes \bigotimes_{i\in S} O(t_i)$, which equals 
$|S|-2\deg L-2$, is nonnegative, so $\deg(L)\le \frac{|S|}{2}-1$. Moreover, if 
this inequality is strict then we can reduce the number of poles of $\phi$ by renormalizing it by a 
section of $O(1)$, a contradiction. Thus $|S|$ is an even number $2r$ and $L\cong O(1-r)$, i.e., $E\in D_S$, as claimed. 
\end{proof} 

Let us give an explicit description of the components $D_S$ for $S\ne \emptyset$. In this case the bundle $E$ is trivial for a Zariski dense open subset of $D_S$: $E=O\oplus O$. Note that an inclusion $\iota: O(1-r)\hookrightarrow O\oplus O$ is defined by a rational function $f(z)=p(z)/q(z)$ of degree $r$ (i.e., $p,q$ are polynomials of degree $\le 
r$ without common roots, and at least one of them has degree exactly $r$). So the condition cutting out $D_S$ is that there exists such $f$ with 
$$
f(t_i)=y_i,\ i\in S, 
$$
i.e., 
$$
p(t_i)-y_iq(t_i)=0,\ i\in S.  
$$
As explained in Subsection \ref{interpo}, this condition can be written as the vanishing condition of 
the determinant of a $2r$ by $2r$ matrix:  
$$
\det(t_i^j,y_it_i^j,i\in S,0\le j\le r-1)=0.
$$

\section{Hecke operators in genus zero}

In this section we assume that $X=\Bbb P^1$ with $N=m+2$ parabolic points. Also until Subsection \ref{A formula for the product of Hecke operators}  we assume that $t_i\in X(F)$. In such a case we may assume that $t_0=0$ and $t_{m+1}=\infty$.

\subsection{Birational parametrizations of $Bun_G(\Bbb P^1,t_0,...,t_{m+1})_0$ and $Bun_G(\Bbb P^1,t_0,...,t_{m+1})_1$}\label{birpar} 
We start with a construction of  birational parametrizations of $Bun_G(\Bbb P^1,t_0,...,t_{m+1})_0$ and $Bun_G(\Bbb P^1,t_0,...,t_{m+1})_1$.

Since a generic quasiparabolic bundle $E\in Bun_G(\Bbb P^1,t_0,...,t_{m+1})_0$ 
is isomorphic to $O\oplus O$ as an ordinary vector bundle, it is determined by an $m+2$-tuple of one-dimensional subspaces (lines) in $\Bbb A^2$. Moreover, since ${\rm Aut}(O\oplus O)=GL_2$, we may assume that these lines are defined by vectors $(1,0)$, $(1,y_1)$,...,$(1,y_m)$, $(0,1)$, where $y_i$ are uniquely determined up to simultaneous scaling.  Let $E_{\bold y,0}$ be the bundle corresponding to $\bold y=(y_1,...,y_m)$. The assignment $\bold y\mapsto E_{\bold y,0}$ gives rise to a rational parametrization of $Bun_G(\Bbb P^1,t_0,...,t_{m+1})_0$ by $\Bbb P^{m-1}$.

A generic quasiparabolic bundle of degree $1$ is isomorphic to $O\oplus O(1)$ as an ordinary bundle. We realize $O(1)$ by gluing the charts $U_0=\Bbb 
P^1\setminus \infty$ and $U_\infty=\Bbb P^1\setminus 0$ using the gluing map $g(w)=w$. Thus a rational section of such bundle is given by a pair 
of rational functions $f_0$ on $U_0$ and $f_\infty$ on $U_\infty$ 
such that $f_0(w)=wf_\infty(w)$. So the vector bundle $O\oplus O(1)$ 
can realized similarly using the equation 
\begin{equation}\label{vecbun}
f_0(w)=\begin{pmatrix} 1 & 0\\ 0& w\end{pmatrix}f_\infty(w),
\end{equation} 
where $f_0,f_\infty$ are now {\bf pairs} of rational functions (written as column vectors). 
Namely, a rational section of this bundle is given by a pair 
of $\Bbb A^2$-valued rational functions $f_0$ on $U_0$ and $f_\infty$ on $U_\infty$ satisfying 
\eqref{vecbun}. We will realize lines in fibers of this bundle in the chart $U_0$ except at $\infty$ where we will use the chart $U_\infty$. 
 
To define a point on $Bun_G(\Bbb P^1,t_0,...,t_{m+1})_1$, we have to fix $m+2$ lines in the fiber of $O\oplus O(1)$ 
at $0,t_1,...,t_m,\infty$. Recall that the group ${\rm Aut}(O\oplus O(1))$ 
consists of maps of the form $(u,v)\to \lambda(u,(aw+b)u+cv)$. Thus
at $0,\infty$ we can fix standard lines spanned by the vectors $(1,0)$, $(1,0)$ (the latter in the chart $U_\infty$), and the remaining lines spanned by $(1,z_i)$, $i=1,...,m$, where $z_i$ are again uniquely determined up to simultaneous scaling. Let $E_{\bold z,1}$ be the bundle corresponding to $\bold z=(z_1,...,z_m)$. The assignment $\bold z\mapsto E_{\bold z,1}$ gives rise to a rational parametrization of $Bun_G(\Bbb P^1,t_0,...,t_{m+1})_1$
by $\Bbb P^{m-1}$.

We thus obtain 

\begin{lemma}\label{param} The assignments $\bold y\mapsto E_{\bold y,0}$, 
$\bold z\mapsto E_{\bold z,1}$ define birational isomorphisms 
$\Bbb P^{m-1}\to Bun_G(\Bbb P^1,t_0,...,t_{m+1})_0$ and $\Bbb P^{m-1}\to Bun_G(\Bbb P^1,t_0,...,t_{m+1})_1$. 
\end{lemma} 

Thus we have a rational map 
\begin{equation}\label{param1} 
\pi: Bun_G^\circ(\Bbb P^1,t_0,...,t_{m+1})_i\to  \Bbb P^{m-1}.
\end{equation} 
inverse to the map defined in Lemma \ref{param}. 
This map is, in fact, regular
(see e.g. \cite{C}).

\subsection{The Hecke correspondence for $X=\Bbb P^1$ with $m+2$ parabolic points}\label{heckecorrr}  
 In this subsection we
 provide an explicit description of the Hecke correspondence on bundles over  $X=\Bbb P^1$ with $m+2$ parabolic points in terms of the rational parametrization of Lemma \ref{param}.

By definition, the {\bf Hecke correspondence} for stable bundles at $x\in 
X$ is the variety $Z_x$ of pairs $(E,s)$ where $E\in Bun_G^\circ(\Bbb P^1,t_0,...,t_{m+1}),s\in \Bbb PE_x$ is such that $HM_{x,s}(E)$ is stable; i.e., it is the fiber over $x\in X$ of the natural map $Z\to X$, where $Z$ is the universal Hecke correspondence for stable bundles defined in Subsection \ref{hmhc}. This variety is 
equipped with maps $q_i: Z_x\to Bun_G^\circ(\Bbb P^1,t_0,...,t_{m+1})$, $i=1,2$, where $q_1(E,s)=E$ and 
$q_2(E,s)=HM_{x,s}(E)$. 

As before, set $t_{m+1}=y_{m+1}=z_{m+1}=\infty$, $t_0=y_0=z_0=0$. 
 
\begin{proposition}\label{heckcor} (i) $HM_{x,s}(E_{\bold y,0})\cong E_{\bold z,1}$, where 
\begin{equation}\label{heckecor} 
z_i(\bold t,x,\bold y,s)=\frac{t_is-xy_i}{s-y_i}.
\end{equation}
In particular, $S_{m+1}(E_{\bold y,0})\cong E_{\bold y,1}$. 

(ii) $HM_{x,s}(E_{\bold y,1})\cong E_{\bold z,0}$.\footnote{Note however that $HM_{x',s'}(HM_{x,s}(E_{\bold y,1}))\ncong HM_{x',s'}(E_{\bold z,0})$, since $HM_{x,s}$ does not preserve $s'$, but maps it to $\frac{x's-xs'}{s-s'}$. 
So we have $HM_{x',s'}(HM_{x,s}(E_{\bold y,1}))\cong HM_{x',\frac{x's-xs'}{s-s'}}(E_{\bold z,0})$.}
\end{proposition} 

\begin{proof} (i)
Since $E_{\bold y,0}$ is trivial as an ordinary bundle, we have an identification $\Bbb PE_x\cong \Bbb P^1$ for any point $x \in X=\Bbb P^1 $. Let $x$ be a point distinct from any of the parabolic points. Assume that the line  $s\in \Bbb P^1$ is spanned by the vector $(1,s)$.\footnote{Here 
and below we abuse notation by using the same letter $s$ to denote a point of $\Bbb A^1\subset \Bbb P^1$ and the corresponding line in the fiber of our vector bundle at $x$.} 
By definition, regular sections of $HM_{x,s}(E_{\bold y,0})$ (over some open set) are then pairs of functions $(g,h)$ regular except possible first order poles at $x$, such that $h-sg$ is regular at $x$. This bundle 
is isomorphic to $O\oplus O(1)$ as an ordinary bundle, 
but to compute $\bold z$ we need to identify it with the standard 
realization of this bundle and see what happens at the parabolic points. 
This is achieved by using the change of variable 
$$
(g,h)\mapsto (h-sg,(w-x)g).
$$ 

So, consider what happens to the lines in the fibers at $0,t_1,...,t_m,\infty$ under this change of variable. 
At $0$ we had the vector $(1,0)$, so after the change we get 
$(-s,-x)$. At $t_i$ we had $(1,y_i)$, so after the change we will get 
$(y_i-s,t_i-x)$. At $\infty$ we had $(0,1)$, so 
$w-x$ drops out and we get $(1,0)$. 

Now we need to bring this $m+2$-tuple of lines to the standard form. 
As noted above, automorphisms of $O\oplus O(1)$ 
have the form $(u,v)\to \lambda(u,(aw+b)u+cv)$, so in terms of $\zeta:=v/u$ we
have $\zeta\mapsto   aw+b+c\zeta$. Thus we have $a=0$ and $
b+cxs^{-1}=0$, which gives 
$b=x,\ c=-s$
up to scaling.  
So, we obtain 
$$
z_i(\bold t,x,\bold y,s)=x+s\frac{t_i-x}{s-y_i}=\frac{t_is-xy_i}{s-y_i},
$$
as claimed. This proves the first statement of (i). The second statement then follows 
by taking the limit $x\to \infty$ and then $s\to \infty$ (note that the order of limits 
is important here!).  

(ii) follows from (i) and the fact that $HM_{x,s}$ commutes with $S_{m+1}$. 
\end{proof} 

It will be convenient to have two more variants of the formula for the Hecke modification. 
First, we can quotient out the dilation symmetry, i.e., impose the condition $t_m=y_m=z_m=1$. Then the formula for the Hecke modification looks like 
\begin{equation}\label{HMlesssym}
z_i=\frac{(t_is-xy_i)(s-1)}{(s-x)(s-y_i)},\ 1\le i\le m-1. 
\end{equation} 
On the other hand, instead of breaking the dilation symmetry, we may keep 
it and moreover restore the translation symmetry, no longer requiring that $t_0=y_0=z_0=0$. 
Then the formula for the Hecke modification takes the form 
\begin{equation}\label{HMmoresym}
z_i=\frac{t_i-x}{s-y_i},\ 0\le i\le m. 
\end{equation} 

From now on we identify $Bun_G(\Bbb P^1,t_0,...,t_{m+1})_0$ with $Bun_G(\Bbb P^1,t_0,...,t_{m+1})_1$ using the map $S_{m+1}$. By Proposition \ref{heckcor}, in the coordinates $y_i,z_i$ this will just be the identity map. 

Now we express the maps $S_i$ in terms of the parametrizations of $Bun_G(\Bbb P^1,t_0,...,t_{m+1})_0$, and $Bun_G(\Bbb P^1,t_0,...,t_{m+1})_1$ given by Lemma \ref{param} (assuming that $t_0=y_0=z_0=0$). 

\begin{proposition}\label{Si} We have 
$$
S_0(y_1,...,y_m)=\left(\frac{t_1}{y_1},...,\frac{t_m}{y_m}\right),
$$
and for $1\le i\le m$ 
$$
S_i(y_1,...,y_m)=(z_1,...,z_m),
$$
where 
$$
z_j=\frac{y_jt_i-y_it_j}{y_j-y_i},\ j\ne i,\quad z_i=t_i.
$$
\end{proposition} 

\begin{proof}
The proposition follows by taking a limit in 
formula \eqref{heckecorrr}, first $x\to t_i$ and then $s\to y_i$. 
\end{proof} 

\subsection{Hecke operators}\label{Hecke operators}

We now pass from algebraic geometry to analysis. Consider the set $Bun_G^\circ(\Bbb P^1,t_0,...,t_{m+1})_i(F)$ of $F$-points
of the variety $Bun_G^\circ(\Bbb P^1,t_0,...,t_{m+1})_i$, $i=0,1$, an analytic 
$F$-manifold. To simplify notation, we will denote this analytic manifold 
by 
$Bun_i^\circ(F)$. 

Consider the Hilbert spaces $\mathcal H^0:=L^2(Bun_0^\circ(F))$, $\mathcal H^1= L^2(Bun_1^\circ(F))$, the spaces of square integrable half-densities. The birational 
isomorphism $S_{m+1}$ provides an identification $\mathcal H^0 \cong \mathcal H^1$, so for brevity we will denote this space by $\mathcal H$. 

The map $\pi$ given by \eqref{param1} defines an identification $\mathcal H\cong L^2(\Bbb P^{m-1}(F))$ of $\mathcal H$ with the space of square integrable half-densities on $\Bbb P^{m-1}(F)$. As well known, the bundle of half-densities on  $\Bbb P^{m-1}(F)$ is $\norm{K}^{\frac{1}{2}}$ where $K=O(-m)$ is the canonical bundle of $\Bbb P^{m-1}$. Thus we may (and will) realize half-densities on $Bun_i^\circ(F)$ as homogeneous complex-valued functions $\psi$ of $(y_1,...,y_m)\in F^m\setminus 0$ of homogeneity degree $-\frac{m}{2}$, i.e., functions $\psi$ such that 
$$
\psi(\lambda \bold y)=\norm{\lambda}^{-\frac{m}{2}}\psi(\bold y),\ \lambda\in F.
$$

\begin{definition} Define $U\subset F^m\setminus 0$ to be the open set of such points $\bold y=(y_1,...,y_m)$ that for any $y\in F$ the equality $y_i=y$ holds for fewer than $\frac{m}{2}+1$ values of $i$ for $y\ne 0$ 
and fewer than $\frac{m}{2}$ values of $i$ for $y=0$.
\end{definition} 
  
\begin{definition} Let $V\subset \mathcal{H}$ be the space of continuous complex-valued functions $\psi$ on $F^m\setminus 0$ of homogeneity degree $-\frac{m}{2}$ (a 
dense subspace in $\mathcal{H}$). Also let $\widetilde{V}\supset V$ be the space of functions $\psi$ of homogeneity degree $-\frac{m}{2}$ defined and continuous on $U$.\footnote{Note that $\widetilde{V}$ is {\bf not} contained in $\mathcal{H}$.}
\end{definition} 

Recall that in \cite{EFK2}, Subsection 1.2 we defined the Hecke operator $H_x$ depending on a point $x\in X(F)$, which is given by the convolution with the ``$\delta$-function" of the Hecke correspondence $Z_x$. Recall also that $H_x$ is not a function of $x$ but rather a $-\frac{1}{2}$-density, i.e., a section of $\norm{K_X}^{-\frac{1}{2}}$. For the purposes of computation, however, 
we will treat $H_x$ as a function, so that the actual Hecke operator is 
$H_x\norm{dx}^{-\frac{1}{2}}$. Then an explicit formula for the Hecke operators 
in the case of genus $0$ is given by the following theorem (in which we set $t_0=0$).\footnote{Our space $V$ (the initial domain for $H_x$) here 
is slightly different than that in \cite{EFK2}. This is a very minor modification that has no effect on the final result.}

\begin{theorem}\label{heckeop1}
The Hecke operator $H_x$ is the operator $V\to \widetilde V$ given by the 
formula
$$
H_x=\norm{\prod_{i=0}^m(t_i-x)}^{\frac{1}{2}}\Bbb H_x, 
$$
where 
\begin{equation}\label{Hxfor}
(\Bbb H_x\psi)(y_1,...,y_m):=\int_{F}\psi\left(\frac{t_1s-xy_1}{s-y_1},...,\frac{t_ms-xy_m}{s-y_m}\right)\frac{\norm{s}^{\frac{m-2}{2}}\norm{ds}}{\prod_{i=1}^m \norm{s-y_i}}.
\end{equation} 
\end{theorem} 

We will call $\Bbb H_x$ the {\bf modified Hecke operator}.  

We  prove Theorem \ref{heckeop1} (in particular, showing that this integral converges) in the next subsection. 
For now, let us record its two equivalent formulations 
corresponding to breaking the dilation symmetry and to restoring the translation symmetry. 

For both formulations, as before, we set $t_{m+1}=y_{m+1}=\infty$.
For the first variant we assume that $t_0=y_0=0,t_m=y_m=1$ (this can be achieved by shift and rescaling). We now realize $\mathcal H$ as the space $L^2(F^{m-1})$ of $L^2$-functions in $m-1$ variables  $w_1,...,w_{m-1}$ (without a homogeneity condition) with the norm given by the formula 
$$
\norm{\phi}^2=\int_{F^{m-1}}|\phi(w_1,...,w_{m-1})|^2\norm{dw_1...dw_{m-1}},
$$
via 
$$
\phi(w_1,...,w_{m-1})=\psi(w_1,...,w_{m-1},1).
$$  
\begin{proposition}\label{heckeop3} In terms of $\phi$, the modified Hecke operator takes the form
\begin{gather}\label{t24}
(\Bbb H_x\phi)(u_1,...,u_{m-1})=\\ 
\int_{F}\phi\left(\frac{(t_1s-u_1x)(s-1)}{(s-u_1)(s-x)},...,
\frac{(t_{m-1}s-u_{m-1}x)(s-1)}{(s-u_{m-1})(s-x)}\right)
\frac{\norm{s(s-1)}^{\frac{m}{2}-1}\norm{ds}}{\norm{s-x}^{\frac{m}{2}}\prod_{i=1}^{m-1}\norm{s-u_i}}. \nonumber
\end{gather}
\end{proposition} 

\begin{proof} This is equivalent to Theorem \ref{heckeop1} using formula (\ref{HMlesssym}) and the homogeneity of $\psi$. 
\end{proof} 

Define the operator $U_{s,x}$ on $L^2(F^{m-1})$ for $x\ne t_i$ by 
$$
(U_{s,x}\phi)(u_1,....,u_{m-1}):=
$$
$$
\norm{\prod_{i=1}^{m-1}\frac{s(s-1)(t_i-x)}{(s-x)(s-u_i)^2}}^{\frac{1}{2}}\phi\left(\frac{(t_1s-u_1x)(s-1)}{(s-u_1)(s-x)},...,
\frac{(t_{m-1}s-u_{m-1}x)(s-1)}{(s-u_{m-1})(s-x)}\right). 
$$
It is easy to check that $U_{s,x}$ is a unitary operator on $\mathcal{H}$. Namely, let $\bold G:=PGL_2^{m-1}$, and $W$ be the unitary representation of $PGL_2(F)$ of principal series, on half-densities on $\Bbb P^1(F)$ (cf. Subsection \ref{prinse}). Then $U_{s,x}$ is the operator defined 
 in $W^{\otimes m-1}$ by the element 
\begin{equation}\label{gsx}
g_{s,x}:=(g_{s,x,1},...,g_{s,x,m-1})\in \bold G(F),\ g_{s,x,i}(u):=\frac{(t_is-ux)(s-1)}{(s-u)(s-x)}.
\end{equation}

We thus obtain the following equivalent form of Proposition \ref{heckeop3} (and hence Theorem \ref{heckeop1}).

\begin{proposition}\label{heckeop4}  
\begin{equation}\label{intunit}
H_x=\int_{F}U_{s,x}\sqrt{\norm{\frac{x(x-1)}{s(s-1)(s-x)}}}\norm{ds}.
\end{equation}
\end{proposition} 

For the second variant, we will think of $\psi$ as a function of $y_0,...,y_m$ which is semi-invariant under the group $y\mapsto ay+b$ of degree $-\frac{m}{2}$, no longer assuming that $t_0=y_0=0$. Then we have 

\begin{proposition}\label{heckeop2} For the curve $X=\Bbb P^1$ with $m+2$ marked points $t_0,...,t_m,\infty$ the modified Hecke operator is given by the formula
$$
(\Bbb H_x\psi)(y_0,...,y_m)=\int_{F}\psi\left(\frac{t_0-x}{s-y_0},...,\frac{t_m-x}{s-y_m}\right)\frac{\norm{ds}}{\prod_{i=0}^m \norm{s-y_i}}. 
$$
\end{proposition} 

\begin{proof} This is equivalent to Theorem \ref{heckeop1} using formula \eqref{HMmoresym} and the translation invariance of $\psi$. 
\end{proof} 

\subsection{Proof of Theorem \ref{heckeop1}} 

We start with showing that $H_x$ is precisely the Hecke operator defined in \cite{EFK2}, Subsection 1.2. We will use the formulation from Proposition \ref{heckeop4}. 
By Proposition \ref{heckcor}, the Hecke operator defined in \cite{EFK2} 
has the form 
\begin{equation}\label{intunit.1}
(H_x\psi)(\bold y)=\int_{F}(U_{s,x}\psi)(\bold y)d\mu(s),
\end{equation}
where $d\mu(s)$ is a certain measure on $\Bbb P^1(F)$, and our job is to compute $d\mu(s)$ (we will see that it is, in fact, independent on $\bold y$, which has to do with the fact that the bundle $E_{\bold y}$ is trivialized for all $\bold y$).  

To any point $(\bold y,\bold z,s)\in Z_x$ we can attach two 1-dimensional 
spaces $T_{1,s},T_{2,s}$, where $T_{1,s}$ is the tangent space at $(\bold 
y,\bold z,s)$ to $q_1^{-1}(\bold y)$ and $T_{2,s}$ is the tangent space at $(\bold y,\bold z,s)$ to $q_2^{-1}(\bold z)$. These spaces define line bundles $T_1,T_2$ on the 
projective line $\Bbb P^1$ with coordinate $s$ when $\bold y$ is fixed. 
The bundle $T_1$ is just the anticanonical bundle $K_{\Bbb P^1}^{-1}$, i.e., $T_{1,s}=T_s\Bbb P^1=s^{*\otimes 2}$, 
while the bundle $T_2$, in view of \eqref{natuiso}, is naturally isomorphic to the canonical bundle $K_{\Bbb P^1}$, i.e., 
$T_{2,s}\cong T^*_s\Bbb P^1=s^{\otimes 2}$. Thus we have a canonical isomorphism $\eta: T_1\to T_2\otimes K_{\Bbb P^1}^{-2}$, and one can see that it is given by the formula $\eta=(a^*)^{\otimes 2}$ where $a$ is the isomorphism of \cite{EFK2}, Theorem 1.1 (for fixed $x$). This isomorphism has the form $\eta(v)=\eta_*(v)\gamma(s)(ds)^{-2}$, where $v\in T_{1,s}$, $\eta_*: T_{1}\to T_{2}$ is inverse to the map induced by the projection $(\bold y,\bold z,s)\mapsto s$, and $\gamma(s)(ds)^{-2}$ is a regular section of $K_{\Bbb P^1}^{-2}$ serving to cancel the poles of $\eta_*$, to make 
sure that $\eta$ is a well defined isomorphism. Moreover, we have $d\mu(s)=\norm{\gamma(s)}^{-\frac{1}{2}}\norm{ds}$. 
Thus the explicit form of the map $\eta_*$ completely determines 
$\gamma(s)$ and hence $d\mu(s)$, at least up to a scalar (depending on $x$).  

To compute $\eta_*$, note that the fiber of $q_2$ near $(\bold y,\bold z,s)$ 
is the parametrized curve $(\bold y(u),\bold z, u)$ with $\bold y(s)=\bold y$, where 
the corresponding Hecke modification $\bold z(u)=\bold z$ remains constant. 
Thus for $v\in T_{1,s}$, 
$$
\eta_*(v)=(\bold y'(s),0,1)ds(v). 
$$
Setting $t_i=t,y_i=y,z_i=z$ for brevity, we have $dz(s)=0$, which 
yields
$$
y'(s)=-\frac{\partial_s z}{\partial_{y} z}.
$$
By formula (\ref{HMlesssym}) we have 
$$
z=\frac{(ts-xy)(s-1)}{(s-y)(s-x)}.
$$
Thus 
$$
z^{-1}\partial_{y}z=\frac{(t-x)s}{(s-y)(ts-xy)}, 
$$
$$
z^{-1}\partial_s z=\frac{1}{s-xyt^{-1}}+\frac{1}{s-1}-\frac{1}{s-x}-\frac{1}{s-y}.
$$
So 
$$
dy=-\frac{(s-y)(ts-xy)}{(t-x)s}\left(\frac{1}{s-xyt^{-1}}+\frac{1}{s-1}-\frac{1}{s-x}-\frac{1}{u-y}\right)ds.
$$
This 1-form has first order poles at $s=0,1,x,\infty$ and no other singularities. 
Thus the same holds for $\eta_*$.
It follows that $\gamma(s)(ds)^{-2}$ has simple zeros at $0,1,x,\infty$, i.e.,  
$$
\gamma(s)^{-1}(ds)^2=C(x)\frac{(ds)^2}{s(s-1)(s-x)}.
$$

It remains to show that $C(x)=x(x-1)$. This can be shown by a slightly more careful analysis, taking into account the variation of $x$. One can also see that $C(x)$ is proportional to $x(x-1)$ by looking at the asymptotics $x\to 0,1,\infty$. 
Thus, we get 
$$
d\mu(s)=\sqrt{\norm{\frac{x(x-1)}{s(s-1)(s-x)}}}\norm{ds}, 
$$
hence the operator of Theorem \ref{heckeop1} coincides with the Hecke operator defined in \cite{EFK2}, Subsection 1.2. 

Now we show that \eqref{Hxfor} defines a linear operator $V\to \widetilde{V}$. 
Suppose $\psi$ is continuous, and that less than $\frac{m}{2}+1$ points 
$y_i$ coincide. Let us show that the integral defining $\Bbb H_x$ converges (uniformly on compact sets in $U$). 
At $s=\infty$ the density in \eqref{Hxfor} behaves as $\norm{s^{-2}ds}$, so we only need to check convergence 
near $s=y_i$, say, for $i=1$. Using the homogeneity of $\psi$, we can 
rewrite \eqref{Hxfor} as follows: 
$$
(\Bbb H_x\psi)(y_1,...,y_m)=
$$
$$
=\int_{F}\psi\left(t_1s-xy_1,\frac{(s-y_1)(t_2s-xy_2)}{s-y_2},...,\frac{(s-y_1)(t_ms-xy_m)}{s-y_m}\right)\frac{\norm{s-y_1}^{\frac{m}{2}}\norm{s}^{\frac{m-2}{2}}\norm{ds}}{\prod_{i=1}^m \norm{s-y_i}}.
$$
By the definition of $U$, this density behaves near $s=y_1$ as $\norm{s-y_1}^{\frac{m}{2}-k}\norm{ds}$, where $k< \frac{m}{2}+1$ if $y_1\ne 0$ and as $\norm{s-y_1}^{\frac{m-2}{2}-k}\norm{ds}$, where $k< \frac{m}{2}$ if $y_1=0$. This density is integrable, so the integral in \eqref{Hxfor} converges. Finally, observe that $\Bbb H_x\psi$ is homogeneous of degree $-\frac{m}{2}$. This implies that $\Bbb H_x: V\to \widetilde{V}$. 

\subsection{Boundedness of Hecke operators}\label{Boundedness of Hecke operators}

We now show that Hecke operators extend to bounded operators on $\mathcal 
H$. 
\begin{proposition}\label{bou} We have $H_x(V)\subset \mathcal{H}$, and $H_x$ extends to a bounded operator on $\mathcal{H}$ which depends continuously on $x$ when $x\ne t_i,\infty$. Moreover,
$$
\norm{H_x}=O\left(\norm{x-t_i}^{\frac{1}{2}}\log\tfrac{1}{\norm{x-t_i}}\right), \ x\to t_i;\ \norm{H_x}=O(\norm{x}^{\frac{1}{2}}\log\norm{x}),
\ x\to \infty. 
$$
\end{proposition} 

\begin{proof} 
By Proposition \ref{heckeop4}, 
\begin{equation}\label{normbound}
\norm{H_x}\le \int_{F}\sqrt{\norm{\frac{x(x-1)}{s(s-1)(s-x)}}}\norm{ds}.
\end{equation}
This integral is convergent (indeed, it is the elliptic integral considered in Subsection \ref{ellint}), so, using Lemma \ref{l2ell} and the symmetry 
between $t_i$ and $0,1,\infty$, we get 
the result. 
\end{proof}

We will see in Proposition \ref{asym} below that the bound of Proposition \ref{bou} is in fact sharp. 

We also have

\begin{proposition}\label{selfad} The operators $H_x$ are self-adjoint and pairwise commuting.  
\end{proposition} 

\begin{proof} By Proposition \ref{bou}, the operators $H_x$ are bounded. Also, it is easy to see that 
\begin{equation}\label{sigmax}
U_{s,x}^\dagger=U_{\sigma_x(s),x},\ \sigma_x(s):=\frac{x(s-1)}{s-x}
\end{equation}
and the measure of integration in \eqref{intunit} is invariant under the involution $s\mapsto \sigma_x(s)$.
This implies the first statement. The second statement follows from Lemma 
\ref{comm}. 
\end{proof} 

\begin{example}\label{ellint1} Let $m=1$. In this case the sets $Bun^\circ_0(F)$
and $Bun^\circ_1(F)$ consist of one point, so the space $\mathcal{H}$ 
is 1-dimensional. Thus the operator $U_{s,x}$ is the identity on this 1-dimensional space. So we see that the operator $\Bbb H_x$ is just a scalar 
function of $x$ given by the formula
$$
\Bbb H_x=E(x),
$$
where $E(x)$ is the elliptic integral defined in Subsection \ref{ellint}. 

\end{example} 

\subsection{Compactness of Hecke operators}\label{Compactness of Hecke operators}

Let $t_{m+1}=\infty$, $t_0,...,t_m\in \Bbb C$.

\begin{proposition}\label{compa} The Hecke operator $H_x$ is compact and continuous in the operator norm in $x$ when $x\ne t_i,\infty$.  
\end{proposition} 

\begin{proof} We first show that $H_x$ is compact. Let $n=3m-3$. 
We will approximate the operator $H_x^n$ in the operator norm 
by trace class (hence compact) operators. Since the space of compact operators 
is closed with respect to the operator norm, this will imply that $H_x^n$ 
is compact, hence $H_x$ is also compact (as it is self-adjoint). 

We have 
$$
H_x^n=\int_{F^n}U_{s_1,x},...,U_{s_n,x}d\nu_x(s_1)...d\nu_x(s_n),
$$
where $\nu_x$ is the measure of integration in \eqref{intunit}, i.e., 
\begin{equation}\label{nuxs}
d\nu_x(s)=\sqrt{\norm{\frac{x(x-1)}{s(s-1)(s-x)}}}\norm{ds}.
\end{equation} 
Recall that $\bold G=PGL_2^{m-1}$. Let $\xi_n: \Bbb A^n\to \bold G$ be the rational map given by 
$$
\xi_n(s_1,...,s_n):=g_{s_1,x}...g_{s_n,x},
$$ 
where $g_{s,x}\in \bold G$ is defined by \eqref{gsx} (except that here we 
work with algebraic varieties and not yet with their points over $F$). 
Let $\lambda:=\xi_{n*}(\nu_x^{\boxtimes n})$ be the direct image of the 
measure $\nu_x^{\boxtimes n}$ under $\xi_n$. 
Since $n=\dim \bold G$, $\xi_n$ is dominant by Lemma \ref{gt}. 
Hence 
$$
d\lambda=f_x(g)dg,
$$
where $dg$ is the Haar measure on $\bold G(F)$ and $f_x$ is an $L^1$-function on $\bold G(F)$. Indeed, this follows from the fact that the preimage of a set of measure zero under
$\xi_{n}$ has measure zero. 

Recall that $W$ denotes the principal series representation of $PGL_2(F)$ 
on half-densities on $\Bbb P^1(F)$, and 
$W^{\otimes m-1}$ is the corresponding unitary representation of $\bold G(F)$. Let $\rho$ be the corresponding representation map. We have 
$$
H_x^n=\int_{\bold G(F)} \rho(g)f_x(g)dg.
$$

By Harish-Chandra's Theorem (Theorem \ref{hcthm}), for any smooth compactly supported function $\phi$ on $\bold G(F)$, the operator 
$$
A_\phi:=\int_{\bold G(F)} \rho(g)\phi(g)dg
$$
is trace class, therefore compact. But since the function $f_x$ is $L^1$, 
it can be approximated in $L^1$-norm by a
smooth compactly supported function $\phi$ with any precision $\varepsilon>0$. This implies that the operator $H_x^n$ is compact, hence so is $H_x$. 

It remains to show that $H_x$ is norm-continuous in $x$. To this end, fix $x$ and $\varepsilon>0$. Note that 
$\nu_y$ depends continuously on $y$ in the $L^1$ metric, hence so does $\nu_y^{\boxtimes n}$. Thus $f_y=\xi_{n*}(\nu_y^{\otimes n})/dg$ depends continuously on $y$ in the $L^1$ metric.
Therefore, there is $\delta>0$ such that if $\norm{y-x}<\delta$ then $\norm{f_y-f_x}_{L^1}<\varepsilon$. Then $\norm{H_y-H_x}<\varepsilon$, as claimed. 
\end{proof} 

\subsection{The leading eigenvalue} By the spectral theorem for compact self-adjoint operators, 
the commuting operators $H_x$, being compact, have a common orthogonal eigenbasis. Moreover, we have the following proposition. 

\begin{proposition}\label{posi}
The largest eigenvalue $\beta_0(x)$ of $H_x$ is positive and has multiplicity $1$, with a unique positive normalized eigenfunction $\psi_0$ (independent on $x$), and $\norm{H_x}=\beta_0(x)$. 
\end{proposition} 

\begin{proof} This follows by the Krein-Rutman theorem (an infinite-dimensional analog of the Frobenius-Perron theorem, see \cite{KR}), since the Schwartz kernel of $H_x^n$ is a strictly positive function for large enough $n$.
\end{proof} 

\subsection{Asymptotics of Hecke operators as $x\to t_i$ and $x\to \infty$}\label{Asymptotics of Hecke operators}

When $x\to t_i$ and $x\to \infty$, Hecke operators have singularities. So 
we would like to compute the leading coefficient of the asymptotics. 

\begin{proposition}\label{asym} (i) In the sense of strong convergence, we have\footnote{This implies that the ``true" Hecke operator $H_x\norm{dx}^{-\frac{1}{2}}$ has the strong asymptotics $\norm{\frac{dz}{z}}^{-\frac{1}{2}}\log\norm{z^{-1}}S_i$ near the parabolic point $t_i$, where $z$ is the local coordinate near the parabolic point. Note that this statement is independent on the choice of the local coordinate $z$.}  
$$
H_x\sim \norm{x-t_i}^{\frac{1}{2}}\log(\norm{x-t_i}^{-1})S_i, \ x\to t_i; 
$$
in other words, for any $\psi\in \mathcal H$ we have 
$$
\frac{H_x\psi}{\norm{x-t_i}^{\frac{1}{2}}\log(\norm{x-t_i}^{-1})}\to S_i\psi
$$ 
as $x\to t_i$. Similarly, in the sense of strong convergence
$$
H_x\sim \norm{x}^{\frac{1}{2}}\log\norm{x},\ x\to \infty.\footnote{Note that the asymptotics formula near $t_i\in F$ appears to be different than at $x=\infty$: the power of the local coordinate is $\frac{1}{2}$ in the first case and $-\frac{1}{2}$ in the second case. This is in full agreement with projective invariance, however, since $H_x$ is actually a $-\frac{1}{2}$-density with respect to $x$, but we are treating it as a function by multiplying by $dx^{\frac{1}{2}}$ and thus breaking the symmetry between the parabolic points.}
$$

(ii) We have 
$$
\norm{H_x}\sim \norm{x-t_i}^{\frac{1}{2}}\log(\norm{x-t_i}^{-1}), \ x\to t_i;\ \norm{H_x}\sim \norm{x}^{\frac{1}{2}}\log\norm{x},
\ x\to \infty.
$$
\end{proposition} 

\begin{proof} (i) We establish the asymptotics at $x\to \infty$; the case 
$x\to t_i$ then follows by symmetry. By Proposition \ref{bou}, $\frac{\norm{H_x}}{\norm{x}^{\frac{1}{2}}\log\norm{x}}$ is bounded when $x\to \infty$. Therefore, it suffices to show 
that 
\begin{equation}\label{limi}
\lim_{x\to \infty}\frac{H_x\psi}{\norm{x}^{\frac{1}{2}}\log\norm{x}}=\psi
\end{equation} 
for $\psi$ belonging to a dense subspace of $\mathcal H$. In particular, it is enough to prove this in the case when $\psi$ is continuous. 

We have 
$$
(U_{s,\infty}\phi)(u_1,....,u_{m-1})=
$$
$$
\norm{\prod_{i=1}^{m-1}\frac{s(s-1)}{(s-u_i)^2}}^{\frac{1}{2}}\phi\left(\frac{u_1(s-1)}{s-u_1},...,
\frac{u_{m-1}(s-1)}{s-u_{m-1}}\right). 
$$
Also by Lemma \ref{l2ell} the density 
$$
\frac{1}{\norm{x}^{\frac{1}{2}}\log\norm{x}}\sqrt{\norm{\frac{x(x-1)}{s(s-1)(s-x)}}}\norm{ds}=
\norm{1-\frac{1}{x}}^{\frac{1}{2}}\frac{E_x}{\norm{x}^{-\frac{1}{2}}\log\norm{x}}
$$
converges as $x\to \infty$ to $\delta_\infty$.  
This implies \eqref{limi} for continuous $\psi$ 
by formula \eqref{intunit}, since 
$$
\lim_{s\to \infty}\frac{u(s-1)}{s-u}=u.
$$ 

(ii) By Proposition \ref{posi}, 
$\norm{H_x}=(H_x\psi_0,\psi_0)$, 
so the statement follows from (i).  
\end{proof} 

\begin{remark} It follows from Proposition \ref{compa} 
that the convergence in Proposition \ref{asym} is only strong and not in the norm, since a sequence of compact operators cannot converge in the norm to an invertible operator.  
\end{remark} 

\subsection{The spectral theorem} \label{The spectral theorem}
Here is one of our main results. 

\begin{theorem}\label{specthe} (i) There is an orthogonal decomposition 
$$
\mathcal H=\oplus_{k=0}^\infty \mathcal H_k,
$$ 
where $\mathcal H_k$ are the eigenspaces of $H_x$: 
$$
H_x\psi=\beta_k(x)\psi,\ \psi\in \mathcal H_k, 
$$
where $\beta_k$ are continuous real functions of $x$ defined for $x\ne t_0,...,t_m,\infty$. 

(ii) 
$$
\beta_k(x)\sim \norm{x}^{\frac{1}{2}}\log\norm{x},
\ x\to \infty.
$$
In particular, the function $\beta_k$ is not identically zero for any $k$. Thus the spaces $\mathcal H_k$ are finite dimensional.

(iii) There is a leading positive eigenvalue $\beta_0(x)$ of $H_x$ such that 
$$
|\beta_i(x)|<\beta_0(x)
$$ 
for all $i>0$, and the corresponding eigenspace $\mathcal H_0$ is 1-dimensional, spanned by a positive eigenfunction $\psi_0(\bold y)$. Moreover, $\norm{H_x}=\beta_0(x)$. 
\end{theorem} 

\begin{proof} (i) This follows from Propositions \ref{selfad} and \ref{compa} and the 
spectral theorem for compact self-adjoint operators. 

(ii) This follows from (i) and Proposition \ref{asym}. 

(iii) Follows from Proposition \ref{posi}. 
\end{proof} 

Theorem \ref{specthe} implies Conjecture 1.2 of \cite{EFK2} 
for $G=PGL_2$ and curves of genus zero.  

\begin{corollary} On every compact subset $C\subset F\setminus \lbrace{t_0,...,t_m\rbrace}$,  
the sequence $\beta_k(x)$ is equicontinuous and converges uniformly to $0$ as $k\to \infty$. 
\end{corollary} 

\begin{proof} The function $x\mapsto H_x$ is continuous on $C$ by Proposition \ref{compa}, therefore by Cantor's theorem it is uniformly continuous. So for any $\varepsilon>0$ there is $\delta>0$ 
such that if $|y-x|<\delta$ then $\norm{H_y-H_x}<\varepsilon$. Then 
$|\beta_k(y)-\beta_k(x)|<\varepsilon$ for all $i$, which proves equicontinuity. 

Suppose $\beta_k$ does not go uniformly to $0$ on $C$. Then 
there is a sequence $k_j$ such that ${\rm sup}_C|\beta_{k_j}(x)|\ge \varepsilon$ 
for some $\varepsilon>0$. The sequence $\beta_k$ is uniformly bounded (since so are the operators $H_x$ for $x\in C$), so by the Ascoli-Arzela theorem, the subsequence 
$\beta_{k_j}$ has a uniformly convergent subsequence, to some 
nonzero function $\beta(x)$. This is a contradiction since 
by Proposition \ref{compa}, for any $x$ we have $\lim_{k\to \infty}\beta_k(x)=0$. 
\end{proof} 

Now recall that we have an action of the group $\Bbb V=(\Bbb Z/2)^{m+1}$ on $\mathcal H$ 
by the operators $S_i$. Since by Proposition \ref{asym}, $S_i$ are the leading 
coefficients
of $H_x$ as $x$ approaches $t_i$, it follows that $S_i$ act by scalars ($\pm 1$) on each eigenspace 
$\mathcal H_k$. Let us denote this character of $\Bbb V$ by $\chi_k$; i.e., 
$S_iv=\chi_k(S_i)v$ for $v\in \mathcal{H}_k$. Note that $\chi_0=1$. We thus obtain 

\begin{corollary}\label{asymeig} One has 
$$
\beta_k(x)\sim \norm{x-t_i}^{\frac{1}{2}}\log(\norm{x-t_i}^{-1})\chi_k(S_i), \ x\to t_i.
$$ 
\end{corollary} 

The following corollary may be interpreted as the statement that the operator 
$H_x$ is smooth in $x$. 

\begin{corollary} The function $\beta_k(x)$ is smooth in $x$ for $x\ne t_i,\infty$ 
(meaning locally constant in the non-archimedian case). 
\end{corollary} 

\begin{proof} This follows from the fact that $\beta_k(x)=(H_x\psi,\psi)$ 
for an ($x$-independent) normalized eigenfunction $\psi\in \mathcal{H}_k$. The integral defining 
$(H_x\psi,\psi)$ is a smooth function of $x$. 
\end{proof} 

Let us also derive a formula for the  Schwartz kernel $K(x_1,...,x_r,\bold y,\bold z)$ 
of the product $H_{x_1}...H_{x_r}$ of Hecke operators 
(which may be a singular distribution) in terms of eigenfunctions, i.e. as a ``reproducing kernel". We may choose an orthonormal basis $\psi_{k,j}$ in each $\mathcal{H}_k$, so that all $\psi_{k,j}$ are real-valued; this 
is possible since the Schwartz kernel of a power of the Hecke operator is 
real-valued. Then we have 
\begin{equation}\label{reproker} 
K(x_1,...,x_r,\bold y,\bold z)=\sum_{k,j}\beta_k(x_1)...\beta_k(x_r)\psi_{k,j}(\bold y)\psi_{k,j}(\bold z). 
\end{equation} 

In particular, for $n\ge 1$ we have 
\begin{equation}\label{trac}
{\rm Tr}(H_{x_1}...H_{x_n})=\sum_k d_k\beta_k(x_1)...\beta_k(x_N), 
\end{equation} 
where $d_k:=\dim \mathcal H_k$, provided that this series is absolutely 
convergent. 

\subsection{The subleading term of the asymptotics of $H_x$ as $x\to \infty$.} 

\begin{proposition}\label{limexists} The operator $\widetilde{H}_x:=\norm{x}^{-\frac{1}{2}}H_x-\log\norm{x}$ has a strong limit $M=\widetilde{H}_\infty$ as $x\to \infty$, which is an unbounded self-adjoint operator on 
$\mathcal H$, acting diagonally in the basis of eigenfunctions $\psi_n(\bold y)$: 
$$
M \psi_k=\mu^{(k)}\psi_k.
$$
That is, for any $\psi$ in the domain of $M$, we have 
$$
\lim_{x\to \infty}\widetilde{H}_x\psi=M\psi.
$$
Here 
$$
\mu^{(k)}=\lim_{x\to \infty} (\norm{x}^{-\frac{1}{2}}\beta_k(x)-\log\norm{x}).
$$

Namely, the operator $M$ is given by the formula
\begin{gather*}\label{Hxfor1} 
(M\psi)(y_1,...,y_m):=\\
\int_F \left( \psi(y_1-t_1s,...,y_m-t_ms)+
\frac{\psi\left(\frac{y_1}{1-y_1s^{-1}},...,\frac{y_m}{1-y_ms^{-1}}\right)}{\prod_{i=1}^m \norm{1-y_is^{-1}}}
-\psi(y_1,...,y_m)\right)\frac{\norm{ds}}{\norm{s}}.
\end{gather*}
\end{proposition}

Note that integral (\ref{Hxfor1}) is convergent (say, for $\psi$ smooth) since the third summand cancels the logarithmic divergences at $s=0$ and $s=\infty$ generated by the first and the second summand, respectively.

\begin{proof} 
We first prove that $M$ is given by the claimed formula up to adding the operator of multiplication by a function $h(\bold y)$. For this it suffices to check that for a fixed generic $\bold y=(y_1,...,y_m)$, if 
$\psi$ is smooth with $\psi(\bold y)=0$ then 
\begin{equation}\label{twosum}
(M\psi)(\bold y)=\int_F \psi(y_1-t_1s,...,y_m-t_ms)\frac{\norm{ds}}{\norm{s}}+\int_F
\frac{\psi\left(\frac{y_1}{1-y_1s^{-1}},...,\frac{y_m}{1-y_ms^{-1}}\right)}{\prod_{i=1}^m \norm{1-y_is^{-1}}}\frac{\norm{ds}}{\norm{s}}.
\end{equation} 
To prove this formula, consider the closure $Z_{x,\bold y}$ in  $\Bbb P^{m-1}$ of the preimage of $\bold y$ in $Z_x$.
This is a parametrized rational curve of degree $m$ in $\Bbb P^{m-1}$ 
given by 
$$
z_i(s)=\frac{t_is-xy_i}{s-y_i}.
$$
Since $(H_x\psi)(\bold y)$ is defined by integration over $Z_{x,\bold y}(F)$, the function $(M\psi)(\bold y)$ is defined by integration over $Z_{\infty,\bold y}(F)$, where $Z_{\infty,\bold y}$ is the degeneration of the 
curve $Z_{x,\bold y}$ as 
$x\to \infty$. By definition, $\bold z\in Z_{\infty,y}$ if there exist $s,\lambda\in \Bbb C((x^{-1/n}))$ for some $n$ such that 
$$
\lim_{x\to \infty}\lambda(x)\frac{t_is(x)-xy_i}{s(x)-y_i}=z_i
$$
(the factor $\lambda(x)$ is needed since $z_i$ are defined only up to simultaneous scaling). 
If $s$ has a finite limit $s_0$ at $x=\infty$, we have 
$$
z_i=\frac{y_i}{1-y_is_0^{-1}}
$$
up to scaling (namely, if $s_0=y_i$ for some $i$, we get 
$z_i=1$ and $z_j=0$ for $j\ne i$, and if $s_0=0$ then $z_i=1$ for 
all $i$).  
On the other hand, if the order of $s(x)$ at
$\infty$ is $-r$ for a rational number $r>0$ then if $0<r<1$, we get $z_i=y_i$ for all $i$, for $r>1$ we get $z_i=t_i$ for all $i$, 
so the only interesting case is $r=1$, i.e., $s(x)=s_0x+o(x)$, $s_0\ne 0$. In this case we have 
$$
z_i=y_i-t_is_0
$$
up to scaling. This shows that $Z_{\infty,\bold y}$ is the union of two components:  
$Z_{\infty,\bold y}^1$ of degree $m-1$
defined by the parametric equations 
$$
z_i(s)=\frac{y_i}{1-y_is^{-1}},
$$
and $Z_{\infty,\bold y}^2$ of degree $1$ (a line) defined by the parametric equations
$$
z_i(s)=y_i-t_is
$$
(which are permuted by the birational involution $S_0$ of $\Bbb P^{m-1}$ 
transforming $y_i$ into $\frac{t_i}{y_i}$). 
Thus $(M\psi)(\bold y)$ is the sum of two integrals, over $Z_{x,\bold y}^1(F)$ and $Z_{x,\bold y}^2(F)$, which yields the desired formula. 

It remains to show that $h=0$. To this end,  
let us apply the operator $M$ 
to the function $\psi(y_1,...,y_m)=\prod_{i=1}^m \norm{y_i}^{-\frac{1}{2}}$. 
We have 
$$
(M\psi)(y_1,...,y_m)=I(y_1,...,y_m)\prod_{i=1}^m \norm{y_i}^{-\frac{1}{2}},
$$
where 
$$
I(y_1,...,y_m):=\lim_{x\to \infty}\left(\int_F \frac{1}{\norm{\prod_{i=1}^m (1-s^{-1}y_i)(1-st_iy_i^{-1}x^{-1})}^{\frac{1}{2}}}\frac{\norm{ds}}{\norm{s}}-\log\norm{x}\right).
$$
But direct asymptotic analysis shows that 
$$
I(y_1,...,y_m)=\int_F \left(\frac{1}{\prod_{i=1}^m \norm{1-s^{-1}y_i}^{\frac{1}{2}}}+
\frac{1}{\prod_{i=1}^m \norm{1-st_iy_i^{-1}}^{\frac{1}{2}}}-1\right)\frac{\norm{ds}}{\norm{s}},
$$
which implies that $h=0$. 
\end{proof} 

We note that the two summands in formula \eqref{twosum} are permuted by $S_0$, 
while $S_i$ for $1\le i\le m$ preserves each summand (in fact, this holds 
even before integration if we change $s$ to $s^{-1}$ in one of the summands). Thus we have 
$$
S_0M=MS_0=S_0Q+QS_0,
$$
where $Q$ is a self-adjoint operator such that 
$$
(Q\psi)(\bold y)=\int_F \psi(y_1-t_1s,...,y_m-t_ms)\frac{\norm{ds}}{\norm{s}}
$$
for $\psi$ smooth with $\psi(\bold y)=0$ for a given point $\bold y$. Moreover,  
$[Q,S_i]=0$ for $1\le i\le m$. 

For $m=2$ ($4$ parabolic points), the operator $M$ is studied in more detail 
in Subsection \ref{subleading}. In this case, it is easy to check that 
the two summands in formula \eqref{twosum} coincide with each other. 
This agrees with the fact that in this special case $\prod_{i=0}^mS_i=S_0S_1S_2=1$
(as we have mentioned, this is not so for $m>2$).

\subsection{Traces of powers of $|H_x|$} Let $m\ge 2$. 
Note that the series \eqref{trac} (say, with all $x_i$ equal) cannot be absolutely convergent if $n\le 2(m-2)$. Indeed, the  Schwartz kernel of the operator $H_x^{m-2}$ is supported on the $m-2$-th convolution power of the Hecke correspondence $Z_x$, which is a proper subvariety (since ${\rm 
codim}Z_x=m-2$), hence a set of measure zero. So this operator cannot be Hilbert-Schmidt and thus ${\rm Tr}(H_x^{2(m-2)})=\infty$. 

On the other hand, if $n$ is sufficiently large, the series \eqref{trac} does have to converge absolutely. To show this, let ${\bf n}(\varepsilon)={\bf n}(H_x,\varepsilon)$ be the number of eigenvalues of $H_x$ (counted with multiplicities) of magnitude $\ge \varepsilon$, and define  
$$
b_m:={\rm limsup}_{\varepsilon\to 0}\frac{\log {\bf n}(\varepsilon)}{\log (\varepsilon^{-1})}.
$$

\begin{lemma}\label{tracelem}
For any $a>b_m$,  
$$
{\rm Tr}|H_x|^a=\sum_{k\ge 0}d_k|\beta_k(x)|^a<\infty,
$$
and for any $a<b_m$ 
$$
{\rm Tr}|H_x|^a=\sum_{k\ge 0}d_k|\beta_k(x)|^a=\infty,
$$
where $d_k=\dim\mathcal H_k$. 
\end{lemma} 

\begin{proof} 
This follows from Lemma \ref{seq} applied to the sequence $|\beta_i(x)|$ (repeated $d_i$ times). 
\end{proof} 

\begin{proposition}\label{numeig} (i) We have  
$2(m-1)\le b_m<\infty$.  

(ii) For sufficiently large $n$ the series \eqref{trac} converges absolutely. 
\end{proposition}

\begin{proof} (i) We first show that $b_m<\infty$. 
Let $I_\varepsilon$ be the union of the $\varepsilon$-neighborhood of 
the points $0,x\in F$ and its image under the map $\sigma_x$ defined by \eqref{sigmax}. Consider the cutoff measure $d\nu_{x,\varepsilon}(s)$ on $F$ which equals $d\nu_x(s)$ (defined by \eqref{nuxs}) when $s\notin I_\varepsilon$ and zero otherwise. 
Set 
$$
H_{x,\varepsilon}=\int_{F}U_{s,x}d\nu_{x,\varepsilon}(s).
$$
This operator is self-adjoint since $\nu_{x,\varepsilon}$ is invariant under 
$\sigma_x$. Moreover, it is compact (the proof is the 
same as Proposition \ref{compa}). Finally, $\norm{\nu_{x,\varepsilon}-\nu_x}_{L^1}\le C\varepsilon^{\frac{1}{2}}$ for some $C>0$. Thus 
$$
\norm{H_{x,\varepsilon}-H_x}\le C\varepsilon^{\frac{1}{2}}.
$$
By Lemma \ref{compalem}(ii), this implies that 
\begin{equation}\label{doubling}
{\bf n}(H_x,2C\varepsilon^{\frac{1}{2}})\le {\bf n}(H_{x,\varepsilon},C\varepsilon^{\frac{1}{2}}).
\end{equation} 

\begin{lemma}\label{tra} There exist an even integer $\ell>0$ and a real $c>0$ such that 
${\rm Tr}H_{x,\varepsilon}^\ell=O(\varepsilon^{-c})$ as $\varepsilon\to 
0$. 
\end{lemma} 

\begin{proof} Similarly to the proof of Proposition \ref{compa}, one can show that 
$$
H_{x,\varepsilon}^{3(m-1)}=\int_{\bold G(F)}\rho(g)f_{x,\varepsilon}(g)dg,
$$
where $f_{x,\varepsilon}(g)\le f_x(g)$ is an $L^1$-function with compact support. 
Since $f_x(g)dg$ is the direct image 
of the norm of a rational function under a finite covering map $\xi_{3(m-1)}$, 
it follows that $f_x\in L^{\frac{r}{r-1}}$ for sufficiently large even integer $r$. 
Then using Young's inequality for convolutions (\cite{We1}, p.54-55), it follows by induction that the convolution power $f_x^{*k}$ belongs to $L^{\frac{r}{r-k}}$ when $k\le r$; 
in particular, $f_{x}^{*r}$ belongs to $L^\infty$. We then denote by 
$M_{r}$ the $L^\infty$-norm of this function, and by $S_\varepsilon$ 
its (compact) support. Clearly, 
$$
\norm{f_{x,\varepsilon}^{*r}}_{L^\infty}\le M_{r}.
$$ 

Let $\ell:=3(m-1)r$. Then 
$$
{\rm Tr}H_{x,\varepsilon}^\ell=\int_{\bold G(F)}\chi^{\otimes m-1}(g)f_{x,\varepsilon}^{*r}(g)dg,
$$
where $\chi=\chi_W$ is the character of the principal series representation $W$ (a locally $L^1$ function given by Proposition \ref{charsl2}). It follows from Proposition \ref{charsl2} that $\chi(g)\ge 0$ and for some 
$c>0$, 
$$
\int_{S_{\varepsilon}}\chi^{\otimes m-1}(g)dg=O(\varepsilon^{-c}).
$$
Thus  
$$
\int_{\bold G(F)}\chi^{\otimes m-1}(g)f_{x,\varepsilon}^{*r}(g)dg=\int_{S_\varepsilon}\chi^{\otimes m-1}(g)f_{x,\varepsilon}^{*r}(g)dg\le 
M_r\int _{S_\varepsilon}\chi^{\otimes m-1}(g)dg=O(\varepsilon^{-c}).
$$ 
This implies the statement. 
\end{proof}

Now, using Lemma \ref{tra} and \eqref{doubling}, we get 
$$
{\bf n}(H_x,2C\varepsilon^{\frac{1}{2}})\le {\bf n}(H_{x,\varepsilon},C\varepsilon^{\frac{1}{2}})\le (C\varepsilon^{\frac{1}{2}})^{-\ell}{\rm Tr}H_{x,\varepsilon}^\ell=O(\varepsilon^{-c-\ell/2}).
$$
This implies that 
$$
{\bf n}(H_x,\varepsilon)=O(\varepsilon^{-2c-\ell}), 
$$ 
i.e., $b_m\le 2c+\ell<\infty$. 

On the other hand, by Lemma \ref{Schatt} 
$b_m\ge 2(m-1)$, which completes the proof (i). 

(ii) follows from (i), Lemma \ref{tracelem} and the arithmetic and geometric mean inequality. 
\end{proof} 

Thus, the exponent of the sequence $|\beta_i(x)|$ repeated $d_i$ times (Subsection \ref{exponent}) is $\frac{1}{b_m}$. 

\begin{remark} It would be interesting to determine this exponent precisely. As we have shown, it is $\le \frac{1}{2(m-1)}$. Moreover, one can show that for $m=2$ 
the exponent equals $\frac{1}{2(m-1)}=\frac{1}{2}$. 
\end{remark} 

\subsection{A formula for the Hecke operator $H_{\bold x}$ for $\bold x\in S^nX(F)$}
\label{A formula for the product of Hecke operators} 

Let $x_1,...,x_n\in F$ and $\bold x=(x_1,...,x_n)$. Consider the operator
$$
H_{\bold x}:=H_{x_1}...H_{x_n}. 
$$
Since the factors commute, this product is invariant under permutations of $x_i$. 
However, if we write an explicit formula for $H_{\bold x}$ by iterating the definition of $H_x$, 
the resulting expression won't be manifestly $\Bbb S_n$-invariant. The goal of 
this subsection is 
to write another, manifestly symmetric formula for $H_{\bold x}$. More precisely, we will write a symmetric formula for 
$$
\Bbb H_{\bold x}:=\Bbb H_{x_1}...\Bbb H_{x_n}.
$$ 
This formula will then also make sense 
for $\bold x=(x_1,...,x_n)\in S^nX(F)$, where the individual coordinates $x_i$ are not 
required to be defined over $F$. 

Let $\bold y=(y_0,...,y_m)$
and $E_{\bold y,0}$, $E_{\bold y,1}$  be the vector bundle $O\oplus O$, respectively $O\oplus O(1)$,  with parabolic structures at $t_i$ given by the vectors $(1,y_i)$, $0\le i\le m$ and the vector $(0,1)$, respectively $(1,0)$ at $\infty$ (see Subsection \ref{birpar}).   
Let $E$ be the vector bundle obtained from $E_{\bold y,0}$ by simultaneous Hecke modification at points $x_1,...,x_n$ along the lines $s_1,...,s_n$. The sections of $E$ are pairs of rational functions 
$(g,h)$ with at most simple poles at $x_1,...,x_{n}$ and $h-s_ig$ regular 
at $x_i$ for $i=1,...,n$. 

We are now ready to write a formula for $\Bbb H_{\bold x}$. We will use Lemma \ref{interp}. 
First consider the case of even $n=2r$. In this case generically $E\cong O(r)\oplus O(r)$. Explicitly, this isomorphism is given by the map 
$$
(g,h)\mapsto (q_1h-p_1g, q_2h-p_2g), 
$$
where $p_1,q_1,p_2,q_2$ are polynomials of degrees $r,r-1,r-1,r$ with $p_1,q_2$ monic 
such that the rational functions 
$$
f_k(\bold x,\bold s,w)=f_k(w)=\frac{p_k(w)}{q_k(w)}
$$ 
satisfy the conditions 
$$
f_k(x_i)=s_i,\ i=1,...,n.
$$ 
(this uniquely determines the coefficients of $p_1,q_1,p_2,q_2$ from a system of linear equations). This means that 
$$
f_1=\frac{p_1}{q_1}=\iota_{\bold x,\infty}^{-1}(\bold s,\infty),\quad 
f_2=\frac{p_2}{q_2}=\iota_{\bold x,\infty}^{-1}(\bold s,0). 
$$
where $\iota$ is the map from Lemma \ref{interp}. 
Thus we get $E=E_{\bold z,0}\otimes O(r)$, where 
\begin{equation}\label{ziform} 
z_i=\frac{p_2(\bold x,\bold s,t_i)-q_2(\bold x,\bold s,t_i)y_i}{p_1(\bold x,\bold s,t_i)-q_1(\bold x,\bold s,t_i)y_i}.
\end{equation} 

Now consider the case of odd $n=2r+1$. Then generically  $E\cong O(r)\oplus O(r+1)$.
Explicitly this isomorphism is given as in the even case, except with $p_1,q_1,p_2,q_2$ being polynomials of degrees $r,r,r+1,r-1$ with $q_1,p_2$ monic, and 
$$
f_1=\frac{p_1}{q_1}=\iota_{\bold x}^{-1}(\bold s),\quad f_2=\frac{p_2}{q_2}=\iota_{\bold x,\infty}^{-1}(\bold s,\infty^2),
$$
where 
$$
\iota_{\bold x,\infty}^{-1}(\bold s,\infty^2):=\lim_{b\to \infty} \iota_{\bold x,b,\infty}^{-1}(\bold s,\infty,\infty)
$$
is the rational function of degree $r+1$ taking values $s_i$ at $x_i$ for 
$1\le i\le 2r+1$ 
and growing quadratically at $\infty$. So we get $E=E_{\bold z,1}\otimes O(r)$, with $z_i$ given by \eqref{ziform}. 

Alternatively, formula \eqref{ziform} for $n\ge 2$ can be written as 
$$
z_i=\frac{\sum_{j=1}^n(-1)^{j-1}R(\widehat{\bold x}_j,\widehat{\bold s}_j)\frac{s_j-y_i}{t_i-x_j}}{\sum_{j=1}^n(-1)^{j-1}T(\widehat{\bold x}_j,\widehat{\bold s}_j)\frac{s_j-y_i}{t_i-x_j}},
$$
where $\widehat{\bold x}_i,\widehat{\bold s}_i$ are obtained from $\bold x$, $\bold s$ 
by omitting the $i$-th coordinate, and $R,T$ 
are the cofactors arising from Lemma \ref{interp} when expanding the determinants 
in the numerator and denominator in the $0$-th column. To write explicit formulas for these cofactors, 
define the matrix $M_{k,l}(\bold x,\bold s)$ with $k+l=n$ by the formula 
$$
M_{k,l}(\bold x,\bold s)_{ij}=s_ix_i^{j-1},\ 1\le j\le k;\ M_{k,l}(\bold x,\bold s)_{ij}=x_i^{j-k-1},\ k+1\le j\le n.
$$
Then for $n=2r-1$ we have 
$$
T(\bold x,\bold s)=\det M_{r-1,r}(\bold x,\bold s),\ R(\bold x,\bold s)=\det M_{r,r-1}(\bold x,\bold s),
$$
while for $n=2r$ we have 
$$
T(\bold x,\bold s)=\det M_{r,r}(\bold x,\bold s),\ R(\bold x,\bold s)=\det M_{r+1,r-1}(\bold x,\bold s).
$$

Altogether we obtain the following proposition. 

\begin{proposition}\label{prodheck} We have 
$$
(\Bbb H_{\bold x}\psi)(y_0,...,y_m)=
$$
\scriptsize
$$
\int_{F^n}\psi\left(\frac{p_2(\bold x,\bold s,t_0)-q_2(\bold x,\bold s,t_0)y_0}{p_1(\bold x,\bold s,t_0)-q_1(\bold x,\bold s,t_1)y_0},...,\frac{p_2(\bold x,\bold s,t_m)-q_2(\bold x,\bold s,t_m)y_m}{p_1(\bold x,\bold s,t_m)-q_1(\bold x,\bold s,t_m)y_m}\right)\norm{\frac{\Delta(\bold x)T(\bold 
x,\bold s)^{-2}d\bold s}{\prod_{i=0}^m(p_1(\bold x,\bold s,t_i)-q_1(\bold x,\bold s,t_i)y_i)}},
$$ \normalsize
where $\Delta(\bold x):=\prod_{i<j}(x_i-x_j)$ is the Vandermonde determinant. 
\end{proposition} 

\begin{example}\label{symfor} {\bf 1.} $n=1$, so $r=0$. Then $\bold s=s$, $\bold x=x$, 
so\footnote{This case $n=1$ is somewhat degenerate since 
$q_2=0$, so $f_2$ is not well defined. 
Nevertheless, computation of the coefficients 
of $p_k,q_k$ using the corresponding system of linear equations 
produces the given answer.}  
$$
p_1(w)=s,\ q_1(w)=1,\ p_2(w)=w-x,\ q_2(w)=0,\ \Delta(\bold x)=T(\bold x,\bold s)=1.
$$
Thus we recover exactly formula \eqref{heckeop2} for $\Bbb H_x$. 

{\bf 2.} $n=2$, so $r=1$. Then $\bold s=(s_1,s_2)$, $\bold x=(x_1,x_2)$, and 
$$
T(x_j,s_j)=1,\ R(x_j,s_j)=s_j,\ T(\bold x,\bold s)=s_1-s_2,
$$
$$
p_1(w)=w-\frac{s_1x_2-s_2x_1}{s_1-s_2},\ q_1(w)=\frac{x_1-x_2}{s_1-s_2}.
$$

Thus we get 
\begin{gather}\label{2pts} 
(\Bbb H_{x_1,x_2}\psi)(y_0,...,y_m)=\nonumber \\
\int_{F^2}\psi\left(
\frac{
\frac{s_1-y_0}{t_0-x_1}s_2-
\frac{s_2-y_0}{t_0-x_2}s_1
}
{\frac{s_1-y_0}{t_0-x_1}-
\frac{s_2-y_0}{t_0-x_2}
},...,
\frac{
\frac{s_1-y_m}{t_m-x_1}s_2-
\frac{s_2-y_m}{t_m-x_2}s_1
}
{\frac{s_1-y_m}{t_m-x_1}-
\frac{s_2-y_m}{t_m-x_2}
}
\right)\norm{\frac{(x_1-x_2)(s_1-s_2)^{-2}ds_1ds_2}{\prod_{i=1}^m(t_i-\frac{s_1x_2-s_2x_1}{s_1-s_2}- \frac{x_1-x_2}{s_1-s_2}y_i)}}. 
\end{gather}
\end{example} 

\begin{remark} Formula \eqref{2pts} (as well as its analogs for $n>2$) may be checked by 
direct composition of 1-point Hecke operators. Namely, in the case $n=2$, if $s$ is the integration variable associated to $x_2$ and $s'$ to $x_1$ then the composition formula for $\Bbb H_{x_1}\Bbb H_{x_2}$ turns into 
the symmetric formula \eqref{2pts} by the change of variable $s=s_2,s'=\frac{x_2-x_1}{s_1-s_2}$. Thus the change of variable in the composition formula needed to exhibit the commutativity of $H_{x_1}$ and $H_{x_2}$ is 

$$
((s',x_1),(s,x_2))\mapsto ((s+\tfrac{x_2-x_1}{s'},x_1),(s',x_2)).
$$
\end{remark} 

Proposition \ref{prodheck} is useful for generalizing the theory of Hecke 
operators 
to the case when the point $x\in \overline{F}$ is not defined over $F$. 
In this case, we cannot define the Hecke operator $H_x$ attached to $x$ but can define the Hecke operator attached to the Galois orbit $\bold x:=\Gamma x=\lbrace{x_1,...,x_n\rbrace}$, where 
$\Gamma:={\rm Gal}(\overline{F}/{F})$; 
it is an analog of $H_{x_1}...H_{x_n}$ but the individual factors in this 
product are now not defined. Let $\Gamma_i\subset \Gamma$ be the stabilizers of $x_i$, and let $\Bbb K_i:=\overline{F}^{\Gamma_i}=F(x_i)$. Let 
$\Bbb K$ 
be the set of $\bold s=(s_1,...,s_n)\in \overline{F}^n$ such that $s_i\in \Bbb K_i$ 
and $gs_i=s_j$ whenever $gx_i=x_j$, $g\in \Gamma$. Note that $\Bbb K$ is a local field, and we have a natural identification 
$\Bbb K\cong \Bbb K_i$ sending $\bold s$ to $s_i$. 

Now we can define the modified Hecke operator $\Bbb H_{\bold x}: V\to \widetilde{V}$ attached to $\bold x$ by the same formula as in Proposition \ref{prodheck} but with integration over $\Bbb K$: 
$$
(\Bbb H_{\bold x}\psi)(y_0,...,y_m)=
$$ \scriptsize
$$
\int_{\Bbb K}\psi\left(\frac{p_2(\bold x,\bold s,t_0)-q_2(\bold x,\bold s,t_0)y_0}{p_1(\bold x,\bold s,t_0)-q_1(\bold x,\bold s,t_0)y_0},...,\frac{p_2(\bold x,\bold s,t_m)-q_2(\bold x,\bold s,t_m)y_m}{p_1(\bold x,\bold s,t_m)-q_1(\bold x,\bold s,t_m)y_m}\right)\norm{\frac{\Delta(\bold x)T(\bold x,\bold s)^{-2}d\bold s}{\prod_{i=0}^m(p_1(\bold x,\bold s,t_i)-q_1(\bold x,\bold s,t_i)y_i)}},
$$ \normalsize
and define the usual Hecke operator by the formula 
$$
H_{\bold x}=\norm{\prod_{i,j}(t_i-x_j)}^{\frac{1}{2}}\Bbb H_{\bold x}
$$
(note that $p_k(\bold x,\bold s,t_i),q_k(\bold x, \bold s,t_i)\in F$ because of symmetry, and that $\norm{\cdot}$ denotes the norm for $F$). 
In fact, using this formula, we can define the Hecke operator corresponding to any effective 
divisor $\bold x$ on $\Bbb P^1$ defined over $F$ and not containing the parabolic points; namely, we write $\bold x$ as the sum of Galois orbits and then take the product of the corresponding operators $H_\bold x$. 

\begin{theorem}\label{compa1} The operators $H_{\bold x}$ extend to commuting self-adjoint compact operators on $\mathcal H$. 
\end{theorem} 
 
\begin{proof} The proof is analogous to the proof of Theorem \ref{compa}, using the representation theory of $PGL_2(\Bbb K)$.  
\end{proof}  
 
\begin{example} Let $F=\Bbb R$ and $\bold x=(x,\overline x)$ where $x\in \Bbb C$. 
Let $\bold s=(s,\overline s)$. Then the formula of Example \ref{symfor}(2) yields 
$$
(\Bbb H_{x,\overline x}\psi)(y_0,...,y_m)=
$$
$$
\frac{1}{4}\int_{\Bbb C}\psi\left(\frac{{\rm Im}((s-y_0)(t_0-\overline{x})\overline s)}{{\rm Im}((s-y_0)(t_0-\overline x))},...,\frac{{\rm Im}((s-y_m)(t_m-\overline{x})\overline s)}{{\rm Im}((s-y_m)(t_m-\overline x))}
\right)\frac{|{\rm Im}(x){\rm Im}(s)^{m-1}|dsd\overline{s}}{\prod_{i=0}^m|{\rm Im}((s-y_j)(t_j-\overline{x}))|}.
$$
where $dsd\overline s$ is the ordinary Lebesgue measure on $\Bbb C$. Thus, introducing real coordinates by 
$s:=u+iv,x:=a+ib$, we get 
$$
(\Bbb H_{x,\overline x}\psi)(y_0,...,y_m)=
$$
\scriptsize
$$
\frac{1}{4}\int_{\Bbb R^2}\psi\left(\frac{(v(t_0-a)-ub)y_0+(u^2+v^2)b}{by_0+v(t_0-a)-ub},...,\frac{(v(t_m-a)-ub)y_m+(u^2+v^2)b}{by_m+v(t_m-a)-ub}
\right)\frac{|bv^{m-1}|dudv}{\prod_{i=0}^m|by_i+v(t_i-a)-ub|}.
$$
\normalsize
\end{example} 

\section{Genus zero, the archimedian case} 

In this section we will consider the archimedian case, i.e., we assume that $F=\Bbb R$ or $F=\Bbb C$.\footnote{For simplicity for $F=\Bbb R$ we restrict to the case $t_i\in \Bbb R$, although the results can be generalized 
to the case when there are complex conjugate pairs of parabolic points.} 
The modified Hecke operators $\Bbb H_x$ are given by the formula
$$
(\Bbb H_x\psi)(y_0,...,y_m)=\frac{1}{2}\int_{\Bbb R}\psi\left(\frac{t_0-x}{s-y_0},...,\frac{t_m-x}{s-y_m}\right)\frac{ds}{\prod_{i=0}^m |s-y_i|} 
$$
for $F=\Bbb R$ and 
$$
(\Bbb H_x\psi)(y_0,...,y_m)=\frac{1}{\pi}\int_{\Bbb C}\psi\left(\frac{t_0-x}{s-y_0},...,\frac{t_m-x}{s-y_m}\right)\frac{dsd\overline{s}}{\prod_{i=0}^m |s-y_i|^2}
$$
for $F=\Bbb C$.  

\subsection{The Gaudin system} 

\begin{definition} The {\bf Gaudin operators} are second order differential operators in the variables $y_0,...,y_m$ defined by 
the formula 
$$
G_i:=\sum_{0\le j\le m, j\ne i}\frac{1}{t_i-t_j}\left(-(y_i-y_j)^2\partial_i\partial_j+(y_i-y_j)(\partial_i-\partial_j)+\frac{1}{2}\right). 
$$
\end{definition} 

It is easy to see that 
$$
\sum_i G_i=0,\ G_i^*=G_i
$$ 
(i.e., $G_i$ are algebraically symmetric), and on translation invariant functions of homogeneity degree $-\frac{m}{2}$ 
one has 
$$
\sum_i t_iG_i=\frac{m}{4}.
$$
It is well known and easy to check that $[G_i,G_j]=0$. Therefore the operators $G_i$ form a 
quantum integrable system. Namely, it is a the
quantum Hitchin system for $G=PGL_2$ for $X=\Bbb P^1$ with parabolic points.\footnote{For more details on the Gaudin operators see \cite{EFK}, Section 7 and references therein.} 
Hence for every $\bold \bmu=(\mu_0,...,\mu_m)\in \Bbb C^{m+1}$ such that $\sum_i \mu_i=0$ and $\sum_i t_i\mu_i=\frac{m}{4}$, we have 
the holonomic system of differential equations
$$
G_i\psi=\mu_{i}\psi, i=1,...,m-1
$$
on a translation-invariant function $\psi(y_0,...,y_m)$ (under simultaneous translation of all variables) of homogeneity degree $-\frac{m}{2}$, which is called the {\bf Gaudin system}. As explained in Subsection \ref{hit}, the Gaudin system defines an $O$-coherent twisted $D$-module\footnote{The twisting here is by the line bundle of half-forms $K^{\frac{1}{2}}$. We will often regard the Gaudin system as a usual $D$-module by tensoring it with the dual bundle.}  $M(\bmu)$ on the open subset $Bun^{\rm vs}_0\subset Bun^\circ_0$ of very stable bundles, of rank $2^{m-1}$.  

\begin{proposition} \label{irredu} The $D$-module $M(\bmu)$ is
irreducible. 
\end{proposition} 

\begin{proof}
The lemma follows from the explicit identification of $M(\bmu)$ with the
$D$-module obtained by Drinfeld's first construction \cite{Dr} from
the symmetric power of the $D$-module on $\Bbb P^1$ attached to
the corresponding oper. This identification (called the {\bf separation of variables transform}) is explained in \cite{F:icmp}, Subsections 6.5 and 6.6, using the results of
\cite{Skl}. The irreducibility of the oper $D$-module on $\Bbb P^1$ then
implies the irreducibility of Drinfeld's $D$-module, and hence
the irreducibility of the $D$-module $M(\bmu)$ corresponding to the Gaudin
system.\footnote{Note that using the results of Gaitsgory
\cite{Gai:outline} on the uniqueness of Hecke eigensheaves one can
identify the quantum Hitchin $D$-module on $\Bun_{PGL_2}$ constructed
by Beilinson and Drinfeld \cite{BD}, which generalizes the Gaudin
$D$-module, with the $D$-module obtained by Drinfeld's first
construction \cite{Dr} for an arbitrary curve.} 
\end{proof} 

\subsection{Differential equations for Hecke operators}\label{Differential equations for Hecke operators}

In this section we show that the Hecke operators $H_x$ satisfy a second order differential equation with respect to $x$ which can be used to describe their spectrum more explicitly.

Let 
$$
\widehat G_i:=G_i-\sum_{j\ne i}\frac{1}{2(t_i-t_j)}.
$$

\begin{proposition}\label{opereq} (i) We have 
$$
\left(\partial_x^2+\sum_{j\ge 0}\frac{1}{x-t_j}\partial_x\right)\Bbb H_x-\Bbb H_x\sum_{i\ge 0}\frac{\widehat G_i}{x-t_i}=0.
$$
More precisely, let $U\subset F^{m}\setminus 0$ be the open set defined in Subsection \ref{Hecke operators} and $\psi$ be a smooth function on $U$ 
homogeneous of degree $-\frac{m}{2}$ whose support modulo dilations is compact. Then the function $x\mapsto \Bbb H_x\psi\in \mathcal{H}$ is smooth 
for $x\ne t_i,\infty$ and we have 
$$
\left(\partial_x^2+\sum_{j\ge 0}\frac{1}{x-t_j}\partial_x\right)(\Bbb H_x\psi)-\Bbb H_x\sum_{i\ge 0}\frac{\widehat G_i\psi}{x-t_i}=0
$$
in $\mathcal{H}$. 

(ii) In the same sense, we have 
$$
\left(\partial_x^2+\sum_{i\ge 0}\frac{1}{4(x-t_i)^2}\right)H_x-
H_x\sum_{i\ge 0}\frac{G_i}{x-t_i}=0. 
$$ 
\end{proposition} 

\begin{remark} This is essentially a special case of \cite{EFK2}, Theorem 
1.15, but here we will give a more elementary proof by direct computation. 
\end{remark}

\begin{proof} Let $u_i=u_i(s):=y_i-s$ and $\psi_i,\psi_{ij}$ be the first and second derivatives of $\psi$ evaluated at the point $\bold z$ with coordinates $z_i:=\frac{t_i-x}{u_i}$. Also let $d\mu(s):=\norm{\frac{ds}{\prod_{i=0}^m (s-y_m)}}$. Let $\eta$ be another smooth function on $U$ homogeneous of degree $-\frac{m}{2}$ with compact support modulo dilations. 
Note that $\sum_{j\ge 0}\psi_j=0$ (as $\psi$ is translation invariant)
and $\sum_{j\ge 0}\partial_j\psi_0=-\partial_s\psi_0$. Thus we have 
$$
\partial_x^2(\eta,\Bbb H_x\psi)=\left(\eta,\int_{F}\sum_{i,j\ge 0}\frac{\psi_{ij}}{u_iu_j}d\mu(s)\right), 
$$
$$
\left(\eta,\sum_{i\ge 0}\frac{\Bbb H_x\widehat G_i\psi}{x-t_i}\right)=\left(\eta,\int_{F}\sum_{i\ne j}\frac{-(\frac{t_i-x}{u_i}-\frac{t_j-x}{u_j})^2\psi_{ij}+(\frac{t_i-x}{u_i}-\frac{t_j-x}{u_j})(\psi_i-\psi_j)}{(x-t_i)(t_i-t_j)}d\mu(s)\right).
$$
Subtracting, we get 
$$
\partial_x^2(\eta,\Bbb H_x\psi)-\left(\eta,\Bbb H_x\sum_{i\ge 0}\frac{\widehat G_i\psi}{x-t_i}\right)=
$$
$$
\left(\eta,\int_{F}\left(\sum_i\frac{\psi_{ii}}{u_i^2}-\sum_{i\ne j}\frac{-(\frac{(t_i-x)^2}{u_i^2}+\frac{(t_j-x)^2}{u_j^2})\psi_{ij}-(\frac{t_i-x}{u_i}-\frac{t_j-x}{u_j})(\psi_i-\psi_j)}{(x-t_i)(t_i-t_j)}\right)d\mu(s)\right)=
$$
$$
\left(\eta,\int_{F}\left(\sum_{i,j\ge 0}\frac{(t_i-x)\psi_{ij}}{(t_j-x)u_i^2}-\sum_{i\ne j}\frac{(\frac{t_i-x}{u_i}-\frac{t_j-x}{u_j})(\psi_i-\psi_j)}{(x-t_i)(t_i-t_j)}\right)d\mu(s)\right).
$$
Now, using integration by parts (justified since $\eta,\psi$ have compact 
support modulo dilations which is contained in $U$), we have 
$$
\left(\eta,\int_{F}\sum_{i,j\ge 0}\frac{(t_i-x)\psi_{ij}}{(t_j-x)u_i^2}d\mu(s)\right)=
\left(\eta,\int_{F}\sum_{j\ge 0}\frac{1}{x-t_j}\sum_{i\ge 0}\partial_i\psi_{j}d\mu(s)\right)=
$$
$$
\left(\eta,\int_{F}\sum_{j\ge 0}\frac{1}{t_j-x}\partial_s \psi_jd\mu(s)\right)=
\left(\eta,\int_{F}\sum_{i,j\ge 0}\frac{1}{(t_j-x)u_i}\psi_jd\mu(s)\right). 
$$
Thus we get 
$$
\partial_x^2(\eta,\Bbb H_x\psi)-\left(\eta,\Bbb H_x\sum_{i\ge 0}\frac{\widehat G_i\psi}{x-t_i}\right)=
$$
$$
\left(\eta,\int_{F}\left(\sum_{i,j\ge 0}\frac{1}{(t_j-x)u_i}\psi_j+\frac{1}{2}\sum_{i\ne j}\frac{(\frac{t_i-x}{u_i}-\frac{t_j-x}{u_j})(\psi_i-\psi_j)}{(t_i-x)(t_j-x)}\right)d\mu(s)\right)=
$$
$$
\left(\eta,\int_{F}\left(\sum_{i\ge 0}\frac{1}{(t_j-x)u_i}\psi_i+\sum_{i\ne j}\frac{1}{(t_j-x)u_i}\psi_i\right)d\mu(s)\right)=
$$
$$
\left(\eta,\sum_{j\ge 0}\frac{1}{t_j-x} \int_{F}\sum_{i\ge 0}\frac{\psi_i}{u_i}d\mu(s)\right)=\sum_{j\ge 0}\frac{1}{x-t_j}\partial_x\left(\eta,\Bbb H_x\psi\right).
$$
Thus
$$
\left(\partial_x^2-\sum_{j\ge 0}\frac{1}{x-t_j}\partial_x\right)(\eta,\Bbb H_x\psi)=\left(\eta,\Bbb H_x\sum_{i\ge 0}\frac{\widehat G_i\psi}{x-t_i}\right),
$$
which implies that 
\begin{equation}\label{diffe}
\partial_x^2(\eta,H_x\psi)=
\left(\eta,H_x\sum_{i\ge 0}\left(\frac{G_i}{x-t_i}-\frac{1}{4(x-t_i)^2}\right)\psi\right). 
\end{equation} 
At the algebraic level, we are now done, as this is the claimed equation 
for the matrix coefficient $(\eta,H_x\psi)$. Analytically, however,  
getting rid of $\eta$ is not automatic, as we are working in an infinite dimensional Hilbert space. To start with, we need to show that $H_x\psi$ is twice differentiable. So in order to dispose of $\eta$ and complete the proof, we will perform a double integration to turn differential equation \eqref{diffe} into an integral equation.\footnote{This helps because integrals are better behaved analytically than derivatives, and is similar to 
the beginning of the standard proof of Picard's theorem on the existence of solutions of ODE.}
 
Namely, picking a point $x_0\ne t_i$ for any $i$, we get (for fixed $\psi$): 
$$
(\eta,H_x\psi)=
\int_{x_0}^x\int_{x_0}^y\left(\eta,H_x\sum_{i\ge 0}\left(\frac{G_i}{x-t_i}-\frac{1}{4(x-t_i)^2}\right)\psi\right)dtdy+c_0(\eta)+c_1(\eta)x= 
$$
$$
\left(\eta,\int_{x_0}^x\int_{x_0}^yH_x\sum_{i\ge 0}\left(\frac{G_i}{x-t_i}-\frac{1}{4(x-t_i)^2}\right)\psi dtdy\right)+c_0(\eta)+c_1(\eta)x.
$$
By Proposition \ref{compa}, the left hand side of this equation and the first summand on the right hand side are continuous in $\eta$ in the metric of $\mathcal{H}$. Therefore, the second and third summands on the right 
hand side are also continuous. So, since the possible $\eta$ are dense in 
$\mathcal{H}$,  we have 
$$
H_x\psi=
\int_{x_0}^x\int_{x_0}^yH_x\sum_{i\ge 0}\left(\frac{G_i}{x-t_i}-\frac{1}{4(x-t_i)^2}\right)\psi dtdy+c_0+c_1 x. 
$$
for some $c_0,c_1\in \mathcal H$. This implies that $H_x\psi$ is twice differentiable in $x$. So differentiating twice, we get
$$
\left(\partial_x^2+\sum_{i\ge 0}\frac{1}{4(x-t_i)^2}\right)(H_x\psi)
-H_x\sum_{i\ge 0}\frac{G_i\psi}{x-t_i}=0
$$ 
and 
$$
\left(\partial_x^2+\sum_{j\ge 0}\frac{1}{x-t_j}\partial_x\right)(\Bbb H_x\psi)-\Bbb H_x\sum_{i\ge 0}\frac{\widehat G_i\psi}{x-t_i}=0,
$$
as claimed. 
\end{proof}

\begin{example}\label{ellint2} Let $m=1$, $t_0=0,\ t_1=1$. In this case, as explained in Example \ref{ellint1}, $\mathcal{H}$ is 1-dimensional 
and $H_x=\norm{x(x-1)}^{\frac{1}{2}}E(x)$. Thus $\Bbb H_x=E(x)$ is given by the elliptic integral. Also $G_i$, $i=0,1$ act by the numbers $\mu_i$ such that $\mu_0=-\frac{1}{4},\mu_1=\frac{1}{4}$, so $\widehat G_0=\frac{1}{4},\widehat G_1=-\frac{1}{4}$. So the equation of Proposition 
\ref{opereq} takes the form
$$
\left(\partial_x^2+\left(\frac{1}{x}+\frac{1}{x-1}\right)\partial_x+\frac{1}{4x(x-1)}\right)E(x)=0.
$$ 
This is the classical Picard-Fuchs equation for the elliptic integral. 
\end{example} 

Now recall that we have a spectral decomposition $\mathcal H=\oplus_k \mathcal H_k$ 
with respect to the action of the operators $H_x$. Recall also that $\mathcal H$ 
is naturally a subspace of the space of distributions on $U$. 

\begin{proposition}\label{eigenfun} 
Let $\eta\in \mathcal H_k$. Then for all $i$ the distribution $G_i\eta$ on $U$ belongs to $\mathcal H$ and equals $\mu_{i,k}\eta$ for some scalar $\mu_{i,k}\in \Bbb C$. \end{proposition} 

\begin{proof} Let $\eta\in \mathcal{H}_k$. By Proposition \ref{opereq}, 
$$
\left(\partial_x^2+\sum_{i\ge 0}\frac{1}{4(x-t_i)^2}\right)(H_x\psi,\overline\eta)
-\left(H_x\sum_{i\ge 0}\frac{G_i\psi}{x-t_i},\overline\eta\right)=0.
$$ 
But $\overline\eta\in \mathcal H_k$, so $H_x\overline\eta=\beta_k(x)\overline\eta$. Thus, using that $H_x$ is self-adjoint and $G_i^*=G_i$, we 
get 
$$
\left(\partial_x^2+\sum_{i\ge 0}\frac{1}{4(x-t_i)^2}\right)\beta_k(x)(\psi,\overline\eta)
-\sum_{i\ge 0}\frac{1}{x-t_i}\beta_k(x)(\psi,\overline{G_i\eta})=0.
$$ 
This holds for all test functions $\psi$, so we get 
$$
\sum_{i\ge 0}\frac{G_i}{x-t_i}\eta=\left(\frac{\partial_x^2\beta_k(x)}{\beta_k(x)}+\sum_{i\ge 0}\frac{1}{4(x-t_i)^2}\right)\eta.
$$ 
Thus, for any $x$ the operator $\sum_{i\ge 0}\frac{G_i}{x-t_i}$ acts on $\mathcal H_k$ by a scalar. Applying this statement for distinct $x_0,...,x_m$ and using that 
$\det(\frac{1}{x_j-t_i})\ne 0$ (the Cauchy determinant), we deduce that 
for all $i$ the operator $G_i$ acts on the space $\mathcal H_k$ by some scalar $\mu_{i,k}$, as claimed. 
\end{proof}

\subsection{The Schwartz space} \label{Schwartz space} 
For $\psi\in \mathcal H$ let $\psi_k$ be the projection of $\psi$ to $\mathcal H_k$.

\begin{definition} Define the {\bf Schwartz space} $\mathcal S\subset \mathcal H$ to be the space of vectors $\psi\in \mathcal H$ such that 
for any $i_1,...,i_r$ 
$$
\sum_k |\mu_{i_1,k}...\mu_{i_r,k}|^2\norm{\psi_k}^2<\infty,
$$
where $\mu_{i,k}$ are the scalars appearing in Proposition \ref{eigenfun}. 
\end{definition} 

Let $\mathcal{A}$ be the commutative algebra of differential operators on 
$Bun_0^\circ(F)$ (regarded as a real analytic manifold) generated 
by $G_i$ for $F=\Bbb R$ and by $G_i,\overline G_i$ for $F=\Bbb C$. 
This algebra has a conjugation map $\dagger$ given by $G_i^\dagger=\overline G_i$ 
(where $\overline{G_i}=G_i$ for $F=\Bbb R$). 

Let $\bmu_k: \mathcal{A}\to \Bbb C$ be the conjugation-equivariant character defined by $\bmu_k(G_i)=\mu_{i,k}$. Thus $\mathcal S$ is the space of $\psi$ such that for every $A\in \mathcal A$ one has 
$$
\sum_k |\bmu_k(A)|^2\norm{\psi_k}^2<\infty.
$$

Recall that $U\subset F^{m}\setminus 0$ denotes the open set defined in Subsection \ref{Hecke operators}. 

\begin{definition} Define $\mathcal V=C^\infty_0$ to be the space of smooth functions on  $U$ of homogeneity degree $-\frac{m}{2}$ with compact support modulo dilations.  
\end{definition} 

Note that $\mathcal A$ acts on $\mathcal V$. 

\begin{proposition}\label{VinS} One has $\mathcal V\subset \mathcal S$. Moreover, for $\phi\in \mathcal V$, $A\in \mathcal A$ we have 
$$
A\phi=\sum_k \bmu_k(A)\phi_k.
$$
\end{proposition} 

\begin{proof} Let $\phi\in \mathcal V$. Then for any $A\in \mathcal A$, 
we have $A\phi\in \mathcal V\subset \mathcal H$. 
Thus we get 
$$
\norm{A\phi}^2=\sum_k \norm{(A\phi)_k}^2<\infty.
$$
But by Proposition \ref{eigenfun} for $v\in \mathcal H_k$ we have 
$$
(v,(A\phi)_k)=(v,A\phi)=(\overline A v,\phi)=\overline{\bmu_k(A)}(v,\phi)=\overline{\bmu_k(A)}(v,\phi_k)=(v,\bmu_k(A)\phi_k).
$$
Thus $(A\phi)_k=\bmu_k(A)\phi_k$. This implies both statements.  
\end{proof} 

We can now define a representation of the algebra $\mathcal{A}$ on $\mathcal{S}$ by the formula
$$
A\psi:=\sum_k\bmu_k(A)\psi_k.
$$
By Proposition \ref{VinS}, this extends to $\mathcal{S}$ the usual action 
of $\mathcal{A}$ on $\mathcal V$. So from now on let us regard $\mathcal{A}$ as an algebra of endomorphisms of $\mathcal{S}$. Note that for each $v,u\in \mathcal{S}$, we have 
$$
(Av,u)=\sum_k((Av)_k,u_k)=\sum_k \bmu_k(A)(v_k,u_k)=\sum_k(v_k,(A^\dagger u)_k)=(v,A^\dagger u).
$$
 Thus the subalgebra $\mathcal{A}_{\Bbb R}$ of elements $A\in \mathcal{A}$ such that $A^\dagger=A$ acts on $\mathcal{S}$ by symmetric operators.

Recall (\cite{EFK}, Definition 11.7) that $S(\mathcal{A})$ denotes 
the space of $u\in \mathcal{H}$ such that the linear functional 
$v\mapsto (Av,u)$ on $\mathcal{S}$ is continuous in the metric of $\mathcal{H}$ 
for all $A\in \mathcal{A}$. 

\begin{proposition} (i) One has $S(\mathcal{A})=\mathcal{S}$. 

(ii) Any $f\in \mathcal{S}$ is smooth on the set $Bun^{\rm vs}(F)$ of very stable bundles. 
\end{proposition} 

\begin{proof} (i) If $u\in \mathcal{S}$ then $(Av,u)=(v,A^\dagger u)$, so 
it is continuous in $v$, hence $u\in S(\mathcal A)$. 
Conversely, suppose $u\in S(\mathcal{A})$. Then $(Av,u)=(v,w)$ for some 
$w\in \mathcal{H}$. 
So we have 
$$
\sum_k (v_k,w_k)=(v,w)=(Av,u)=\sum_k((Av)_k,u_k)=\sum_k \bmu_k(A)(v_k,u_k). 
$$
Taking $v=v_k\in \mathcal{H}_k \subset \mathcal{S}$, we have 
$w_k=\overline{\bmu_k(A)}u_k$ for all $k$. Thus 
$$
\sum_k |\bmu_k(A)|^2\norm{u_k}^2=\sum_k \norm{w_k}^2<\infty, 
$$
i.e., $u\in \mathcal{S}$. 

(ii) This follows since the quantum Hitchin system is holonomic, hence elliptic on $Bun^{\rm vs}$ (cf. \cite{EFK}, Example 11.12(3)). Namely, the holonomicity of the quantum Hitchin system on $Bun^{\rm vs}$ implies that for every $E\in Bun^{\rm vs}(F)$ there exist {\it real} differential operators $D_1,...,D_r\in \mathcal A$ such that the common zero set of their symbols in $T^*_E Bun^{\rm vs}$ is $\lbrace 0\rbrace$. Then the operator $D:=D_1^2+...+D_r^2\in \mathcal A$ is elliptic near $E$, and for any $n\ge 0$ the function $D^nf$ is square integrable. This implies that $f$ is smooth near $E$. 
\end{proof} 

Recall (\cite{EFK}, Definition 11.8) that the algebra $\mathcal A_{\Bbb R}$ is said to be {\bf essentially self-adjoint} on $\mathcal{S}$ if every 
$A\in \mathcal{A}_{\Bbb R}$ 
is essentially self-adjoint on $S(\mathcal{A})$. 

\begin{proposition}\label{strocomm} 
The algebra $\mathcal A_{\Bbb R}$ is essentially self-adjoint on $\mathcal{S}$. 
In particular, $G_i, \overline{G_i}$ are unbounded normal operators on $\mathcal{H}$ (self-adjoint for $F=\Bbb R$) which (strongly) commute with 
each other and with the Hecke operators $H_x$ (see \cite{EFK}, Subsection 
11.2). 
\end{proposition}

\begin{proof} This follows immediately from the fact that $S(\mathcal{A})=\mathcal{S}$ contains $\mathcal{H}_k$ for all $k$. 
\end{proof} 

Proposition \ref{strocomm} immediately yields

\begin{corollary}\label{commuta} For every $\phi\in \mathcal V$, $\bold x=(x_1,...,x_n)$ (with $x_i\ne t_j,\infty)$
and $0\le i\le m$, the distribution $G_iH_{\bold x}\phi$ belongs to $\mathcal{H}$ 
and we have $G_iH_{\bold x}\phi=H_{\bold x}G_i\phi$. 
\end{corollary}

We also obtain 

\begin{corollary} For $\psi\in \mathcal{H}$, the map $x\mapsto H_x\psi$ 
is twice differentiable in $x$ as a function with values in distributions 
on $U$ for $x\ne t_i,\infty$, and we have 
$$
\left(\partial_x^2+\sum_{i\ge 0}\frac{1}{4(x-t_i)^2}\right)(H_x\psi)=
\sum_{i\ge 0}\frac{G_i}{x-t_i}H_x\psi 
$$ 
as distributions on $U$. 
\end{corollary} 

\begin{proof} For a test function $\phi\in \mathcal V$, we have $(H_x\psi,\phi)=(\psi,H_x\phi)$, which is twice differentiable in $x$ by Proposition \ref{opereq}. Thus again using Proposition \ref{opereq}, 
$$
\left(\left(\partial_x^2+\sum_{i\ge 0}\frac{1}{4(x-t_i)^2}\right)(H_x\psi),\phi\right)=
\left(\partial_x^2+\sum_{i\ge 0}\frac{1}{4(x-t_i)^2}\right)(H_x\psi,\phi)=
$$
$$
\left(\partial_x^2+\sum_{i\ge 0}\frac{1}{4(x-t_i)^2}\right)(\psi,H_x\phi)=
\left(\psi,\left(H_x\sum_i \frac{G_i}{x-t_i}\right)^\dagger\phi\right)=
\left(\sum_i \frac{G_i}{x-t_i}H_x\psi,\phi\right),
$$
as claimed. 
\end{proof} 

\subsection{The differential equation for eigenvalues}\label{The differential equation for eigenvalues} 

Proposition \ref{opereq} implies

\begin{corollary}\label{spec2} The function $\beta_k(x)$ 
satisfies the differential equation 
\begin{equation}\label{ope}
L(\bmu_k)\beta_k(x)=0,
\end{equation}
where 
$$
L(\bmu):=\partial_x^2+\sum_{i\ge 0}\frac{1}{4(x-t_i)^2}-\sum_{i\ge 0}\frac{\mu_{i,k}}{x-t_i}
$$ 
is an {\bf $SL_2$-oper} (where $\sum_i \mu_i=0,\ \sum_i t_i\mu_i=\frac{m}{4}$). 
\end{corollary} 

Note that equation \eqref{ope} is Fuchsian at the points $t_i$ with characteristic exponents $\frac{1}{2},\frac{1}{2}$ and, since 
$$
\sum_i \mu_{i,k}=0,\ \sum_i t_i\mu_{i,k}=\frac{m}{4},
$$
 it is also Fuchsian at $\infty$ with characteristic exponents 
 $-\frac{1}{2},-\frac{1}{2}$. In other words, basic solutions behave 
 as $(x-t_i)^{\frac{1}{2}}$ and $(x-t_i)^{\frac{1}{2}}\log(x-t_i)$ 
 near $t_i$ and as $x^{\frac{1}{2}},x^{\frac{1}{2}}\log x$ at $\infty$.\footnote{We remind the reader that since the oper $L(\bmu)$ is a map from $K^{-\frac{1}{2}}$ to $K^{\frac{3}{2}}$, solutions of the equation $L(\bmu)\beta=0$ are sections of 
 $K^{-\frac{1}{2}}$, but we view them as functions by multiplying by $(dx)^{-\frac{1}{2}}$.}
Thus the monodromy operators of \eqref{ope} at $t_i$ and $\infty$ are conjugate 
 to $\begin{pmatrix} -1 & 1\\ 0 & -1\end{pmatrix}$.
  
 \subsection{Spectral decomposition in the complex case}\label{The complex 
case} Let $F=\Bbb C$ and
 $\mathcal{R}\subset \Bbb C^{m-1}$ be the set of points $\bmu$ such that the oper $L(\bmu)$, when viewed as a rank $2$ local system, has {\bf a real monodromy representation} (i.e., landing in $SL_2(\Bbb R)$ up to conjugation). It is 
known that $\mathcal{R}$ is discrete (see \cite{F}, Section 7, Corollary to Theorem 12).

\begin{theorem}\label{mainth} 
(i) The Hecke operators $H_x$ have a simple joint spectrum $\Sigma$ on $\mathcal H$. 

(ii) For any $\bmu\in \mathcal{R}$ there is a unique up to scaling nonzero single-valued real analytic half-density $\psi_\bmu$, defined outside the wobbly divisor $D$
in $Bun_0^\circ(\Bbb C)$ (see Subsectioin \ref{hit}), such that 
\begin{equation}\label{Gau}
G_i\psi_\bmu=\mu_i\psi_\bmu,\ \overline{G_i}\psi_\bmu=\overline\mu_i\psi_\bmu.
\end{equation} 

(iii) There is a natural inclusion $\Sigma\hookrightarrow \mathcal{R}$.

(iv) Let $0\ne \psi\in \mathcal{H}$ be such that $G_i\psi=\mu_i\psi$ as distributions on $U$ for some $\mu_i\in \Bbb C$. Then $H_x\psi=\beta(x)\psi$ for some function $\beta(x)$. 
Thus $\psi\in \mathcal{H}_k$ for some $k$ and $\beta(x)=\beta_k(x)$. 

(v)  
$\bmu\in \mathcal{R}$ belongs to $\Sigma$ (i.e., $\psi_\bmu$ is an eigenfunction) if and only if $\psi_\bmu$ belongs to the Schwartz space $\mathcal{S}$. 
\end{theorem} 

\begin{remark} 1. We expect that the condition in (v) always holds, so $\Sigma=\mathcal{R}$ (this is Conjecture 1.8(2) in \cite{EFK2}). We will see below that this is true in the case of four and five points. 

2. See Footnote \ref{foot1} on the difference in normalization of eigenvalues of the Hecke operators used in the present paper and in Conjecture 1.11 of \cite{EFK2}.
\end{remark} 

\begin{proof} (i),(ii),(iii) Consider a joint eigenspace $\mathcal H_k$ of $H_x$.  
By Proposition \ref{eigenfun}, the operators $G_i$ act on this space by some scalars 
$\mu_{i,k}$, and by Corollary \ref{spec2} the eigenvalue $\beta_k(x)$ of $H_x$ on $\mathcal H_k$ 
satisfies the differential equations 
\begin{equation}\label{opersys}
L(\bmu_k)\beta_k=0,\ \overline{L(\bmu_k)}\beta_k=0. 
\end{equation} 
Since these equations have a non-zero single-valued solution $\beta_k(x)$, the oper $L(\bmu_k)$ has a real monodromy representation, i.e., $\bmu_k\in \mathcal{R}$ (see e.g. \cite{EFK}, Subsection 3.4, \cite{EFK2}, Subsection 5.1). Furthermore, since an oper connection 
is always irreducible (\cite{BD,BD:opers}), this single-valued solution is unique up to scaling. Hence by Corollary \ref{asymeig} it is uniquely determined by $\bmu_k$ (namely, its asymptotics at $\infty$ fixes the scaling). This gives an inclusion $\Sigma\hookrightarrow \mathcal R$.  

Also every element $\psi\in \mathcal H_k$ is a single-valued real analytic solution
of the holonomic system \eqref{Gau}. By Proposition \ref{irredu}, 
the holomorphic part of this system is an irreducible $D$-module. This implies that
$\psi$ is unique up to scaling, i.e., $\psi=\psi_{\bmu_k}$, and $\dim
\mathcal H_k=1$.

(iv) For $N=4$ this is shown in the next section, so assume that $N\ge 5$. 
In this case let $\mathcal V_0\subset \mathcal V$ be the subspace of functions $\phi$ supported on the open set $Bun^{\rm vs}(F)\subset U$ of very stable bundles (i.e., ones that have no nonzero nilpotent Higgs field, see Subsection \ref{hit}).Then by Corollary \ref{ves1}, 
 $H_x\phi\in \mathcal V$. By Corollary \ref{commuta}, 
$G_iH_x\phi=H_xG_i\phi$ as distributions, hence as functions on $U$ (as 
they are both elements of $\mathcal V$).
 Thus for $\phi\in \mathcal V_0$ we have 
$$
(G_iH_x\psi,\overline{\phi})=(H_x\psi,\overline{G_i\phi})=(\psi,\overline{H_x G_i\phi})=
(\psi,\overline{G_iH_x\phi}),
$$
Moreover, since $H_x\phi\in \mathcal V$, we have 
$$
(\psi,\overline{G_iH_x\phi})=(G_i\psi,\overline{H_x\phi})=\mu_i(\psi,\overline{H_x\phi})=(\mu_iH_x\psi,\overline{\phi}).
$$
Thus 
$$
G_iH_x\psi=\mu_iH_x\psi
$$ 
as distributions on $Bun^{\rm vs}$. Since the monodromy representation of the system 
$G_i\psi=\mu_i\psi$ is irreducible, there is at most one single-valued solution up to scaling. Therefore generically on $U$ we have $H_x\psi=\beta(x)\psi$, as claimed.  
 
(v) It is clear that 
$\psi_{\bmu_k}\in \mathcal S\subset \mathcal{H}$. 
Conversely, if $\psi_{\bmu}\in \mathcal H$ 
for some $\bmu\in \mathcal{R}$ then by (iv) we have $\psi_{\bmu}\in \mathcal{H}_k$ for some $k$. 
\end{proof} 

Theorem \ref{mainth} implies the validity of the main conjectures of \cite{EFK} and \cite{EFK2} (namely Conjectures 1.4, 1.9, 1.10 and 1.11 of \cite{EFK} and Conjectures 1.5, 1.11 of \cite{EFK2}) for $G=PGL_2$ and curves of  genus zero.\footnote{Note that the normalization of the eigenvalues $\beta_k(x)$ is uniquely determined by Theorem \ref{specthe}(ii). This way of normalization is possible due to presence of parabolic points.}

\begin{remark} Let $Y$ be the (2-dimensional) space of solutions of the oper equation
$$
L(\bmu_k)f=0
$$
 near some point $x_0\in \Bbb C\Bbb P^1$, $x_0\ne t_0,...,t_{m+1}$.
Then $Y$ carries a monodromy-invariant symplectic form sending $(f,g)$ to the Wronskian $W(f,g)$ (which is a constant function since $L(\bmu_k)$ has no first derivative term). It also carries a monodromy-invariant pseudo-Hermitian inner product $(f,g)\mapsto B(f,g)$, since the monodromy of the corresponding local system is in $SL_2(\Bbb R)\cong SU(1,1)$. Let $f_0,f_1$ be a basis of $Y$ with $B(f_0,f_0)=B(f_1,f_1)=0$ and $B(f_0,f_1)=1$. Assume that $B$ is normalized so
that $|W(f_0,f_1)|=1$ (there are two such normalizations differing by sign). 

Then by Corollary \ref{spec2} the corresponding eigenvalue $\beta_k$ has the form 
$$
\pm\beta_k(x,\overline x)=f_0(x)\overline{f_1(x)}+f_1(x)\overline{f_0(x)}.
$$
Similarly, if $B(f_0,f_0)=1$, $B(f_1,f_1)=-1$ and $B(f_0,f_1)=0$, then we have
$$
\pm\beta_k(x,\overline x)=|f_0(x)|^2-|f_1(x)|^2.
$$
These are special cases of the formula in \cite{EFK2}, Conjecture 1.11. 
Note that we can fix the sign by imposing the condition that $\beta_k(x,\overline x)$ is positive near $\infty$. 
\end{remark} 

\subsection{The leading eigenvalue of the Hecke operator} 
The positive (leading) eigenvalue $\beta_0(x)$ of $H_x$ has a special meaning.

\begin{proposition} 
$\beta_0(x)$ is a single-valued solution of the system \ref{opersys} corresponding to the analytic uniformization of the punctured Riemann surface $X^\circ:=\Bbb C\Bbb P^1\setminus \lbrace{t_0,...,t_{m+1}\rbrace}$ (see \cite{F, Go, Ta}). 
\end{proposition} 

\begin{proof} As explained, e.g., in \cite{Ta}, the only oper with real monodromy and positive single-valued solution $\beta(x)$ of the system \eqref{opersys} is the uniformization oper, and in this case the complete hyperbolic conformal metric of constant negative curvature on $X^\circ$ is given by the formula 
$$
ds^2=\beta(x)^{-2}
$$
(which is well defined as a conformal metric independently of the choice of the coordinate $x$ since $\beta(x)$ is naturally a $-\frac{1}{2}$-density on $X^\circ$). Since the leading eigenvalue 
$\beta_0(x)$ of $H_x$ is positive, it must correspond to the uniformization oper, i.e., $\beta_0(x)=\beta(x)$. 

In more detail, let $J: \Bbb C_+\to \Bbb C\Bbb P^1\setminus \lbrace{t_0,...,t_{m+1}}\rbrace$ be an analytic unformization map of the Riemann surface $\Bbb C\Bbb P^1\setminus \lbrace{t_0,...,t_{m+1}}\rbrace$ by the upper half-plane (recall that it is unique up to the action of $PSL_2(\Bbb R)$ on $\Bbb C_+$). 
Then we have the multivalued holomorphic function $K(x):=J^{-1}(x)$ on $\Bbb C\Bbb P^1\setminus \lbrace {t_0,...,t_{m+1}}\rbrace$. Define the multivalued holomorphic 
$-\frac{1}{2}$-densities
$$
f_0(x):=\frac{e^{\frac{\pi i}{4}}}{\sqrt{K'(x)}}(dx)^{-\frac{1}{2}},\
f_1(x):=\frac{e^{-\frac{\pi i}{4}}K(x)}{\sqrt{K'(x)}}(dx)^{-\frac{1}{2}}.
$$
One can show that they form a basis of solutions of the oper equation 
$L(\bmu_0)\beta=0$ (near some point $x_0\in \Bbb C\Bbb P^1\setminus \lbrace{t_0,...,t_{m+1}}\rbrace$) with Wronskian $1$ in which the monodromy of this equation is in $SU(1,1)$, the group of symmetries of the pseudo-Hermitian form 
${\rm Re}(z_1\overline{z_2})$ (namely, it is the corresponding Fuchsian group $\Gamma\subset SL_2(\Bbb R)\cong SU(1,1)$), and the action of $SL_2(\Bbb R)$ on $J$ is simply transitive on bases of solutions with this property.\footnote{Here the group $PSL_2(\Bbb R)$ gets replaced by its double cover $SL_2(\Bbb R)$ due to the ambiguity of the square root 
in the definition of $f_i(x)$.} Thus the real analytic $-\frac{1}{2}$-density
\begin{equation}\label{betafor} 
\beta(x,\overline x):=f_0(x)\overline {f_1(x)}+f_1(x)\overline{f_0(x)}=
\frac{2{\rm Im}K(x)}{|K'(x)|}(dxd\overline x)^{-\frac{1}{2}}
\end{equation} 
satisfies the system of the oper and anti-oper equations \eqref{opersys}: 
$$
L(\bmu_0)\beta=0,\ \overline{L}(\bmu_0)\beta=0,
$$
and is the only nonzero single-valued solution of these equations up to scaling 
(it is single-valued because the expression $\frac{2{\rm Im}K(x)}{|K'(x)|}$ is invariant under the action of $SL(2,\Bbb R)$ given by $K\mapsto \frac{aK+b}{cK+d}$). It is manifestly positive, and the hyperbolic metric is expressed as $\beta^{-2}$, as explained above. 
\end{proof} 

Thus we obtain an explicit formula \eqref{betafor} for the leading eigenvalues $\beta_0(x) = \beta(x)$ of the Hecke operators $H_x$ in terms of the uniformization map. Other eigenvalues $\beta_k(x)$ can also be written in the form \eqref{betafor} but with $K(x)$ now taking values in the complex plane (rather than the upper half-plane). For this reason, $\beta_k(x)$ vanishes on the union of finitely many analytic contours in $\Bbb C\Bbb P^1\setminus \lbrace{t_0,...,t_{m+1}}\rbrace$.
Thus the corresponding metric $\beta^{-2}$ has constant negative curvature away from these contours and the singularities at the contours which locally look like 
the Poincar\'e metric $\frac{dx^2+dy^2}{y^2}$ near the boundary $y=0$ of the upper half-plane (see \cite{Go} and Section 4 of \cite{Ta}).

\subsection{Spectral decomposition in the real case and balanced local systems}\label{The real case} Let $F=\Bbb R$ and $t_0<t_1<...<t_m$. 
In this case the oper equation $L(\bmu)\beta=0$ for the eigenvalue $\beta(x)$ of the Hecke operator $H_x$ derived in Subsection \ref{The differential equation for eigenvalues} is a second order linear differential equation on the circle $\Bbb R\Bbb P^1=S^1$ with regular singularities at $t_0,...,t_m,t_{m+1}=\infty$. 

To characterize the spectrum of Hecke operators in the real case, we need a replacement for the reality condition for opers used in the complex case.
To this end, even though we are now working over $\Bbb R$, we will need to consider oper connections in the complex domain. This agrees with the general principle that on the spectral side of the Langlands correspondence one should always consider a complex Lie group, regardless of the field $F$. 

Let ${\rm Loc}_m$ be the variety of irreducible rank $2$ local systems 
(i.e., locally constant sheaves) on $\Bbb C\Bbb P^1\setminus \lbrace t_0,...,t_{m+1}\rbrace$ with 
monodromies around $t_i$ conjugate to $\begin{pmatrix} -1 & 1\\ 0 & -1\end{pmatrix}$. By Lemma \ref{irreduci}, this is an irreducible 
smooth variety of dimension $2(m-1)$ (and for odd $m$ any local system 
with such monodromies is automatically irreducible).  

Let $\nabla\in {\rm Loc}_m(\Bbb C)$. 
Let $V_j^+$ be the space of sections of $\nabla$ on $(t_j,t_j+\varepsilon)$, 
and $V_j^-$ be the space of sections of $\nabla$ on $(t_j,t_j-\varepsilon)$
for small $\varepsilon>0$ (where we use the addition law of the circle). We have an isomorphism $\xi_j: V_j^+\cong V_j^-$ which assigns to a section $f$ the section $\xi_j(f):=\frac{i}{2}(f_--f_+)$ when $j\ne m+1$ and $\xi_j(f):=-\frac{i}{2}(f_--f_+)$ for $j=m+1$, where $f_+,f_-$ are the continuations of $f$ above and below the real axis, respectively. Using these isomorphisms, we identify $V_j^-$ with $V_j^+$ and denote the resulting space just by $V_j$.\footnote{The different sign for $j=m+1$ is chosen because for opers we identify half-densities with functions using the half-density $|dx|^{1/2}$ which has a singularity at $t_{m+1}=\infty$.} 
For $j\in \Bbb Z/(m+2)$ we denote by  
$$
B_j: V_j\cong V_j^-\to V_{j-1}^+\cong V_{j-1}
$$
the operator of continuation of sections along the real axis, and define  
$$
\Bbb B:=B_0...B_{m+1}.
$$ 

\begin{definition}\label{nonde} Let us say that $\nabla$ is {\bf nondegenerate} if there exists nonzero $f\in V_{m+1}$ which is not an eigenvector of the monodromy of $\nabla$ around $t_{m+1}=\infty$ but $\Bbb Bf=\lambda f$ for some $\lambda\in \Bbb C^\times$. 
\end{definition} 

It is clear that for a generic $\nabla$ the operator $\Bbb B$ is regular semisimple, so $\nabla$ is nondegenerate and there are two choices of $f$ corresponding to two different eigenvalues $\lambda_1,\lambda_2$ of $\Bbb B$, up to scaling. However, if $\Bbb B$ is a Jordan block (for nondegenerate $\nabla$) then there is only one choice of $f$ up to scaling, while for scalar $\Bbb B$ such choices form an affine line. 

Now let $\nabla$ be nondegenerate. Choose an eigenvector $f$ as in Definition \ref{nonde} and let 
$$
f_j:=B_{j+1}...B_{m+1}f, \ g_j:=f_j^+-if_j, \ j\ne m+1,\ g_{m+1}:=-f_{m+1}^+-if_{m+1}.
$$

\begin{definition} Let us say that the {\bf pair} $(\nabla,f)$ is {\bf nondegenerate} if
 for all $j\in [0,m+1]$, the vector 
 $f_j$ is not an eigenvector of the monodromy around 
$t_j$. 
\end{definition} 

In this case $f_j,g_j$ form a basis of $V_j$.

Again, it is clear that for generic $\nabla$ any pair $(\nabla,f)$ is nondegenerate. 

Let $(\nabla,f)$ be nondegenerate. Then in the bases $f_j,g_j$, the half-monodromy around $t_j$ in the negative direction above and below the real axis 
is given by the matrices $J,J^{-1}$, where $J:=\left(\begin{matrix} i & 
0 \\ 1 & i\end{matrix}\right)$, except $j=m+1$ when we get $-J,-J^{-1}$. 

Moreover, we have
\begin{equation}\label{matrixeq0}
B_j=\left(\begin{matrix} 1 & b_j \\ 0 & -a_j\end{matrix}\right),\ j\ne 0;\ B_0=\left(\begin{matrix} \lambda & b_0 \\ 0 & -a_0\end{matrix}\right).
\end{equation} 
It is clear that these matrices don't change under rescaling of $f$. 

Finally, we have two matrix equations
\begin{equation} \label{matrixeq} 
\prod_{j=0}^{m+1} B_jJ=-1,\ \prod_{j=0}^{m+1}B_jJ^{-1}=-1
\end{equation} 
(the monodromy around $\Bbb R\pm i\varepsilon$ is trivial).  
Taking the determinant, this gives 
\begin{equation}\label{matrixeq1}
\lambda=\prod_{j=0}^{m+1} a_j^{-1}, 
\end{equation} 

Conversely, a collection of matrices $B=(B_i)$ satisfying 
\eqref{matrixeq0}, \eqref{matrixeq}, \eqref{matrixeq1}
defines a local system $\nabla_B$ on $\Bbb C\Bbb P^1\setminus \lbrace t_0,...,t_{m+1}\rbrace $
as well as a vector $f_B=\begin{pmatrix} 1\\ 0\end{pmatrix}$. 

Now let $\widehat {\rm Loc}_m$ be the variety of nondegenerate pairs 
$(\nabla,f)$ (up to scaling of $f$) where $\nabla\in {\rm Loc}_m$. We have seen 
that the map $\pi: \widehat {\rm Loc}_m\to {\rm Loc}_m$ given by 
$(\nabla,f)\mapsto\nabla$ has degree $2$, with at most $1$-dimensional fibers which occur in codimension $\ge 2$, for scalar $\Bbb B$. Since ${\rm Loc}_m$ is irreducible, it follows that so is $\widehat {\rm Loc}_m$. 

\begin{proposition} Let $Y_m$ be the variety of $m+2$-tuples 
of matrices $(B_0,...,B_{m+2})$ of the form 
\eqref{matrixeq0} which satisfy the equations 
\eqref{matrixeq},\eqref{matrixeq1} such that the local system 
$\nabla_B$ is irreducible. Then the assignment 
$(\nabla,f)\mapsto B=(B_0,...,B_{m+2})$ is an isomorphism 
$\eta: \widehat {\rm Loc}_m\to Y_m$. In particular, $Y_m$ 
is an irreducible variety of dimension $2(m-1)$. 
\end{proposition} 

\begin{proof} The inverse to $\eta$ is given by 
$\eta^{-1}(B)=(\nabla_B,f_B)$. 
\end{proof} 

\begin{definition} Let us say that $(\nabla,f)\in \widehat{\rm Loc}_m$ is a {\bf balanced pair} if $a_j=1$ for all $j$ (hence $\lambda=1$). In this case $f$ is said to be a {\bf balancing} of $\nabla$. 
\end{definition} 

In other words, $(\nabla,f)$ is a balanced pair if 
$$
B_j=\left(\begin{matrix} 1 & b_j \\ 0 & -1\end{matrix}\right).
$$

\begin{lemma}\label{equ} If $(\nabla,f)$ is a balanced pair then 
the two equations in \eqref{matrixeq} are equivalent to each other.
\end{lemma} 

\begin{proof} We have 
\begin{equation}\label{bj}
B_jJ^{\pm 1}=\left(\begin{matrix} b_j\pm i& \pm ib_j \\  -1 & \mp i\end{matrix}\right)
\end{equation} 
Thus, conjugating $B_jJ$ by $S:=\left(\begin{matrix} 1 & 2i \\ 0 & 1\end{matrix}\right)$, we get 
$$
SB_jJS^{-1}=\left(\begin{matrix} b_j-i & -ib_j \\  -1 & i\end{matrix}\right)=B_jJ^{-1},
$$
which implies the statement. 
\end{proof} 

\begin{remark} Note that if $a_j=1$ and $b_j=0$ for all $j$ then 
equations \eqref{matrixeq} are not satisfied. Thus 
every collection $B=(B_0,...,B_{m+1})$ satisfying the equations
for a balanced pair automatically defines an {\it irreducible} local system
$\nabla_B$. 
\end{remark} 

A description of the spectrum of the Hecke operators in the real case results from the following proposition. 

\begin{proposition}\label{stron} The local system $\nabla$ of solutions of 
the oper equation $L(\bmu_k)\beta=0$ belongs to ${\rm Loc}_m(\Bbb C)$, admits a balancing by $f=\beta_k$ (the eigenvalue of the Hecke operators $H_x$) and has $b_j\in \Bbb R$.\footnote{Here, similarly to the complex case, $\bmu_k$ is the vector of eigenvalues of the Gaudin operators.}
\end{proposition} 

\begin{proof} First, since $\nabla$ is irreducible and has monordromies around $t_i$ 
conjugate to $\begin{pmatrix} -1 & 1\\ 0 & -1\end{pmatrix}$, it belongs to ${\rm Loc}_m(\Bbb C)$. 

Now fix $k$ and let $f_j^\pm$ be the restrictions of the eigenvalue $\beta_k$ to the positive and negative part of a neighborhood of $t_j$. By Propositions \ref{asym} and \ref{limexists}, we have
$\xi_j(f_j^+)=f_j^-$, so $f_j^\pm$ give rise to a well defined vector $f_j\in V_j$. 
We also have $B_jf_j=f_{j-1}$ for all $j$, so 
$\lambda=1$. Also by Proposition \ref{asym}, 
$$
f_j(x)\sim \mp |x-t_j|^{\frac{1}{2}}\log|x-t_j|,\ x\to t_j,\ j\ne m+1,\quad f_{m+1}(x)\sim |x|^{\frac{1}{2}}\log|x|,\ x\to \infty, 
$$
so setting 
$$
g_j:=f_j^+-if_j,\ j\ne m+1,\ g_{m+1}=-f_{m+1}^+-if_{m+1}
$$ 
as above, we get  
$$
g_j(x)\sim \pm \pi |x-t_j|^{\frac{1}{2}},\ x\to t_j,\ j\ne m+1;\quad g_{m+1}(x)\sim \pi|x|^{\frac{1}{2}},\ x\to \infty.
$$ 
This implies that the Wronskian $W(f_j(x),g_j(x))$ equals $\pi$ on the right of $t_j$ and $-\pi$ on the left of $t_j$, which implies that $\det B_j=-1$. Thus $a_j=1$. 

Finally, since the functions $f_j,g_j$ are real-valued, we have $b_i\in \Bbb R$.
\end{proof} 

Now let $\mathcal{B}$ be the {\bf set of balanced opers}, i.e., balanced pairs $(\nabla,f)$ such that $\nabla=\nabla_L$ is the local system of solutions of an oper $L$ with real coefficients. We have a surjective map $\mathcal{B}\to \mathcal{B}_*\subset \Bbb R^{m-1}$ sending 
$(\nabla(\bmu),f)$ to $\bmu$. 

\begin{proposition} Every point of $\mathcal B_*$ has at most two preimages in $\mathcal B$. 
\end{proposition} 

\begin{proof} If $(\nabla,f)\in \mathcal B$ is a preimage of $L\in \mathcal B_*$ 
and $\beta(x)$ the corresponding function on the circle then the leading coefficients of the asymptotics of $\beta(x)$ as $x\to t_i$ are $\pm 1$, while the leading coefficient at $\infty$ is $1$. So if $(\nabla^i,f^i)\in \mathcal B$, $i=1,2,3$ are three distinct preimages of $L$ and $\beta^i$ are the corresponding functions on the circle (thus also distinct) then the leading coefficients for $\beta^1-\beta^2,\beta^1-\beta^3$ are $0$ or $\pm 2$. But since $L$ has order $2$, these functions must be proportional (as they both have leading coefficient $0$ at $\infty$). So, since $\beta^2\ne \beta^3$, we must have 
$$
\beta^1-\beta^2=\beta^2-\beta^3
$$
as well as all transformations of this equation under permutations of indices.
This contradicts the fact that $\beta^i$ span a 2-dimensional space. 
\end{proof} 

\begin{remark} The map $\mathcal B\to \mathcal B_*$ is bijective if $m$ is odd. Indeed, in this case $\det(\Bbb B)=-1$, so $\Bbb B$ has eigenvalues $1,-1$ and $f$ is uniquely determined by $\nabla$ if exists. However, for even $m$  
in a non-generic situation (when $\Bbb B=1$), a fiber of the map $\mathcal B\to \mathcal B_*$ can consist of two points, see e.g. Subsection \ref{hyge}. 
\end{remark} 

\begin{remark} We expect that $\mathcal{B}_*$ is a discrete subset of $\Bbb R^{m-1}$. 
\end{remark} 

\begin{theorem}\label{specreal} The joint spectrum of the Hecke operators 
$H_x$ on $\mathcal H$ is a subset of $\mathcal{B}$.\footnote{Similarly to the complex case, eigenfunctions of the Hecke operators will also be eigenfunctions of the Gaudin operators. These eigenfunctions will be real analytic on the locus of very stable bundles, which is a disjoint union of regions separated by a system of walls comprised by real points of the wobbly divisor. The eigenfunctions can be described as collections of solutions of the quantum Hitchin (=Gaudin) equations in these regions satisfying certain gluing conditions along the walls. However, the precise nature of these gluing conditions is beyond the scope of this paper. We plan to discuss it in future publications.}
\end{theorem} 

\begin{proof} This follows from Proposition \ref{stron}. Namely, the function $\beta(x)$ 
is completely determined by the balanced oper $(\nabla,f)$ (namely, $\beta$ is proportional to $f$ and its scaling is determined by the asymptotics of $\beta$ at $\infty$). 
\end{proof} 

\subsection{The variety of balanced pairs and $T$-systems} 
Theorem \ref{specreal} shows that it is interesting to consider the affine scheme $X_m$ which is cut out inside $\Bbb A^{m+2}$ with coordinates $b_0,...,b_{m+1}$ by either of the two equations \eqref{matrixeq} (i.e., we assume that $a_j=1$ for all $j$). To describe 
this scheme, introduce the polynomial 
$P_r(b_1,...,b_r)$ which is the lower left corner entry of the matrix
$-B_0JB_1J...B_rJ$ (it is easy to see that it does not depend on $b_0$). For instance, 
$$
P_0=1,\ P_1(b_1)=b_1,\ P_2(b_1,b_2)=b_1b_2-1,\ P_3(b_1,b_2,b_3)=b_1b_2b_3-b_1-b_3, 
$$
and so on. It is easy to see that these polynomials are determined from the recursion
\begin{equation}\label{recur}
P_r(b_1,...,b_r)=P_{r-1}(b_1,...,b_{r-1})b_r-P_{r-2}(b_1,...,b_{r-2}). 
\end{equation} 
This shows, in particular, that $P_r$ have real (in fact, integer) coefficients, even though 
the matrix $J$ is not real. 

{\bf Remark.}\footnote{We thank M. Kapranov for pointing this out.} The polynomials $P_r$ are called {\it Euler's continuants}. They were studied by Euler in the context of continued fractions. Namely, consider the continued fraction 
$$
Q_r:=b_r-\frac{1}{b_{r-1}-\frac{1}{...-\frac{1}{b_1}}}.
$$
Then it is easy to see that 
$$
Q_r=\frac{P_r(b_1,...,b_r)}{P_{r-1}(b_1,...,b_{r-1})}.
$$
Euler showed that 
$$
P_r(b_1,...,b_r)=\sum_{S\subset [1,...,r]} (-1)^{|S|/2}\prod_{i\notin S}b_i,
$$
where $S$ runs over all subsets of $[1,...,r]$ which are disjoint unions of intervals of even length. 
He also showed that the number of terms in $P_r$ is the $r+1$-th Fibonacci number (which follows immediately from \eqref{recur}).   

\begin{proposition}\label{irredra} (i) The scheme $X_m$ is the irreducible rational hypersurface in $\Bbb A^m$ 
defined by the equation 
\begin{equation}\label{pm}
P_m(b_1,...,b_m)=1.
\end{equation} 
The values of $b_0$ and $b_{m+1}$ are recovered by the formulas 
\begin{equation}\label{pmm1}
b_0=P_{m-1}(b_2,...,b_m),\ b_{m+1}=P_{m-1}(b_1,...,b_{m-1}).
\end{equation} 

(ii) $X_m$ is smooth. 

(iii) $X_m$ has a stratification 
$$
X_m=U_{m-1}\sqcup U_{m-3}\times \Bbb A^1\sqcup U_{m-5}\times \Bbb A^2...
$$
where $U_r$ is the open subset of $\Bbb A^{r}$ defined by the condition $P_{r}(b_1,...,b_{r})\ne 0$.

(iv) $X_m(\Bbb R)$ is a smooth connected real manifold of dimension $m-1$ with stratification as in (iii). 
\end{proposition} 

\begin{proof} (i) Given $(b_0,...,b_{m+1})\in X_m$, 
equation \eqref{pm} is obtained from the identity 
$$
-B_0JB_1J...B_mJ=(B_{m+1}J)^{-1} 
$$
by comparing left lower corner entries. 
Equations \eqref{pmm1} are obtained similarly from the identities 
$$
-B_1J...B_mJ=(B_{m+1}JB_0J)^{-1}
$$ 
and 
$$
-B_0J...B_{m-1}J=(B_mJB_{m+1}J)^{-1}.
$$ 
Equation \eqref{recur} implies that the hypersurface 
$P_m=1$ is reduced and irreducible. Also it is rational since 
we can solve the equation $P_m=1$ for $b_m$. 
This implies (i), since 
it is clear that $\dim X_m=m-1$.

(ii) The proof is by induction in $m$. The base is trivial. To make the induction step, note that a singular point of $X_m$ would be a solution of the equations 
$P_m(b_1,...,b_m)=1$, $dP_m(b_1,...,b_m)=0$. 
So using \eqref{recur} we have 
$$
0=\frac{\partial P_m}{\partial b_m}(b_1,...,b_m)=P_{m-1}(b_1,...,b_{m-1}),
$$
 hence $P_{m-2}(b_1,...,b_{m-2})=1$. Also 
 $$
0= \frac{\partial P_m(b_1,...,b_m)}{\partial b_{m-1}}=\frac{\partial P_{m-1}(b_1,...,b_{m-1})}{\partial b_{m-1}}b_m=
 P_{m-2}(b_1,...,b_{m-2})b_m=b_m.
 $$
 Thus $b_m=0$. So for $i\le m-2$ we have 
 $$
 \frac{\partial P_m(b_1,...,b_m)}{\partial b_i}=\frac{\partial P_{m-2}(b_1,...,b_{m-2})}{\partial b_i}=0. 
$$
So $(b_1,...,b_{m-2})$ is a singular point of $X_{m-2}$. 
But by the induction assumption there are no such points. 
This is a contradiction which completes the induction step.  

(iii) Equation \eqref{recur} implies that we have a decomposition $X_m\cong U_{m-1}\sqcup X_{m-2}\times \Bbb A^1$, and the result follows by iteration. 

(iv) Follows from (i),(ii),(iii).  
\end{proof}

\begin{remark} Suppose all the $b_j$ are equal: $b_j=b$. Then equation 
\eqref{recur} takes the form 
\begin{equation}\label{recur1}
P_r(b)=bP_{r-1}(b)-P_{r-2}(b), 
\end{equation} 
with $P_0=1$ and $P_1=b$. Hence $P_r$ is the Chebyshev polynomial 
of the second kind encoding $SL_2$-characters, i.e., $P_r(2\cos x)=\frac{\sin (r+1)x}{\sin x}$. 
Thus equations \eqref{pm},\eqref{pmm1} look like 
$$
P_{m-1}(b)=b,\ P_m(b)=1. 
$$
So $P_{m+1}(b)=0$, i.e. $b=2\cos \frac{\pi k}{m+2}$ 
for $1\le k\le m+1$, $k$ odd (so we have $[\frac{m+2}{2}]$ solutions). 
Note that the solution for $k=1$, i.e., $b=2\cos \frac{\pi}{m+2}$, arises from 
a balanced oper for the leading eigenvalue of Hecke operator for 
the $\Bbb Z/(m+2)$-invariant configuration of points, see Subsection \ref{hyge}.  
\end{remark} 

One can interpret $X_m$ as the space of solutions of the {\bf $T$-system} 
of type $A_1$ of level $m$ (also known as Hirota-Miwa equations). Namely, 
recall (\cite{KNS}) that the $T$-system is the following system of equations for a function $T_i(k)$ 
of two integer variables $i,k$ with even $i+k$: 
$$
T_i(k-1)T_i(k+1)=T_{i-1}(k)T_{i+1}(k)+1.
$$
A solution of the $T$-system {\bf of level $m$} is a solution $T_i(k)$ defined for $0\le i\le m$ 
such that $T_0(k)=1,T_{m}(k)=1$ for all $k$. 

\begin{proposition} 
(i) If $(b_0,...,b_{m+1})\in X_m(\Bbb C)$ then the assignment 
$$
T_i(k):=P_i(b_{\frac{k-i}{2}},...,b_{\frac{k+i-2}{2}})
$$
defines a solution of the $T$-system of level $m$. 
Moreover, for generic $(b_j)$ this solution is nonvanishing.  

(ii) Any nonvanishing solution of the $T$-system of level $m$ is of this form for a unique 
$(b_0,...,b_{m+1})\in X_m(\Bbb C)$.  
\end{proposition} 

\begin{proof} (i) We prove that $T_i(k)$ satisfies the $T$-system by induction in $i$. 
The base case $i=1$ is easy, so let us perform the induction step from $i-1$ to $i$, with $i\ge 2$. 
Using \eqref{recur} and the induction assumption, we have 
$$
P_{i+1}(b_0,...,b_i)P_{i-1}(b_1,...,b_{i-1})=
$$
$$
P_i(b_0,...,b_{i-1})b_iP_{i-1}(b_1,...,b_{i-1})-
P_{i-1}(b_0,...,b_{i-2})P_{i-1}(b_1,...,b_{i-1})=
$$
$$
P_i(b_0,...,b_{i-1})b_iP_{i-1}(b_1,...,b_{i-1})-P_i(b_0,...,b_{i-1})P_{i-2}(b_1,...,b_{i-2})-1=
$$
$$
P_i(b_0,...,b_{i-1})P_i(b_1,...,b_i)-1,
$$
which completes the induction step. The fact that 
$T_i(k)\ne 0$ for generic $(b_j)$ is obvious. 

(ii) Let $T_i(k)$ be a nonvanishing solution of the $T$-system of level $m$. The proof of (i) implies by induction in $i$ that 
$$
T_i(k):=P_i(b_{\frac{k-i}{2}},...,b_{\frac{k+i-2}{2}})
$$
where $b_i:=T_1(2i+1)$. In particular, this means that 
$$
P_m(b_r,...,b_{r+m-1})=1
$$ and 
$$
P_{m+1}(b_r,...,b_{r+m})=0
$$ 
for all $r$. Thus by \eqref{recur}, equations \eqref{pmm1} hold. So 
$(b_0,...,b_{m+1})\in X_m(\Bbb C)$, as claimed. 
\end{proof} 

Let us say that a solution $T_i(k)$ of the $T$-system of level $m$ 
is {\bf half-periodic} with period $m+2$ if 
$$
T_{m-i}(k+m+2)=T_i(k)
$$
for all $i,k$ (clearly, such a solution is periodic with period $2(m+2)$). Since 
$$
-B_0J...B_{i-1}J=(B_iJ...B_{m+1}J)^{-1},
$$ 
the solution $T_i(k)$ of the $T$-system obtained from a point of $X_m$ 
is half-periodic with period $m+2$. In particular, we see that any nonvanishing solution is half-periodic with period $m+2$, 
which is the well known ``Zamolodchikov conjecture" (now a theorem, see 
\cite{FS,GT}). 

We also obtain 

\begin{corollary} $X_m$ is the closure of the variety of nonvanishing solutions 
of the $T$-system of level $m$. 
\end{corollary} 

\begin{example} 1. If $m=1$ then there is only one solution of the $T$-system of level $m$, 
$T_i(k)=1$. This corresponds to the fact that $X_1\subset \Bbb A^1$ is a point, defined by the equation $b_1=1$. So we have $b_j=1$ for all $j$. This reproduces the monodromy representation of the Picard-Fuchs equation, which agrees with Example \ref{ellint2}.\footnote{Note that this representation is real since it can be conjugated into an index $12$ congruence subgroup of $SL(2,\Bbb Z)$. However, for $m>1$ the corresponding representations are not real, in general.} 

2. If $m=2$ then the general solution is defined by the formula $T_2(2j)=b$ if $j$ is even and $T_2(2j)=2/b$ if $j$ is odd, for some nonzero 
number $b$. So we have $b_{2r}=b$ and $b_{2r+1}=2/b$. This corresponds to the fact that $X_2\subset \Bbb A^2$ is the hyperbola defined by the equation $b_1b_2=2$. 

3. If $m=3$ then we have $T_2(2j)=b_j$ for some numbers $b_j$, $j\in \Bbb Z$, and 
$T_3(2j+1)=b_jb_{j+1}-1$. So half-periodic solutions 
correspond to collections of numbers $b_j$, $j\in \Bbb Z/5$ 
such that 
$$
b_{j+2}=b_{j-1}b_j-1. 
$$
The space of such solutions is the surface $X_3\subset \Bbb A^3$ 
defined by the equation $b_1b_2b_3-b_1-b_3=1$. 

So if $b_1=b,b_2=c$ and $bc\ne 1$ then we have  
$$
b_1=b,\ b_2=c,\ b_3=\frac{1+b}{bc-1},\ b_4=bc-1,\ b_0=\frac{1+c}{bc-1}. 
$$ 
If $bc=1$ then we must have $b=c=-1$ and we have 
$$
b_1=-1,\ b_2=-1,\ b_3=d,\ b_4=0,\ b_0=-d-1
$$ 
for some number $d$. In other words, $-b_i$ form a 5-cycle occurring in the 5-term relation for the dilogarithm, see \cite{Z}, Subsection II.2. 
\end{example}

\subsection{A geometric description of balanced pairs} 
Let 
$$
\bold Q:=SL_2(\Bbb C)/SL_2(\Bbb R)=PSL_2(\Bbb C)/PSL_2(\Bbb R).
$$ 
The following lemma is well known. 

\begin{lemma} $\bold Q$ can be naturally identified with the set $\bold H$ of $A\in SL_2(\Bbb C)$ such that $\overline{A}A=1$, 
by $T\in \bold Q\mapsto \overline{T}T^{-1}$.
This identification is $SL_2(\Bbb C)$-equivariant, 
where $SL_2(\Bbb C)$ acts on $\bold H$ by $g\circ A:=\overline{g} Ag^{-1}$. 
\end{lemma}

\begin{proof} 
 Let $A=\left(\begin{matrix} a& ib\\ ic& d\end{matrix}\right)\in SL_2(\Bbb C)$. Then the equation $\overline{A}A=1$ reduces to the equations
$$
\overline{d}=a,\ \overline{b}=b,\ \overline c=c,\ ad+bc=1. 
$$
This yields 
$$
a_1^2+a_2^2+bc=1
$$
 (where $a_1+ia_2=a$), which defines a one-sheeted hyperboloid $\bold H$ of signature $(3,1)$ in the space $\Bbb R^4$ with coordinates $a_1,a_2,b,c$. The Lorentz group $SO(3,1)=PSL_2(\Bbb C)$ thus acts transitively on $\bold H$, and the stabilizer of the point ${\rm Id}=(1,0,0,0)\in \bold H$ is $SO(2,1)=PSL_2(\Bbb R)$. Thus we get an isomorphism $\xi: \bold Q=SO(3,1)/SO(2,1)\to \bold H$, and it is easy to see that it is given precisely by $\xi(T)= \overline{T}T^{-1}$ and is equivariant. 
\end{proof} 

Now let $N\ge 3$ and $(A_0,...,A_{N-1})\in \bold Q^N$. 
Let $J_j:=A_j^{-1}A_{j-1}$,  $j\in \Bbb Z/N$. 
Let $\bold Q^{N}_{\rm irr}\subset \bold Q^{N}$ be the set 
of points such that the collection of operators $\lbrace J_j\rbrace$ is irreducible. 
Let $\M_N:=\bold Q^{N}_{\rm irr}/PSL_2(\Bbb C)$, where $PSL_2(\Bbb C)$
acts diagonally by 
$$
g\circ (A_0,...,A_{N-1}):=(\overline{g}A_0g^{-1},...,\overline{g}A_{N-1}g^{-1}). 
$$
This action is free by Schur's lemma and the irreducibility of $\lbrace J_j\rbrace$, so 
$\M_N$ is the real locus of a real smooth algebraic variety of dimension $3(N-2)$. 

Let $\M_N^0\subset \M_N$ be the real algebraic subset of tuples 
$(A_0,...,A_{N-1})\in \M_N$ such that for each $j$ the matrix $A_j^{-1}A_{j-1}$  
is conjugate to $\begin{pmatrix} -1 & 1\\ 0 & -1\end{pmatrix}$. It follows from Lemma \ref{irreduci} that the closure of $\M_N^0$ is an irreducible complete intersection inside 
$\M_N$ cut out by $N$ equations ${\rm Tr}(A_j^{-1}A_{j-1})=-2$, so 
$$
\dim \M_N^0=3(N-2)-N=2(N-3).
$$ 

Let $t_j\in \Bbb R\Bbb P^1$, $j=0,...,N-1$ be distinct points occurring in the given order. Let $X^\circ:=\Bbb C\Bbb P^1\setminus \lbrace{t_0,...,t_{N-1}\rbrace}$. 

\begin{definition} A local system on $X^\circ$ equivariant under complex conjugation is a local system $E$ on $X^\circ$ equipped with an isomorphism $g: E\to \overline E$ such that $\overline{g}\circ g=1$.
\end{definition}  

\begin{lemma} The real algebraic set $\M_N^0$ is isomorphic to the moduli space 
of irreducible $SL_2$ local systems on $X^\circ$ equivariant under complex conjugation and having monodromy around $t_j$ conjugate to a Jordan 
block with eigenvalue $-1$. 
\end{lemma} 

\begin{proof} To prove the lemma, we will realize local systems (i.e., locally constant sheaves) on $X^\circ$ as representations of its fundamental groupoid. 

Namely, pick a base point $p$ on the upper half-plane, and let $\gamma_i$ be the path from $p$ to $\overline p$ passing between 
the points $t_{j-1}$ and $t_j$. Then given $A=(A_0,...,A_{N-1})\in \M_N^0$, we can define a local system $\rho=\rho_A$ on $X^\circ$ by the formula $\rho(\gamma_j)=A_j$, giving a representation of the fundamental groupoid $\pi_1(X^\circ,p,\overline p)$. This local system is equivariant under complex conjugation since $\overline{A_j}A_j=1$. Then for the closed paths 
$\delta_i:=\gamma_j^{-1}\gamma_{j-1}$ beginning and ending at $p$ 
we have $\rho(\delta_i)=A_j^{-1}A_{j-1}$, defining a representation of the fundamental group 
$\pi_1(X^\circ,p)$. Conversely, the same formula defines a point $A=A_\rho\in \M_N^0$ from an equivariant local system $\rho$ on $X^\circ$. 
\end{proof} 

Now given $A=(A_0,...,A_{N-1})\in \M_N^0$, let $T_j:=\xi^{-1}(A_j)$, $C_j:=T_{j-1}^{-1}T_j$, and $B_j:={\rm Im}C_j$ be the imaginary part of $C_j$. 
Then the tuple $(B_0,...,B_{N-1})$ gives a well defined real representation $V(A)=\oplus_{j=0}^{N-1}V_j(A)$ of the cyclic quiver with $N$ vertices with dimension vector $(2,...,2)$ and $B_j: V_j(A)\to V_{j-1}(A)$. Also since the eigenvalues of $\overline{C}_j^{-1}C_j$ are $-1$, we have $\overline{C}_j^{-1}C_j=-1+iE_j$, where the map $E_j: V_j(A)\to V_j(A)$ is real and $E_j^2=0$. 

Let $B=\oplus_{j=0}^{N-1}B_j, E=\oplus_{j=0}^{N-1}E_j$ be endomorphisms of 
$V(A)$. 

\begin{proposition} Balancings of the local system defined by the representation $\rho_A$ for $A\in \M_N$ correspond to subrepresentations $L\subset V(A)$ with dimension vector $(1,...,1)$ which are invariant under the operator $BE+EB$. Thus elements of $\M_N^0$ which admit a balancing form a connected irreducible real algebraic subset in $\M_N^{\rm bal}\subset \M_N^0$ of dimension $N-3$ (and every local system admits at most two balancings). 
\end{proposition}

\begin{proof} The first statement is proved by a direct calculation. 
The rest follows from Proposition \ref{irredra}.
\end{proof} 

\begin{remark} We see that the natural map $X_{N-2}(\Bbb R)\to \M_N^{\rm bal}$ 
is a normalization map (which is an isomorphism for odd $N$ but not for even $N$). 
For example, as we have seen, the variety $X_2$ is isomorphic to $\Bbb A^1\setminus 0$ via $b\mapsto (b,2/b)$. One can show that 
the local system $\nabla_b$ attached to $b$ determines $b$ 
uniquely except when $b=\sqrt{2}$, in which case $b$ and 
$-b$ give rise to the same local system, and that 
$\M_4^{\rm bal}$ is the curve in $\Bbb R^2$ defined by the equation $v^2-u(u-2)^2$, a punctured affine 
nodal cubic. The natural map $X_2(\Bbb R)
\to \M_4^{\rm bal}$ is then the normalization map given by $b\mapsto (b^2,b(b^2-2))$. This example will be revisited in 
Subsection \ref{archim1}. 
\end{remark} 

\subsection{Hypergeometric opers}\label{hyge} 

In general, the solutions of the oper equation appearing in the above formula for the eigenvalues of Hecke operators are not expected to be explicitly computable.
However, if the configuration of points $t_i$ is very symmetric then 
the leading eigenvalue $\beta_0(x)$ of the Hecke operators $H_x$ may be expressed via the  hypergeometric function. Let us describe this situation 
in more detail. 

First consider $F=\Bbb C$ and $t_j=e^{\frac{2\pi ij}{N}}$, where $N=m+2$. Then the curve 
$X=\Bbb P^1$ with marked points $t_j$ has a symmetry group $\Gamma=\Bbb Z/N$, 
and the quotient $X(\Bbb C)/\Gamma$ is the orbifold $\Bbb C\Bbb P^1$ with two orbifold points $0,\infty$ 
with stabilizer $\Bbb Z/N$ and one marked point $1$. Thus the oper corresponding to the leading 
eigenvalue $\beta_0(x)$ of $H_x$ (i.e., the uniformization oper) is invariant 
under $\Gamma$ and reduces to a hypergeometric connection on $X(\Bbb C)/\Gamma=\Bbb C\Bbb P^1$ 
with singularities at $0,1,\infty$. Explicitly, it is easy to compute that the corresponding oper equation for $\beta_0$ has the form 
$$
\left(\partial^2+\frac{N^2}{4}\frac{x^{N-2}}{(x^N-1)^2}\right)\beta_0=0. 
$$
This gives $\beta_0(x)=|x|\gamma(x^N)$, where 
$\gamma(y)$ satisfies the equation 
\begin{equation}\label{gammaeq} 
\left((y\partial)^2-\frac{1}{4N^2}+\frac{y}{4(y-1)^2}\right)\gamma=0,
\end{equation} 
which reduces to a hypergeometric equation. Namely, we have 
\begin{equation}\label{betaform}
\beta_0(x)=C|1-x^N|\left(|F_-(x^N)|^2-\lambda^2 |F_+(x^N)|^2\right), 
\end{equation} 
where 
$$
F_-(y):=F(\tfrac{1}{2},\tfrac{1}{2}-\tfrac{1}{N},1-\tfrac{1}{N};y), \
F_+(y):=y^{\frac{1}{N}}F(\tfrac{1}{2},\tfrac{1}{2}+\tfrac{1}{N},1+\tfrac{1}{N};y)
$$
are the basic solutions of the Euler hypergeometric equation 
$$
y(1-y)F''+(c-(a+b+1)y)F'-abF=0
$$
with parameters 
$$
a=\tfrac{1}{2},\ b=\tfrac{1}{2}-\tfrac{1}{N},\ c=1-\tfrac{1}{N}.
$$ 
Namely, $F$ is the hypergeometric function 
$$
F(x)={}_2F_{1}(x)=\sum_{n\ge 0}\frac{(a)_n(b)_n}{(c)_n n!}x^n.
$$
Note that the function \eqref{betaform} is real analytic at $x=0$ (for any $C,\lambda$). 

It remains to determine the constants $C$ and $\lambda$. The constant $\lambda$ is determined from the condition that the function $\beta_0$ is single-valued. Namely, using the transformation formula for $F$ from $0$ to 
$\infty$, a direct calculation yields
$$
\lambda=\frac{\Gamma\left(\frac{1}{2}+\frac{1}{N}\right)}{\Gamma\left(\frac{1}{2}-\frac{1}{N}\right)}.
$$
Finally, to determine $C$, consider the asymptotics of $\beta_0(x)$ near $x=1$. 
By definition, we should have 
$$
\beta_0(x)\sim |x-1|\log(|x-1|^{-2}),\ x\to \infty.
$$
Using again the transformation formulas for $F$ (this time from $0$ to $1$), 
from this we obtain after a calculation: 
$$
C=\frac{\Gamma\left(\frac{1}{2}-\frac{1}{N}\right)\Gamma\left(1+\frac{1}{N}\right)}{\Gamma\left(\frac{1}{2}+\frac{1}{N}\right)\Gamma\left(1-\frac{1}{N}\right)}.
$$  
Thus, we obtain 

\begin{proposition}
For $F=\Bbb C$ we have 
$$
\beta_0(x)=
$$
$$
\frac{2|1-x^N|\left(\Gamma\left(\frac{1}{2}-\frac{1}{N}\right)^2|F(\frac{1}{2},\frac{1}{2}-\frac{1}{N},1-\frac{1}{N};x^N)|^2-\Gamma\left(\frac{1}{2}+\frac{1}{N}\right)^2|xF(\frac{1}{2},\frac{1}{2}+\frac{1}{N},1+\frac{1}{N};x^N)|^2\right)}{\Gamma(\frac{1}{N})\Gamma(1+\frac{1}{N})\sin \frac{2\pi}{N}}.
$$
\end{proposition}

\begin{remark} Similar analysis can be carried out 
for the configurations $t_0=0$, $t_j=e^{\frac{2\pi ij}{m+1}}$, $1\le j\le m+1$ 
with $\Bbb Z/(m+1)$-symmetry and $t_0=0$, $t_j=e^{\frac{2\pi ij}{m}}$, $1\le j\le m$, $t_{m+1}=\infty$ with $\Bbb Z/m$-symmetry, giving a hypergeometric formula for the leading eigenvalue. 
\end{remark}

Now consider the case $F=\Bbb R$. Fix the real structure on $\Bbb P^1$ given by $x^*=\overline{x}^{-1}$. Then the real locus is the unit circle $|x|=1$ (with upper half plane $|x|<1$) and the same configuration of 
the points $t_j$ has a symmetry group $\Gamma=\Bbb Z/N$. It is clear that the function $\beta_0$ is invariant under this group, so we can regard 
it as a function of the angle $\theta\in (-\frac{\pi}{2},\frac{\pi}{2})$ such that $x=e^{\frac{2i}{N}(\theta+\frac{\pi}{2})}$. So changing variables (remembering that the eigenvalue is a $-1/2$-form) we find that $\beta_0(\theta)$ is an even solution of the equation
$$
\left(\partial_\theta^2+\frac{1}{N^2}+\frac{1}{4\cos^2\theta}\right)\beta_0=0, 
$$
which is equation \eqref{gammaeq} with $y=-e^{2i\theta}$. 
So we get 
$$
\beta_0(\theta)=C\sqrt{2\cos \theta}\cdot {\rm Re}\left(e^{i\theta(\frac{1}{2}-\frac{1}{N})}F(\tfrac{1}{2},\tfrac{1}{2}-\tfrac{1}{N},1-\tfrac{1}{N};-e^{2i\theta})\right).
$$
The constant $C$ can be found by looking at the asymptotics at $\theta=\pm\pi/2$. 
This yields 
$$
C=\frac{\sqrt{\frac{\pi}{2}} \Gamma(\frac{1}{2}-\frac{1}{N})}{\Gamma(1-\frac{1}{N})\cos\frac{\pi}{2}(\frac{1}{2}-\frac{1}{N})}.
$$
Thus we get 

\begin{proposition} The leading eigenvalue $\beta_0$ of the Hecke operators for $F=\Bbb R$ 
in the $\Bbb Z/N$-symmetric case is given by the formula 
$$
\beta_0(\theta)=\frac{\sqrt{\pi} \Gamma(\frac{1}{2}-\frac{1}{N})}{\Gamma(1-\frac{1}{N})\cos\frac{\pi}{2}(\frac{1}{2}-\frac{1}{N})}\sqrt{\cos \theta}\cdot {\rm Re}\left(e^{i\theta(\frac{1}{2}-\frac{1}{N})}F(\tfrac{1}{2},\tfrac{1}{2}-\tfrac{1}{N},1-\tfrac{1}{N};-e^{2i\theta})\right)
$$
extended periodically. Moreover, if $N$ is even then the same oper admits 
another balancing 
$$
\beta_1(\theta)=\frac{\sqrt{\pi} \Gamma(\frac{1}{2}-\frac{1}{N})}{\Gamma(1-\frac{1}{N})\cos\frac{\pi}{2}(\frac{1}{2}+\frac{1}{N})}\sqrt{\cos \theta}\cdot {\rm Im}\left(e^{i\theta(\frac{1}{2}-\frac{1}{N})}F(\tfrac{1}{2},\tfrac{1}{2}-\tfrac{1}{N},1-\tfrac{1}{N};-e^{2i\theta})\right)
$$
extended antiperiodically, so we have two eigenvalues of Hecke operators corresponding to the same oper. 
\end{proposition} 

Now consider the difference $h:=\beta_0-\beta_1$, which  
is a solution of the oper equation on $(-\frac{\pi}{2},\frac{\pi}{2})$ regular 
at $\theta=\pm \pi/2$ (defined for both even and odd $N$). Then analytic continuation 
from $\pi/2$ to $-\pi/2$ gives $\Bbb B\beta_0=\beta_0$, $\Bbb B\beta_1=-\beta_1$, 
so $\Bbb B h=-h+2\beta_0$. Moreover, the half-monodromy around $\pi/2$ is given by 
$Jh=ih$, $J\beta_j=i\beta_j+\lambda h$ for some constant $\lambda$. To find $\lambda$, 
let $c_\pm = \cos\frac{\pi}{2}(\frac{1}{2}\pm \frac{1}{N})$, and consider the function 
$$
\beta=c_-\beta_0+ic_+\beta_1=\frac{\sqrt{\pi} \Gamma(\tfrac{1}{2}-\tfrac{1}{N})}{\Gamma(1-\frac{1}{N})}\sqrt{\cos \theta}\cdot e^{i\theta(\tfrac{1}{2}-\tfrac{1}{N})}F\left(\tfrac{1}{2},\tfrac{1}{2}-\tfrac{1}{N},1-\tfrac{1}{N};-e^{2i\theta}\right).
$$ 
Then $J\Bbb B \beta=\zeta^2 \beta$, where $\zeta:=e^{\frac{\pi i}{2N}}$, so 
$$
J(c_-\beta_0-ic_+\beta_1)=\zeta^2(c_-\beta_0+ic_+\beta_1). 
$$ 
Thus 
$$
i(c_-\beta_0-ic_+\beta_1)+\lambda(c_--ic_+)(\beta_0-\beta_1)=\zeta^2(c_-\beta_0+ic_+\beta_1). 
$$ 
Note that $c_--ic_+=e^{-\frac{\pi i}{4}}\zeta$. 
So we get 
$$
ic_-+\lambda e^{-\frac{\pi i}{4}}\zeta=\zeta^2 c_-,\ c_+-\lambda e^{-\frac{\pi i}{4}}\zeta=i\zeta^2 c_+.
$$
It is easy to see that these equations are equivalent to each other, and yield
$$
\lambda=\frac{1}{2}(\zeta^2+\zeta^{-2})=\cos \frac{\pi}{N}.
$$
Thus setting $g=\lambda h$, $f=\beta_0$, we get 
$Jf=if+g$ and $\Bbb Bg=-g+2\lambda f=-g+2\cos\frac{\pi}{N}\cdot f$. 

This shows that the numbers $b_j$ attached to this balanced oper are all equal to 
$2\cos\frac{\pi}{N}$.

\section{The case of $X=\Bbb P^1$ with four parabolic points}\label{Four parabolic points} 

The goal of this section is to consider in more detail the special case $m=2$, i.e., $X=\Bbb P^1$ with four parabolic points, which we may assume to be $t_0=0,t_1=t,t_2=1,t_3=\infty$. In particular, we will provide more explicit versions and alternative proofs of some of the results of the previous section in this case. 

\subsection{The moduli space of stable bundles}

\begin{proposition}\label{bun0} The variety ${\it Bun}_0^\circ$ is isomorphic to $\Bbb P^1\setminus \lbrace{0,t,1,\infty\rbrace}$. 
\end{proposition} 

\begin{proof} The proof is well known but we give it for reader's convenience. 
Any bundle $E\in Bun_0^\circ$ has the form $E=O(r)\oplus O(-r)$ for some $r\ge 0$, and its parabolic degree is $4\cdot \frac{1}{2}=2$, so the parabolic slope is $1$. This implies that for a stable bundle we must have $r=0$, since otherwise $O(r)$ would have parabolic degree (hence slope) $\ge 1$ which is forbidden for stable bundles. 

Thus $E=O\oplus O$. So the parabolic structure on $E$ is defined by a choice of four lines $y_0,y_1,y_2,y_3\in \Bbb P^1$ at the four marked points. If a subbundle
 $O\subset O\oplus O$  contains $k$ of these four lines then its parabolic degree (=slope) is $k/2$, so for stable bundles we must have $k\le 1$. This means that all
 $y_i$ are distinct. Such quadruples modulo M\"obius transformations are parametrized by $\Bbb P^1\setminus \lbrace{0,1,\infty\rbrace}$, using the cross ratio $u=\frac{(y_0-y_1)(y_2-y_3)}{(y_0-y_2)(y_1-y_3)}$. 

However, the condition that $y_i\neq y_j$ for $i\neq j$ does not guarantee the 
stability. There are many embeddings $ L\cong O(-1)\ho O\oplus O$ as a subbundle  and if $L$ contains $k$ of the four lines, its parabolic degree will be $-1+\frac{k}{2}$. Thus we must have $k\le 3$.  So we have a single forbidden case $k=4$, i.e. the case when the (distinct!) parabolic lines at the four points are the fibers of $L$. This removes one more point (namely $t$) from the moduli space. Finally, note that if $n\ge 2$ then the parabolic degree $-n+\frac{k}{2}$ of any subbundle $O(-n)\subset O\oplus O$ is automatically $<1$ (as $k\le 4$). Thus we see that the stable locus is ${Bun}_0^\circ=\Bbb P^1\setminus \lbrace{0,t,1,\infty\rbrace}$. 
\end{proof} 

\begin{remark} By Proposition \ref{invo} and Remark \ref{invo1}(3), the Hecke modification at $\infty$ along the parabolic line defines a natural isomorphism between $Bun_0^\circ$ and $Bun_1^\circ$. It  is nevertheless instructive to see directly in this example that the sets of stable bundles of degrees $0$ and $1$ have exactly the same structure.\footnote{It is 
easy to show that the moduli spaces of {\it semistable} bundles 
$Bun_0\cong  Bun_1$ are isomorphic to $\Bbb P^1$, but we will not use these spaces.} 

We may realize $Bun_1^\circ$ as the space of stable rank 2 vector bundles 
of degree $1$. Such a bundle has the form $E=O(-r)\oplus O(r+1)$, where 
$r\ge 0$. 
The parabolic degree of such a bundle is $1+4\cdot\frac{1}{2}=3$, so the parabolic slope is $3/2$. Thus in the stable case we must have $r=0$, 
i.e., $E=O\oplus O(1)$ as a bundle.   

We realize $O(1)$ as we did in Subsection \ref{birpar}. The conditions of stability are that the (unique) subbundle of $E$ isomorphic to $O(1)$ contains none of the fixed lines (as its slope is $1$), while any subbundle isomorphic to $O$ contains at most two. The first condition means that the generating vectors of the four fixed lines over the marked points 
have a nonzero first coordinate, which we may assume to be $1$, and encode the vectors by the second coordinate. 

Now let us see what constraints are imposed by the second condition. It is easy to see that all  embeddings $O\ho  O\oplus O(1)
$ as a subbundle
form a single orbit under $\Bbb P{\rm Aut}(O\oplus O(1))=\Bbb G_m\ltimes \Bbb G_a^2$. Also, it 
is clear that for any $E$, some copy of $O$ in $O\oplus O(1)$ contains any given two of the four fixed lines. We may therefore assume that the standard copy $O\subset O\oplus O(1)$ contains the lines $y_1$ at $0$ and $y_4$ at $\infty$. Thus $y_1,y_4$ are both generated by the vector $(1,0)$. 
Now, the line $y_2$ at the point $1$ cannot be the same, so after rescaling it is spanned by a vector $(1,1)$. Finally, the line $y_3$ at the point $t$ cannot lie in the copy of $O$ passing through $y_1,y_4$, so it is spanned by a vector $(1,z)$ with $z\ne 0$. Also it can't lie in the copy of $O$ through $y_1,y_2$, which yields $z\ne t$. Finally, it can't lie in the copy of $O$ through $y_2,y_4$, which yields $z\ne 1$. 
Thus, $Bun_1^\circ=\Bbb P^1\setminus \lbrace{0,t,1,\infty\rbrace}$, i.e., is isomorphic to $Bun_0^{\circ}$. 
\end{remark} 

If we identify $Bun_0^\circ$ and $Bun_1^\circ$ using the involution $S_3$ 
then 
the involutions $S_i$ take the form 
$$
S_0(y)=\frac{t}{y},\  S_1(y)=\frac{t(y-1)}{y-t},\  S_2(y)=\frac{y-t}{y-1},\ S_3(y)=y.
$$ 
Thus they define an action of the Klein 4-group $(\Bbb Z/2)^2$, which acts transitively on the singular points $0,t,1,\infty$. Note that this is a 
special feature 
of the case of $4$ points: for $N\ge 5$ points we have a {\it faithful} action 
of the group $\Bbb V=(\Bbb Z/2)^N_0$ on $Bun_0^\circ$, while for 
$N=4$ the action of $\Bbb V=(\Bbb Z/2)^4_0=(\Bbb Z/2)^3$ factors 
through $(\Bbb Z/2)^2$, as $S_0S_1S_2=1$. 

\subsection{The Hecke correspondence and Hecke operators} \label{The Hecke correspondence and Hecke operators}

Let $E_{y,0}$ denote the bundle of degree $0$ corresponding to $y\in \Bbb 
P^1\setminus \lbrace{0,t,1,\infty\rbrace}$ 
and $E_{z,1}$ the bundle of degree $1$ corresponding to $z\in \Bbb P^1\setminus \lbrace{0,t,1,\infty\rbrace}$. 
Specializing Proposition \ref{heckcor} to the case $m=2$, we obtain 

\begin{proposition} We have $HM_{x,s}(E_{y,0})=E_{z,1}$, where 
$$
z(t,x,y,s)=\frac{(1-s)(xy-st)}{(x-s)(y-s)}.
$$
\end{proposition} 

This implies 

\begin{corollary} The modified Hecke operator has the form given by the specialization 
of \eqref{t24}: 
\begin{equation}\label{t24a}
(\Bbb H_x\psi)(y)=
\int_{F}\psi\left(\frac{(1-s)(xy-st)}{(x-s)(y-s)}\right)
\norm{\frac{ds}{(x-s)(y-s)}}. 
\end{equation}
\end{corollary}

It turns out that a more convenient coordinate than $s$ is $r:=\frac{(s-x)(t-x)}{(s-y)(t-y)}$. In this coordinate, the equation of the Hecke correspondence looks like 
$$
z=\frac{((x-1)(x-t)-(y-1)(y-t)r)(yr-x)}{r(y-x)^2}.
$$
(see \cite{udB}, Theorem 7.3). This gives rise to a quadratic equation for $r$ in terms of $t,x,y,z$:
$$
y(y-1)(y-t)r^2-(y(x-1)(x-t)+x(y-1)(y-t)-z(x-y)^2)r+x(x-1)(x-t)=0.
$$
The discriminant of this equation is $D:=((x-y)^2rz'(r))^2$, and 
$$
rz'(r)=\frac{x(x-1)(x-t)r^{-1}-y(y-1)(y-t)r}{(x-y)^2}. 
$$
Thus 
$$
(y(y-1)(y-t)r-x(x-1)(x-t)r^{-1})^2=
$$
$$
(y(y-1)(y-t)r+x(x-1)(x-t)r^{-1})^2-4x(x-1)(x-t)y(y-1)(y-t)=
$$
$$
((x-y)^2z-y(x-1)(x-t)-x(y-1)(y-t))^2-4x(x-1)(x-t)y(y-1)(y-t)=(x-y)^2f_t(x,y,z), 
$$
where 
$$
f_{t}(x,y,z):=(t+xy-zx-zy)^2-4(z-t)(z-1)xy=(xy+xz+yz-t)^2+4(1+t-x-y-z)xyz
$$
is the polynomial considered by Kontsevich in 
\cite{K}, p.3. We also see that the degree 2 map $r\mapsto z$ encoding the projection 
$$
p_{1y}: p_0^{-1}(y)\subset {\mathcal H}_x\to Bun_1^\circ 
$$
branches at the points 
$r=r_\pm:=\pm \sqrt{\frac{x(x-1)(x-t)}{y(y-1)(y-t)}}$ (zeros of $D$). 
This means 
that 
$$
p_{1y}^*(K)\cong K\otimes O(r_+)^{-1}\otimes O(r_-)^{-1},
$$
where 
$K$ is the canonical bundle on $\Bbb P^1$. But we have an isomorphism $O(r_+)\otimes O(r_-)\cong K^{-1}$, so 
we get $p_{1y}^*(K)=K^2$. Thus $p_{1y}^*(\norm{K}^{1/2})=\norm{K}$, 
the bundle of densities. This means that integration of $p_{1y}^*\psi$ will 
be well defined as soon as we choose an identification 
$$
O(r_+)\otimes O(r_-)\cong K^{-1}
$$ 
up to a phase factor. 
Such an identification is determined by a 1-form on $\Bbb P^1$ (with coordinate $r$) which has simple poles 
at $r=r_\pm$ and no other singularities. Such a form, up to scaling depending on $x,y$, 
is $\frac{dr}{(r-r_+)(r-r_-)}$, and the correct scaling
turns out to be 
$$
\omega=\frac{1}{y(y-1)(y-t)}\frac{dr}{(r-r_+)(r-r_-)}=\frac{dr}{y(y-1)(y-t)r^2-x(x-1)(x-t)}. 
$$  
This implies that the modified Hecke operator is given by the formula
\begin{equation}\label{heck}
(\Bbb H_x\psi)(y)=2\int_{F}\psi(z)\theta(f_t(x,y,z))\norm{\frac{dr}{r(x-y)}},
\end{equation}
where $\theta(a)=1$ if $a\in F$ is a square and $\theta(a)=0$ otherwise, and we view functions
on $\Bbb P^1(F)$ as half-densities by using the map $\psi(z)\mapsto \psi(z)\norm{dz}^{\frac{1}{2}}$. 
(Here the factor $2$ appears because the map $r\mapsto z$ has degree $2$). 

Hence we have 
$$
(\Bbb H_x\psi)(y)=
2\int_{F}\psi(z)\theta(f_t(x,y,z))\norm{\frac{dz}{rz'(r)(x-y)}}=
$$
$$
2\int_{F}\psi(z)\theta(f_t(x,y,z))\norm{\frac{(x-y)dz}{y(y-1)(y-t)r-x(x-1)(x-t)r^{-1}}}.
$$
Thus we obtain

\begin{proposition}\label{p1} The modified Hecke operator $\Bbb H_x$ is given by the formula
$$
(\Bbb H_x\psi)(y)=2\int_{F}\psi(z)\frac{\theta(f_t(x,y,z))\norm{dz}}{\sqrt{\norm{f_t(x,y,z)}}}.
$$
\end{proposition} 

For example, for $F=\Bbb C$ we have $\theta=1$ so we just get 
$$
(\Bbb H_x\psi)(y)=\frac{2}{\pi}\int_{\Bbb C}\psi(z)\frac{dzd\overline z}{|f_t(x,y,z)|}.
$$
Note also that $f_t(x,y,z)$ is symmetric in $x,z$, so the operator $\Bbb H_x$ is manifestly symmetric. 

\subsection{Boundedness and compactness of $\Bbb H_x$}\label{Boundedness and compactness}
We already know that the operator $\Bbb H_x$ is bounded and moreover compact on $\mathcal{H}$ (Proposition \ref{bou}, Proposition \ref{compa}). In 
this subsection we provide alternative proofs of these facts in the case of four points. 

We start with boundedness. Let $\phi$ be a positive half-density 
on $\Bbb P^1(F)\setminus\lbrace{0,t,1,\infty\rbrace}$ with logarithmic singularities at 
$0,1,t,\infty$, i.e. $\phi=\phi(w)\norm{dw}^{1/2}$ in a local coordinate near each of these points, 
with $\phi(w)\sim {\rm const}\cdot \log\norm{w^{-1}}$ as $w\to 0$ (note that such a half-density is automatically in $L^2$). 
Note that for a fixed $y$ the function $f_t(x,y,z)$ is quadratic in $z$ with simple zeros, 
which collide into a double zero when $y=0,t,1,\infty$.\footnote{An exception is $y=x$, where $f_t(x,y,z)$ is linear in $z$.} Therefore, Proposition \ref{p1} and 
Lemma \ref{l2} imply that there exists $C>0$ such that 
$(\Bbb H_x\phi)(y) \le C\phi(y)$
for all $y\in \Bbb P^1(F)\setminus \lbrace{0,t,1,\infty\rbrace}$. Thus by Schur's test (\cite{HS}, Theorem 5.2), $\Bbb H_x$ is bounded with $\norm{\Bbb H_x}\le C$. 

Thus, $\Bbb H_x$ is a bounded self-adjoint operator on $\mathcal H$. 

Now let us establish the compactness of $\Bbb H_x$. The compactness would 
follow from $\Bbb H_x$ 
being trace class. At first sight this appears possible since the Schwartz kernel of 
$\Bbb H_x$ is locally $L^1$ (this is a special feature of the case of $4$ points). However, it turns out that 
$\Bbb H_x$ is {\bf not} trace class, nor even Hilbert-Schmidt, since 
$$
{\rm Tr}(\Bbb H_x^2)=4\int_{F^2}\theta(f_t(x,y,z))\norm{\frac{dydz}{f_t(x,y,z)}}=\infty;
$$
namely, the integral logarithmically diverges at the divisor of zeros of the polynomial $f_t(x,y,z)$ (for fixed $t,x$).\footnote{It is easy to show that the curve $f_t(x,y,z)=0$ always has points over $F$. Indeed, the discriminant of $f_t(x,y,z)$ as a polynomial of $z$ is $16x(x-1)(x-t)y(y-1)(y-t)$, and this is a square if $y$ is close to $x$.} Nevertheless, we have the following result. 

\begin{proposition}\label{comp} The operator $\Bbb H_x^2$ is Hilbert-Schmidt. Hence the operator $\Bbb H_x$ is compact. 
\end{proposition} 

\begin{proof} 
We have 
$$
(\Bbb H_x^2\psi)(u)=4\int_{F^2}\psi(z)\theta(f_t(x,y,z))\theta(f_t(x,y,u))\frac{\norm{dydz}}{\sqrt{\norm{f_t(x,y,z)f_t(x,y,u)}}}. 
$$
Thus the  Schwartz kernel of $\Bbb H_x^2$ is 
$$
K(u,z):=4\int_{F}\theta(f_t(x,y,z))\theta(f_t(x,y,u))\frac{\norm{dy}}{\sqrt{\norm{f_t(x,y,z)f_t(x,y,u)}}}.
$$
It follows from Lemma \ref{l2} that $K(u,z)\le K_0\log\norm{\frac{1}{u-z}}$ 
near the diagonal $u=z$ for some $K_0>0$. 
This implies that 
$$
{\rm Tr}(\Bbb H_x^4)=\int_{F^2}K(u,z)^2\norm{dudz}<\infty, 
$$
hence the proposition. 
\end{proof} 

\begin{remark} In fact, since the Schwartz kernel of $\Bbb H_x^2$ has a logarithmic singularity on the diagonal, 
we have ${\rm Tr}|\Bbb H_x|^{2+\varepsilon}<\infty$ for any $\varepsilon>0$. 
\end{remark} 

We also recover Proposition \ref{asym} on the asymptotics of Hecke operators near the parabolic points:
$$
H_x\sim \norm{x}^{\frac{1}{2}}\log\norm{x},\ x\to \infty;\quad H_x\sim \norm{x}^{\frac{1}{2}}\log\norm{\tfrac{1}{x}}S_0,\ x\to 0;
$$
$$
\quad H_x\sim \norm{x-t}^{\frac{1}{2}}\log\norm{\tfrac{1}{x-t}}S_1,\ x\to t;\ H_x\sim \norm{x-1}^{\frac{1}{2}}\log\norm{\tfrac{1}{x-1}}S_2,\ x\to 1.
$$
Indeed, this follows from Proposition \ref{p1} and the formulas 
$$
f_t(0,y,z)=(t-yz)^2,\ f_t(1,y,z)=(y+z-t-yz)^2,\ f_t(t,y,z)=(ty+tz-t-yz)^2. 
$$
Namely, these formulas show that for $i=0,1,2$ we have $f_t(t_i,y,z)=h_i(y,z)^2$, where the graph of $S_i$ is defined by the equation $h_i(y,z)=0$. 

\subsection{The spectral decomposition} 
By the spectral theorem for compact self-adjoint operators, 
the commuting operators $\Bbb H_x$ have a common 
eigenbasis $\psi_n(y)$, $n\ge 0$ of $\mathcal H$, with $\norm{\psi_n}=1$ and 
\begin{equation}\label{eigenv}
\Bbb H_x\psi_n=\widetilde\beta_n(x)\psi_n
\end{equation} 
for real-valued functions $\widetilde\beta_n(x)=\norm{x(x-1)(x-t)}^{-1/2}\beta_n(x)$, none of them identically zero. Thus all joint eigenspaces of $\Bbb H_x$ are finite dimensional. Equation \eqref{eigenv} implies that $\psi_n$ are smooth outside $0,t,1,\infty$ (choosing $x$ such that $\beta_n(x)\ne 0$ and writing 
$\psi_n=\frac{H_x\psi_n}{\beta_n(x)}$). Moreover, $\psi_n$ can 
be chosen real-valued since $\Bbb H_x$ has a real-valued Schwartz kernel. 
Finally, we can pick the sign of $\psi_n$ so that it is positive near $\infty$, which fixes $\psi_n$ uniquely. 

\begin{corollary}\label{reprodker} We have $\widetilde\beta_n(x)=c_n\psi_n(x)$ for some $c_n>0$. Thus
$$
\frac{2\theta(f_t(x,y,z))}{\sqrt{\norm{f_t(x,y,z)}}}=\sum_n c_n\psi_n(x)\psi_n(y)\psi_n(z).
$$
\end{corollary} 

\begin{proof} It follows from \eqref{eigenv} that 
$$
\frac{2\theta(f_t(x,y,z))}{\sqrt{\norm{f_t(x,y,z)}}}=\sum_{n\ge 0} \widetilde\beta_n(x)\psi_n(y)\psi_n(z).
$$
Since $f_t$ is symmetric, we have $\widetilde\beta_n(x)=c_n\psi_n(x)$, as claimed. 
\end{proof}

\begin{corollary}\label{multlaw} We have 
$$
\Bbb H_x\Bbb H_y=\int_{F}\frac{2\theta(f_t(x,y,z))}{\sqrt{\norm{f_t(x,y,z)}}}\Bbb H_z\norm{dz}.
$$
\end{corollary} 

\begin{proof}
Both sides act by the same eigenvalues on the basis $\psi_n$. 
\end{proof} 

\begin{corollary} We have $c_n=\psi_{n,\infty}^{-1}$, where 
$$
\psi_{n,\infty}:=\lim_{x\to \infty}\frac{\psi_n(x)}{\norm{x}^{-1}\log\norm{x}}.
$$ 
In other words, we have
\begin{equation}\label{kerfor}
\frac{2\theta(f_t(x,y,z))}{\sqrt{\norm{f_t(x,y,z)}}}=\sum_n\frac{\psi_n(x)\psi_n(y)\psi_n(z)}{\psi_{n,\infty}}. 
\end{equation} 
\end{corollary} 

\begin{proof} This follows from Proposition \ref{asym}. 
\end{proof} 

\begin{corollary} Let 
$$
Q:=\int_{F}\Bbb H_x^2\norm{dx}.
$$
Then $Q$ is a positive Hilbert-Schmidt operator such that $Q\psi_n=c_n^2\psi_n$. 
Moreover, 
$$
{\rm Tr}(Q^s)=\sum_n c_n^{2s}. 
$$
In particular, 
$$
\sum_n c_n^4<\infty.
$$
\end{corollary} 

\begin{corollary} In the case of $4$ points, the joint eigenspaces of the 
Hecke operators $H_x$ are 1-dimensional.
\end{corollary} 

\begin{proof} This follows from formula \eqref{kerfor}. 
\end{proof} 

This recovers Theorem 2 of \cite{K} as well as the package of properties from \cite{K}, p.3. 

\begin{remark} It is interesting to consider the trace of the modified Hecke operator
$$
{\rm Tr}(\Bbb H_x)=\sum_n \widetilde \beta_n(x).
$$ 
Unfortunately, as we have seen, this trace is not well defined 
since $\Bbb H_x$ is not trace class, i.e., ${\rm Tr}|\Bbb H_x|=\sum_n |\widetilde \beta_n(x)|=\infty$
(so the series $\sum_n \widetilde \beta_n(x)$ is not absolutely convergent and the sum depends on the order of summation). Thus we can only talk about this trace in the regularized sense, 
and have to choose a regularization procedure. Let us choose a ``geometric" regularization procedure, i.e. write the Schwartz kernel $K_t(x,y,z)=\frac{2\theta(f_t(x,y,z))}{\sqrt{\norm{f_t(x,y,z)}}}$ of $\Bbb H_x$ as a limit $\lim_{N\to \infty}K_{t,N}(x,y,z)$ (uniform on compact sets not containing the singularities of $K_t(x,y,z)$)  of a pointwise increasing sequence of positive continuous kernels (for example, we can set $K_{t,N}:={\rm min}(K_t,N)$). Then we can define the regularized trace of $\Bbb H_x$ by the formula 
$$
T(t,x)=\lim_{N\to \infty}\int_{F}K_{t,N}(x,z,z)\norm{dz}. 
$$
It is easy to see that 
$$
T(t,x)=\int_{F}K_t(x,z,z)\norm{dz}. 
$$  
We have 
$$
f_t(x,z,z)=(t-z^2)^2-4z(z-1)(z-t)x.
$$
So we get 
$$
T(t,x)=\int_{F} \frac{2\theta((t-z^2)^2-4z(z-1)(z-t)x)}{\sqrt{\norm{(t-z^2)^2-4z(z-1)(z-t)x}}}\norm{dz}. 
$$
The discriminant of the polynomial $P_{t,x}(z):=(t-z^2)^2-4z(z-1)(z-t)x$ in $z$ is 
$$ 
\Bbb D=2^{12}(x(x-1)t(t-1)(x-t))^2. 
$$
So the trace integral converges for $x\ne 0,1,t$, where it diverges logarithmically. 

It can be shown by a direct computation that the cross-ratio of the roots 
of the polynomial $(t-z^2)^2-4z(z-1)(z-t)x$ (whose Galois group is $\Bbb Z/2\times \Bbb Z/2$) is $\frac{t(1-x)}{x(1-t)}$. 
Thus $T(t,x)$ can be expressed in terms of the modified elliptic integral 
$E_+(\frac{t(1-x)}{x(1-t)})$ (see Subsection \ref{ellint}).
\end{remark} 

\subsection{The archimedian case} \label{archim} 

As we showed in the previous section, in the archimedian case the operators $\Bbb H_x$ commute with Gaudin 
(=quantum Hitchin) Hamiltonians (Proposition \ref{strocomm}) and also satisfy a second order ODE (an oper equation) with respect to $x$ (Proposition \ref{opereq}). In the special case of $4$ points, because of the symmetry of $f_t(x,y,z)$ with respect to permutations of $x,y,z$, these two (in general, completely different) types of equations turn out to be equivalent. Namely, both boil down to the 
following result. 

Let 
$$
\Bbb L_z=\partial_zz(z-1)(z-t)\partial_z +z
$$ 
be the {\bf Lam\'e operator} (=the quantum Hitchin Hamiltonian for $4$ points). 

\begin{proposition}\label{l1aa}  
For any $t,x$ we have 
$$
(\Bbb L_y-\Bbb L_z)\frac{1}{\sqrt{f_t(x,y,z)}}=0.
$$
\end{proposition} 

This follows from Proposition \ref{strocomm} or Proposition \ref{opereq} or by a (rather tedious) direct computation. 

In fact, an even stronger statement holds (and can be checked similarly): 
we have 
$$
(\Bbb L_y-\Bbb L_z)\frac{1}{\sqrt{\norm{f_t(x,y,z)}}}=0
$$
in the sense of distributions. So in the real case we have 
$$
(\Bbb L_y-\Bbb L_z)\frac{1}{\sqrt{|f_t(x,y,z)|}}=0
$$
and in the complex case 
$$
(\Bbb L_y-\Bbb L_z)\frac{1}{|f_t(x,y,z)|}=(\overline {\Bbb L}_y-\overline {\Bbb L}_z)\frac{1}{|f_t(x,y,z)|}=0. 
$$
Moreover, in the real case we also have 
$$
({\Bbb L}_y-{\Bbb L}_z)\frac{2\theta(f_t(x,y,z))}{\sqrt{f_t(x,y,z)}}=0
$$
as distributions. These statements are equivalent to the statement that the operator $\Bbb H_x$ commutes with $\Bbb L$ (and in the complex case also with $\overline{\Bbb L}$) and satisfies the oper equation. 

Thus in the archimedian case we see that 
$$
\Bbb L\psi_n=\Lambda_n\psi_n
$$
for certain eigenvalues $\Lambda_n$. So for $F=\Bbb C$ the operators $\Bbb L-\Lambda_n$ for various $n$ correspond exactly to the real opers, and thus the eigenvalues $\Lambda_n$ are distinct. So we have 
\begin{equation}\label{kerfor1}
\frac{2}{\pi |f_t(x,y,z)|}=\sum_n\frac{\psi_n(x)\psi_n(y)\psi_n(z)}{\psi_{n,\infty}}, 
\end{equation} 
where $\psi_n$ runs over single-valued eigenfunctions of $\Bbb L$ 
normalized to be positive near $\infty$ and have norm $1$. 

\subsection{The real case} \label{archim1}
On the other hand, for $F=\Bbb R$ the situation is more subtle. To explain what is going on, fix an oper $L(\bmu)$ that admits an eigenfunction, i.e., a solution $\psi$ of the equation $L(\bmu)\psi=0$ which near $\infty$ looks like 
$$
\psi(x)=|x|^{1/2}(C_3\log|x|+h_3(x))
$$ 
for a continuous function $h_3$, and at each $t_j$, $j=0,1,2$ looks like 
$$
\psi(x)=|x-t_j|^{1/2}(C_j\log|x-t_j|+h_j(x))
$$ 
for continuous $h_j$. One of the constants $C_j$ must be nonzero, so without 
loss of generality we may assume that $C_3\ne 0$, and set $C_3=1$. 
Moreover if $C_0=C_1=C_2=0$ then the function 
$$
\widehat{\psi}(x):=\frac{\psi(x)}{\sqrt{x(x-1)(x-t)}}
$$ 
is entire and vanishes at $\infty$, which is a contradiction, 
so $C_j\ne 0$ for some $j\in \lbrace{0,1,2\rbrace}$. 

It is therefore easy to see that any possible configuration is equivalent to one of following four:

(1) all $C_j$ are nonzero; 

(2) $C_0=0$, $C_1,C_2\ne 0$; 

(3) $C_0,C_1=0$, $C_2\ne 0$; 

(4) $C_0,C_2=0$, $C_1\ne 0$.

We define the functions $f_j,g_j$ near $t_j$ as follows. 
First, $f_3$ is the restriction of $\psi$ and $g_3=-f_3^+-if_3$. 
Next, for $0\le j\le 2$, if $C_j\ne 0$ then we set 
$f_j$ to be the restriction of $\psi$ and $g_j=f_j^+-if_j$. 
Finally, if $C_j=0$, we choose $f_j$ to be any solution 
with leading asymptotics $\pm |x-t_j|^{\frac{1}{2}}\log|x-t_j|$ 
near $t_j$ and $g_j:=f_j^+-if_j$ (thus in this case we have a freedom of replacing $f_j$ by $\pm f_j+\lambda g_j$). Then $(f_j,g_j)$ 
is a basis of solutions for each $j$. 

Now consider cases 1-4 one by one. 

{\bf Case 1:} all $C_j$ are nonzero. In this case we have 
$$
B_jf_j=f_{j-1},\ B_jg_j=-a_jg_{j-1}+b_jf_{j-1}, 
$$
where $a_0a_1a_2a_3=1$ (here $a_j,b_j\in \Bbb R$). 
The equation $\prod_j B_jJ=-1$ then yields by a direct calculation: 
$$
b_0=b_2=b,\ b_1=b_3=2b^{-1},\ a_0=a_2=a,\ a_1=a_3=a^{-1},
$$
and either $a=1$ and $b^2\ne 2$ ({\bf case 1a}) or $b^2=2a$ ({\bf case 1b}). In {\bf case 1a}, 
$C_j=\pm 1$ for all $j$, so $\beta=\psi$ defines a balancing of the local system $\nabla(\bmu)$ corresponding to the oper $L(\bmu)$, and we get that $\bmu\in \mathcal B_*$. Moreover, this balancing is unique, so 
the fiber in $\mathcal B$ over $\bmu$ consists of one point. 

On the other hand, in {\bf case 1b}, we have $B_jB_{j+1}=1$ for all $j$, so besides $\psi$ we have another eigenfunction $\eta$ which is regular 
at $t_1$ and $t_3$. Thus there are two values 
of $c$ such that $\beta=\psi+c\eta$ is a balancing for $\nabla(\bmu)$. So $\bmu\in \mathcal{B}_*$ 
and the fiber over $\bmu$ in $\mathcal B$ consists of two points. 

{\bf Case 2:} $C_0=0$, $C_1,C_2\ne 0$. Thus $B_2,B_3$ are as in Case 1, 
but 
$$
B_1f_1=a_0g_0, \ B_1g_1=a_1f_0,\ B_0f_0=b_0f_3+a_0g_3, B_0g_0=a_0^{-1}f_3
$$
(choosing $f_0$ to be a multiple of $B_1g_1$ and using that $\det B_0=-1$). Then the lower left entry of the equation \ref{matrixeq} written as $B_0JB_1J=-(B_2JB_3J)^{-1}$
yields $b_3=0$, hence the upper right entry yields $a_1=0$, which is a contradiction. 
So this case is impossible. 

{\bf Case 3:} In this case $B_3$ is as in Case $1$ but 
$$
B_2f_2=a_0a_1g_1, \ B_2g_2=a_2f_1,\  B_1f_1=-a_1f_0+b_1g_0,\ B_1g_1=a_1^{-1}g_0,
$$
$$
B_0f_0=a_0g_3, B_0g_0=a_0^{-1}f_3
$$
(choosing $f_0$ to be a multiple of $B_0^{-1}g_3$, $f_1$ to be a multiple 
of $B_2g_2$ and using that $\det B_0=\det B_1=-1$). Then the upper right entry of the same equation as in Case 2 gives $a_0a_1a_2a_3=-1$, while the determinant of this equation gives $a_0a_1a_2a_3=1$, again a contradiction. So this case is impossible as well. 

{\bf Case 4:} $C_0,C_2=0$, $C_1\ne 0$. In this case we have 
$$
B_3f_3=a_3^{-1}g_2,\ B_3g_3=a_3f_2,\ B_2f_2=b_2f_1+a_2g_1,\ B_2g_2=cf_1,
$$
$$
B_1f_1=c^{-1}a_0a_3g_0,\ B_1g_1=a_1f_0+b_1g_0,\ B_0f_0=a_0g_3,\ B_0g_0=a_0^{-1}f_3
$$
(choosing $f_2$ to be a multiple of $B_3g_3$, $f_0$ to be a multiple of $B_0^{-1}g_3$, and using that $\det B_3=\det B_0=-1$). Then a direct calculation shows that for a suitable choice of signs of $f_0,f_2$ the equation 
$B_0JB_1J=-(B_2JB_3J)^{-1}$ is equivalent to the equations 
$$
b_0=b_1=b_2=b_3=0,\ a_0=\frac{1}{\sqrt{2}},\ a_1=a,\ a_2=a^{-1},\ a_3=\sqrt{2},\ c=2a^{-1}.
$$
Since $\det(B_0B_1)=\frac{a^2}{2}$, we have $f_1\sim \pm\frac{a}{\sqrt{2}}|x-t_1|^{\frac{1}{2}}\log|x-t_1|$ near $x=t_1$. Thus defining $\beta_\pm=f_3\pm \frac{1}{\sqrt{2}}g_3$, we see that $\beta_\pm$ gives rise to a balancing for the local system $\nabla(\bmu)$, so again $\bmu\in \mathcal{B}_*$. In other words, Case 4 is equivalent to Case 1b by changing the choice of a periodic eigenfunction $\psi$. Note that this case arises for 4 points in Subsection \ref{hyge}, as $2\cos \frac{\pi}{4}=\sqrt{2}$.  

Now let $\beta_n$, $n\ge 0$, be all possible balancings for local systems attached to opers; so they run over the set $\mathcal{B}$. By the $\Bbb S_3$-symmetry of the Schwartz kernel of Hecke operators, we have 
$$
H_x\beta_n=\beta_n(x)\beta_n,
$$ 
Let $\psi_n=\beta_n/\norm{\beta_n}$. 
Thus we obtain 
\begin{equation}\label{kerfor2}
\frac{\theta(f_t(x,y,z))}{\sqrt{f_t(x,y,z)}}=\sum_n\frac{\psi_{n}(x)\psi_{n}(y)\psi_{n}(z)}{\psi_{n,\infty}}. 
\end{equation}  

Thus we see that the Hecke operators $H_x$ fix a particular self-adjoint extension 
of the operator $\Bbb L$ in the real case and a particular normal extension of $\Bbb L$ in the complex case. 
In the complex case, this is exactly the extension described in \cite{EFK}, Part II, while in the real case it is the extension corresponding to the space $V_1$ of Subsection \ref{selfadjo}. 
 
\subsection{The subleading term of asymptotics of $H_x$ as $x\to \infty$}\label{subleading}

We have shown in Proposition \ref{asym} that the operator $\Bbb H_x$ has the asymptotics 
$\Bbb H_x\sim \norm{x}^{-1}\log\norm{x}$ as $x\to \infty$. So one may  ask for the next (subleading) 
term of the asymptotics.

\begin{proposition} 
There exists a strong limit (on the Schwartz space $\mathcal S$) 
$$
M:=\lim_{x\to \infty} (\norm{x}\Bbb H_x-\log\norm{x}),  
$$
which extends to an  unbounded self-adjoint operator on $\mathcal H$, essentially self-adjoint 
on $\mathcal{S}$. This operator is defined by the formula 
$$
M\psi_k=\mu^{(k)}\psi_k, 
$$
where 
$$
\mu^{(k)}:=\lim_{x\to \infty}(\norm{x}\psi_{k,\infty}^{-1}\psi_k(x)-\log\norm{x})\in \Bbb R. 
$$
\end{proposition} 

This proposition will follow from the explicit computation of $M$. We already gave a formula for $M$ for any number of parabolic points in Proposition \ref{Hxfor1}, but here we would like to do the same computation in a slightly different way. Namely, we will compute the Schwartz kernel $K_M(y,z)$ of $M$. Note that $f_t(x,y,z)$ is a quadratic polynomial in $x$ 
with leading coefficient $(z-y)^2$. This implies that outside of the diagonal we have 
$$
K_M(y,z)=\frac{2}{\norm{z-y}},\ y\ne z. 
$$
However, we are not yet done since $K_M$ turns out to have a singular part concentrated on the diagonal, and in any case the kernel $\frac{2}{\norm{z-y}}$ does not give rise to a well defined operator 
since it is not locally $L^1$. One possible regularization is given by 
$$
(M_0f)(y):=\int_{F}\frac{2f(z)-(1-{\rm sign}(\log\norm{z-y}))f(y)}{\norm{z-y}}\norm{dz}. 
$$
It remains to compute the operator $M-M_0$. We have 
$$
(M-M_0)\psi=h\psi, 
$$
where 
$$
h(y):=
\lim_{x\to \infty}
\left(\int_{F}\left(\frac{2\norm{x}\theta(f_t(x,y,z))}{\sqrt{\norm{f_t(x,y,z)}}}-\frac{1+{\rm sign}(\log\norm{z-y})}{\norm{z-y}}\right)\norm{dz}-\log\norm{x}\right).
$$
By translating $z$ by $-\frac{2(1+t)xy-(x+y)(t+xy)}{(x-y)^2}=y+O(x^{-1})$, $x\to \infty$ to complete the square and neglecting $O(x^{-1})$, we obtain  
$$
h(y):=
\lim_{x\to \infty}
\left(\int_{F}\left(\frac
{2\theta(z^2-\frac{D(x,y)}{4})}
{\sqrt{\norm{z^2-\frac{D(x,y)}{4}}}}-\frac{1+{\rm sign}(\log\norm{z})}{\norm{z}}\right)\norm{dz}-\log\norm{x}\right),
$$
where $D(x,y)$ is the discriminant of the quadratic polynomial $f_t(x,y,z)/(x-y)^2$ in $z$. 
By Lemma \ref{l33}, the integral under the limit equals $-\log\norm{\frac{D(x,y)}{16}}$. But 
$$
\frac{D(x,y)}{16}=\frac{x(x-1)(x-t)y(y-1)(y-t)}{(x-y)^4}\sim \frac{y(y-1)(y-t)}{x},\ x\to \infty.  
$$
Therefore Lemma \ref{l33} implies that 
$$
h(y)=-\log\norm{y(y-1)(y-t)}. 
$$
We thus obtain the following proposition. 

\begin{proposition}\label{Mope} We have 
$$
(Mf)(y)=\int_{F}\frac{2f(z)-(1-{\rm sign}(\log\norm{z-y}))f(y)}{\norm{z-y}}\norm{dz}
-f(y)\log\norm{y(y-1)(y-t)}. 
$$ 
\end{proposition} 

In other words, we have 
$$
K_M(y,z)=\frac{2}{\norm{z-y}}-\delta(z-y)\log\norm{y(y-1)(y-t)},
$$
where $\frac{2}{\norm{z-y}}$ stands for the Schwartz kernel of the operator $M_0$. 

Note that the distribution $\frac{1}{\norm{y}}$ regularized as above, i.e. by 
setting 
$$
(\tfrac{2}{\norm{y}},f):=\int_F \frac{2f(y)-(1-{\rm sign}(\log\norm{y})f(0)}{\norm{y}}\norm{dy}, 
$$
has Fourier transform 
$$
\mathcal F\left(\frac{2}{\norm{y}}\right)=-2\log\norm{p}.
$$
Thus the operator $M$ can be written as 
$$
M=-\mathcal F\circ 2\log\norm{y}\circ \mathcal F^{-1}-\log\norm{y(y-1)(y-t)}.
$$

Thus for $F=\Bbb R$ we have 
$$
M=-2\log|\partial|-\log|y(y-1)(y-t)|
$$
and for $F=\Bbb C$ we have 
$$
M=-2\log|\partial|^2-\log|y(y-1)(y-t)|^2. 
$$
In these two cases, we can easily reprove directly that $M$
indeed commutes with the Lam\'e operator $\Bbb L=\partial z(z-1)(z-t)\partial+z$. 
Namely, one just needs to establish the formal algebraic identity 
\begin{equation}\label{formalg} 
[2\log \partial+\log P, \partial P\partial+z]=0  
\end{equation} 
for any monic cubic polynomial $P=P(z)$, and then apply it to $P(z)=z(z-1)(z-t)$.\footnote{This identity makes sense, as both $[2\log\partial, 
\partial P\partial+z]$ and $[\log P, \partial P\partial+z]$ are well defined elements in the Weyl algebra $\Bbb C[z,\partial]$.}  
But \eqref{formalg} easily follows by a direct calculation in the fraction field of the 
Weyl algebra.    

\begin{remark} Let $F=\Bbb R$ or $\Bbb C$. Then we have $\Bbb L\psi_n=\Lambda_n\psi_n$. 
Let $f_n,g_n$ be the solutions of this equation near $\infty$ such that $f_n(z)\sim z^{-1}(\log z+o(1))$, 
$g_n(z)\sim z^{-1}(1+o(1))$, $z\to +\infty$. Then for $F=\Bbb C$
$$
\psi_{n,\infty}^{-1}\psi_n(z)=f_n(z)\overline{g_n(z)}+g_n(z)\overline{f_n(z)}+
\mu^{(n)}|g_n(z)|^2, 
$$
and $\mu^{(n)}$ is the unique constant for which this expression is monodromy invariant. 
On the other hand, if $F=\Bbb R$ then 
$$
\psi_{n,\infty}^{-1}\psi_n(z)=f_n(z)+\mu^{(n)}g_n(z).
$$
\end{remark} 

\begin{remark} In the archimedian case, this analysis suggests that the eigenvalues of the operator $M$ go to $+\infty$ logarithmically. Indeed, we expect that $M$ obeys the  {\bf Weyl law}: the number of eigenvalues $\le N$ can be obtained by semiclassical analysis from the asymptotics of the volume of the region 
$$
2\log\norm{p}+\log\norm{x(x-1)(x-t)}\le K
$$ 
in $T^*\Bbb P^1(F)$ as $K\to \infty$. But this volume equals $E(t)e^{K/2}$, where $E(t)$ is the elliptic integral defined in Subsection \ref{ellint}. This would imply that the number of eigenvalues of $M$ which are $\le N$ is asymptotic to $C_FE(t)e^{N/2}$, where $C_{\Bbb R}=\frac{2}{\pi}$ and 
$C_{\Bbb C}=\frac{1}{4}$, i.e., that 
the eigenvalues grow logarithmically. 
\end{remark} 

\subsection{Comparison with the work of S. Ruijsenaars}\label{ruij} 

In this subsection we would like to explain the connection of our results in the case of four points and $F=\Bbb R$ with the work of Ruijsenaars (\cite{Ru}).   

The points $0,t,1,\infty$ divide $Bun_0(\Bbb R)=\Bbb R\Bbb P^1$ into four intervals $I_0=[0,t]$, $I_1=[t,1]$, $I_2=[1,\infty]$, $I_3=[-\infty,0]$. 
Thus 
\begin{equation}\label{deco} 
\mathcal H=\mathcal H_0\oplus \mathcal H_1\oplus\mathcal H_2\oplus\mathcal H_3,
\end{equation} 
where $\mathcal H_j$ is the subspace of half-densities supported on $I_j$. Consider 
the self-adjoint extension $\Bbb L_0$ of the Lam\'e operator $\Bbb L$ corresponding to the space $V_0$ 
of Subsection \ref{selfadjo} (note that it does {\bf not} coincide with the extension $\Bbb L_1$ of $\Bbb L$ defined by the Hecke operators, which 
corresponds to the space $V_1$ of Subsection \ref{selfadjo}) 
It is clear that the operator $\Bbb L_0$ (unlike $\Bbb L_1$) preserves the subspaces $\mathcal H_j$, i.e., it is block-diagonal with respect to decomposition \eqref{deco}, since elements of $V_0$ are not required to satisfy any gluing conditions at $x_i$ (instead, they are just required to be bounded). 

Also, the subleading term $M$ of $H_x$ as $x\to \infty$ computed in Subsection \ref{subleading} is given by a 4-by-4 block matrix with entries $M_{ij}: \mathcal H_j\to \mathcal H_i$. It follows that $[M_{ij},\Bbb L_0]=0$ for all $i,j$. Note also that while $M$ is not bounded, the operators $M_{j,j+2}$ with $j\in \Bbb Z/4$ are Hilbert-Schmidt, since they have a continuous Schwartz kernel $\frac{2}{|y-z|}$ (it is continuous 
since the intervals $I_j,I_{j+2}$ are disjoint). Thus, the operator $\Bbb 
L_0: \mathcal H_j\to \mathcal H_j$ commutes with the Hilbert-Schmidt operator $M_{ij}^\dagger M_{ij}=M_{ji}M_{ij}$ on $\mathcal H_j$, which can thus be used to fix a self-adjoint extension of $\Bbb L$ on $\mathcal{H}_j$ (which, of course, coincides with $\Bbb L_0$). This method was proposed by S. Ruijsenaars in \cite{Ru}, and the kernel $\frac{2}{|y-z|}$ coincides (up to a change of variable) with the kernel $\mathcal S(u,v)$ used in \cite{Ru}. 

This shows that the spectrum of $\Bbb L_0$ on $\mathcal H_j$ (i.e., the spectrum of $\Bbb L$ with bounded boundary conditions on $I_j$) is the same as its spectrum on $\mathcal H_{j+2}$. This is also easy to show directly since the equivalence class of the operator $\Bbb L_0$ on $[t_{j-1},t_j]$  
only depends on the cross-ratio of the points $(t_{j-1},t_j,t_{j+1},t_{j+2})$, and this cross-ratio is unchanged under the map $j\mapsto j+2$. This spectrum is the spectrum of the Sturm-Liouville problem (1) from \cite{EFK}, Subsection 10.5. On the other hand, the spectrum of $\Bbb L_0$ on $\mathcal H_1$ and $\mathcal H_3$ is also the same and coincides with the spectrum of the Sturm-Liouville problem (2) from \cite{EFK}, Subsection 10.5. These singular Sturm-Liouville problems were introduced in \cite{Ta} 
following the work of Klein, Hilbert and V. I. Smirnov.

\section{Hecke operators on $\Bbb P^1$ with four parabolic points 
over a non-archimedian local field} \label{nonar} 

In this subsection we study Hecke operators 
over non-archimedian local fields for $G=PGL_2$ 
in the simplest case of $X=\Bbb  P^1$ with four parabolic points. 
We give a proof of the statement in \cite{K} that eigenvalues of Hecke operators are algebraic numbers, and relate these eigenvalues to eigenvalues of the usual Hecke operators over the residue field. 

\subsection{Mollified Hecke operators} \label{nonar1}
Let $F$ be a non-archimedian local field with residue field $\Bbb F_q$. Let $p$ be the characteristic of $\Bbb F_q$, and assume that $p>2$. Let $x_0\in F$, $m\in \Bbb Z$, and consider the {\bf mollified Hecke operator}
$$
H_{x_0,m}:=\int_{\norm{x-x_0}\le q^{-m}} H_x\norm{dx}.
$$
The Schwartz kernel of this operator 
has the form 
$$
K_{x_0,m}(y,z):=2\int_{\norm{x-x_0}\le q^{-m}} \frac{\theta(f_t(x,y,z))}{\sqrt{\norm{f_t(x,y,z)}}}\norm{dx}.
$$

The main result of this subsection is the following theorem. 

\begin{theorem}\label{molli} If $x_0\ne 0,1,t$, and $m\gg 0$ then the  
operator $H_{x_0,m}$ has finite rank. 
\end{theorem} 

The proof of Theorem \ref{molli} is given below. 

\begin{corollary}(\cite{K})\label{alnum} The eigenvalues of the Hecke operators $H_x$ are algebraic numbers.  
\end{corollary} 

\begin{proof}  Since Hecke operators commute, they preserve the finite dimensional vector spaces 
${\rm Im}(H_{x_0,m})$. Also it is clear that the restrictions of $H_x$ to 
${\rm Im}(H_{x_0,m})$ are expressed in a suitable basis by matrices with algebraic entries. Thus the eigenvalues of $H_x$ on ${\rm Im}(H_{x_0,m})$
are algebraic numbers. On the other hand, it is clear that the sum of the 
spaces 
${\rm Im}(H_{x_0,m})$ over various $x_0,m$ is dense in $\mathcal H$. This 
implies the statement. 
\end{proof} 

\begin{remark} Note that Corollary \ref{alnum} is a very special property, since eigenvalues of ``generic" 
$p$-adic integral operators are usually transcendental (see \cite{EK}). According to M. Kontsevich, 
this has to do with the fact that Hecke operators comprise an ``integrable system over a local field" in the sense of \cite{K}, Subsection 2.4.
\end{remark} 

The rest of the subsection is devoted to the proof of Theorem \ref{molli}.
For this purpose it is enough to show that the kernel of $K_{x_0,m}$ 
is of finite rank near each point of $\Bbb P^1\times \Bbb P^1$. 
This is achieved in Proposition \ref{locfi} at the end of this subsection. 

The operator $H_{x_0,m}$ is invariant under the group $\Bbb Z/2\times \Bbb Z/2$
acting simply transitively on the points $0,1,t,\infty$. Thus it suffices 
to consider only the finite region $\Bbb A^1\times \Bbb A^1\subset \Bbb P^1\times \Bbb P^1$. 

Recall that 
$$
f_{t}(x,y,z)=(y-z)^2x^2+2(2(1+t)yz-(y+z)(t+yz))x+(t-yz)^2.
$$
Thus 
$$
f_t(x,0,0)=t^2,\ f_t(x,1,1)=(t-1)^2,\ f_t(x,t,t)=t^2(t-1)^2, 
$$
$$
f_x(x,0,1)=(x-t)^2,\ f_t(x,0,t)=t^2(x-1)^2, f_t(x,1,t)=(t-1)^2x^2.
$$
This implies that for $m\gg 0$ the kernel $K_{x_0,m}(y,z)$ 
is locally constant (hence finite rank) near $(y,z)$ where $y,z\in \lbrace 0,1,t\rbrace$. 

Let us now study neighborhoods of other points. 
When $y\ne z$, we have 
$$
f_t(x,y,z)=(y-z)^2(x^2+bx+c),
$$
where 
$$
b:=2\frac{2(1+t)yz-(y+z)(t+yz)}{(y-z)^2},\ c:=\frac{(t-yz)^2}{(y-z)^2}. 
$$

\begin{lemma}\label{intformu} For $m\in \Bbb Z$, $c\in F$ let  
$$
J_m(c):=\int_{\norm{x-1}\le q^{-m}} \frac{\theta(x^2+c)}{\sqrt{\norm{x^2+c}}}.
$$
Then for $y\ne z$ and $x_0+\frac{b}{2}\ne 0$ we have 
$$
K_{x_0,m}(y,z)=
\frac{2J_{m-k}\left(\frac{c-\frac{b^2}{4}}{(x_0+\frac{b}{2})^2}\right)}{\norm{y-z}}.
$$ 
where $k:=-\log_q \norm{x_0+\frac{b}{2}}$.
\end{lemma} 

\begin{proof} We have 
$$
K_{x_0,m}(y,z)=\frac{2J_m(x_0,b,c)}{\norm{y-z}},
$$
where for $b,c\in F$, 
$$
J_m(x_0,b,c):=\int_{\norm{x-x_0}\le q^{-m}} \frac{\theta(x^2+bx+c)}{\sqrt{\norm{x^2+bx+c}}}\norm{dx}.
 $$
 Making the change of variable $u=\frac{x+\frac{b}{2}}{x_0+\frac{b}{2}}$, we see that 
 $$
 J_m(x_0,b,c)=J_{m-k}\left(\frac{c-\tfrac{b^2}{4}}{(x_0+\frac{b}{2})^2}\right).
 $$
  This implies the statement. 
 \end{proof} 

Let 
$$
D=D_t(y,z):=16y(y-1)(y-t)z(z-1)(z-t)=(b^2-4c)(y-z)^4
$$ 
and
$$
P=P_t(x_0,y,z):=(x_0+\tfrac{b}{2})(y-z)^2=x_0(y-z)^2+2(1+t)yz-(y+z)(t+yz).
$$
Then
$$
\frac{c-\frac{b^2}{4}}{(x_0+\frac{b}{2})^2}=-\frac{D}{4P^2}.
$$
Thus Lemma \ref{intformu} implies 

\begin{corollary} If $y\ne z$ and $x_0+\frac{b}{2}\ne 0$ then
 $$
K_{x_0,m}(y,z)=\frac{2J_{m-k}(-\tfrac{D}{4P^2})}{\norm{y-z}}.
$$ 
\end{corollary}

Let us now proceed with computation of $J_m(c)$. 

{\bf Case 1: $m>0$.} In this case, if $\norm{x-1}\le q^{-m}$ then $\norm{x}=1$. 
Note also that 
$$
x^2+c=(x+1)(x-1)+(1+c)
$$ 
and $\norm{x+1}=1$. Thus we have 
the following cases. 

 {\bf Case 1a.} $ \norm{1+c}=q^{-r}>q^{-m}$. Then
 we have 
 $$
 J_m(c)=q^{-m}\frac{\theta(1+c)}{\sqrt{\norm{1+c}}}.
 $$
 In particular, if $\norm{1+c}>1$ (or, equivalently, $\norm{c}> 1$) then 
$$
J_m(c)=q^{-m}\frac{\theta(c)}{\sqrt{\norm{c}}}.
$$

 {\bf Case 1b.} $\norm{1+c}=q^{-r}\le q^{-m}$. In this case we'll need the following lemma. 
 
 \begin{lemma}\label{leee} Let $\norm{a}=q^{-r}<1$. Then 
 $$
 \int_{\norm{u}=q^{-r}}\frac{\theta(u+a)}{\sqrt{\norm{u+a}}}=\tfrac{1}{2}(1-q^{-1})\sum_{\ell=r}^\infty \frac{1+(-1)^\ell}{2}q^{-\frac{\ell}{2}}-q^{-\frac{r}{2}-1}\theta(a). 
 $$
 \end{lemma} 
 
 \begin{proof} We have 
 $$
 \int_{\norm{u}\le q^{-r}}\frac{\theta(u+a)}{\sqrt{\norm{u+a}}}=
 \int_{\norm{u}=q^{-r}}\frac{\theta(u)}{\sqrt{\norm{u}}}=\tfrac{1}{2}(1-q^{-1})\sum_{\ell=r}^\infty \frac{1+(-1)^\ell}{2}q^{-\frac{\ell}{2}}.
 $$
On the other hand, 
$$
 \int_{\norm{u}\le q^{-r-1}}\frac{\theta(u+a)}{\sqrt{\norm{u+a}}}=q^{-\frac{r}{2}-1}\theta(a).
 $$
 Subtracting, we obtain the desired statement. 
 \end{proof} 
 
 Now, we have 
 $$
 J_m(c)=\sum_{\ell=m}^\infty \int_{\norm{x-1}=q^{-\ell}} \frac{\theta(x^2+c)}{\sqrt{\norm{x^2+c}}}\norm{dx}.
 $$
 Splitting the sum into three parts $\ell<r$, $\ell=r$, $\ell>r$, we get 
 $$
J_m(c)=\tfrac{1}{2}(1-q^{-1})\sum_{\ell=m}^{r-1} \frac{1+(-1)^\ell}{2}q^{-\frac{\ell}{2}}+
\int_{\norm{x-1}=q^{-r}} \frac{\theta(x^2+c)}{\sqrt{\norm{x^2+c}}}\norm{dx}+ q^{-\frac{r}{2}-1}\theta(1+c).
 $$
 To compute the integral, we use the change of variable $x^2-1=u$ and Lemma \ref{leee}. 
 Then we get  
 $$
J_m(c)= \tfrac{1}{2}(1-q^{-1})\sum_{\ell=m}^{r-1} \frac{1+(-1)^\ell}{2}q^{-\frac{\ell}{2}} 
+\tfrac{1}{2}(1-q^{-1})\sum_{\ell=r}^\infty \frac{1+(-1)^\ell}{2}q^{-\frac{\ell}{2}}-q^{-\frac{r}{2}-1}\theta(1+c)+
 q^{-\frac{r}{2}-1}\theta(1+c)
=
$$
$$
\tfrac{1}{2}(1-q^{-1})\sum_{\ell=m}^{\infty} \frac{1+(-1)^\ell}{2}q^{-\frac{\ell}{2}}. 
$$
Thus 
$$
J_m(c)=\tfrac{1}{2}q^{-\frac{m}{2}}
$$
if $m$ is even and 
$$
J_m(c)=\tfrac{1}{2}q^{-\frac{m+1}{2}}
$$
if $m$ is odd. 

{\bf Case 2: $m\le 0$.} In this case we have 
$$
J_m(c)=\int_{\norm{x}\le q^{-m}} \frac{\theta(x^2+c)}{\sqrt{\norm{x^2+c}}}\norm{dx}.
$$

Setting $x=\pi^m y$, where $\pi\in F$ is a uniformizer, we get 
$$
J_m(c)=\int_{\norm{y}\le 1} \frac{\theta(y^2+c\pi^{-2m})}{\sqrt{\norm{y^2+c\pi^{-2m}}}}\norm{dy}=J_0(c\pi^{-2m}). 
$$
So we need to compute $J_0(c)$. 

{\bf Case 2a.} If $\norm{c}>1$ then we have 
$$
J_0(c)=\frac{\theta(c)}{\sqrt{\norm{c}}}.
$$

{\bf Case 2b.} If $\norm{c}=q^{-2r+1}<1$ then
$$
J_0(c)=(1-q^{-1})r=-\tfrac{1}{2}(1-q^{-1})(\log_q\norm{c}-1). 
$$
On the other hand, if $\norm{c}=q^{-2r}\le 1$ then we have 
$$
J_0(c)=-\tfrac{1}{2}(1-q^{-1})\log_q\norm{c}+\int_{\norm{x}=q^{-r}}\frac{\theta(x^2+c)}{\sqrt{\norm{x^2+c}}}\norm{dx}+
q^{-1}\theta(c). 
$$
We also have 
$$
\int_{\norm{x}=q^{-r}}\frac{\theta(x^2+c)}{\sqrt{\norm{x^2+c}}}\norm{dx}=
\int_{\norm{y}=1}\frac{\theta(y^2+c\pi^{-2r})}{\sqrt{\norm{y^2+c\pi^{-2r}}}}\norm{dy}.
$$
We will now use the following lemma. 

\begin{lemma}\label{leee2} Let $\norm{a}=1$. Then 
$$
\int_{\norm{y}=1}\frac{\theta(y^2-a)}{\sqrt{\norm{y^2-a}}}\norm{dy}=\tfrac{1}{2}(1-q^{-1})-q^{-1}\theta(-a).
$$  
\end{lemma} 

\begin{proof} Denote the integral in question by $I$. Assume first that $a$ is a non-square. Then $\norm{y^2-a}=1$ for all $y$ with $\norm{y}=1$. So $I=q^{-1}N$, where 
$N$ is the number of $y\in \Bbb F_q^\times$ such that $y^2-a=x^2$ for some $x$. 
Then the number of solutions of the equation $y^2-a=x^2$ in $\Bbb F_q$ such that $y\ne 0$ is $2N$. 
But this equation can be written as $(y-x)(y+x)=a$, and $y-x$ can be chosen any nonzero element in $\Bbb F_q$, which completely determines $y+x$, hence $y$ and $x$. However, we need to exclude the case $y=0$, which gives two solutions iff $-a$ is a square. 
Thus $2N=q-1-2\theta(-a)$, which implies the statement. 

Now assume that $a$ is a square. Then we have 
$$
I=\left(\int_{\norm{y}=1, y\ne \pm \sqrt{a}\text{ mod }\pi}+\int_{y=\sqrt{a}\text{ mod }\pi}+
\int_{y=-\sqrt{a}\text{ mod }\pi}\right)\frac{\theta(y^2-a)}{\sqrt{\norm{y^2-a}}}\norm{dy}.
$$
The first integral equals $\frac{1}{2}q^{-1}$ times the number of solutions of the equation $y^2-a=x^2$ over $\Bbb F_q$ excluding $(0,\pm \sqrt{a})$ with $y\ne 0$, which is $\frac{1}{2}q^{-1}(q-3-2\theta(-a))$. 
On the other hand, 
$$
\int_{y=\sqrt{a}\text{ mod }\pi}\frac{\theta(y^2-a)}{\sqrt{\norm{y^2-a}}}\norm{dy}=
\int_{y=\sqrt{a}\text{ mod }\pi}\frac{\theta(2\sqrt{a}(y-\sqrt{a}))}{\sqrt{\norm{y-\sqrt{a}}}}\norm{dy}=
\int_{\norm{x}\le q^{-1}}\frac{\theta(x)}{\sqrt{\norm{x}}}\norm{dx}=\tfrac{1}{2}q^{-1}. 
$$
This implies the statement. 
\end{proof} 

Lemma \ref{leee2} implies that
$$
\int_{\norm{x}=q^{-r}}\frac{\theta(x^2+c)}{\sqrt{\norm{x^2+c}}}\norm{dx}=
\tfrac{1}{2}(1-q^{-1})-q^{-1}\theta(c). 
$$
Thus 
$$
J_0(c)=-\tfrac{1}{2}(1-q^{-1})(\log_q\norm{c}-1).
$$
We see that this formula in fact holds for both odd and even powers of $q$. 

\begin{proposition}\label{locfi} The kernel $K_{x_0,m}(y,z)$ is locally of finite rank. 
That is, for any $y_0,z_0\in \Bbb P^1(F)$ there exists $\varepsilon>0$ 
such that when $\norm{y-y_0}<\varepsilon,\norm{z-z_0}<\varepsilon$, 
we can write $K_{x_0,m}(y,z)$ in the form 
$$
K_{x_0,m}(y,z)=\sum_{i=1}^N a_i(y)b_i(z).
$$ 
\end{proposition} 

\begin{proof} We first consider the case $y\ne z$, so $\norm{y-z}>0$. 

{\bf Case A.} $\norm{x_0+\frac{b}{2}}>q^{-m}$, i.e., $k<m$. This is equivalent to  
$$
\norm{P}>q^{-m}\norm{y-z}^2
$$ 
and means we are in Case 1. Then $K_{x_0,m}$ is locally constant, as desired.  

{\bf Case B.} $\norm{x_0+\frac{b}{2}}\le q^{-m}$, 
i.e., $k\ge m$. This is equivalent to the condition that
$$
\norm{P}\le q^{-m}\norm{y-z}^2.
$$ 
This means we are in Case 2.  Then the condition for Case 2a is 
$$
\norm{D}> q^{-2m}\norm{y-z}^4
$$ 
which means that $K_{x_0,m}$ is locally constant. 
So it remains to consider the case  
$$
\norm{D}\le q^{-2m}\norm{y-z}^4,
$$ 
which means we are in Case 2b. Then it is easy to check that 
$K_{x_0,m}$ is locally constant unless $D=0$, in which case it is not locally constant but is the product of a function of $y$ and a function of 
$z$ (as so is $D$), so still of finite rank. 

 Finally, consider the case $y=z$. 

{\bf Case C.} $y=z$. In this case $P=-2y(y-1)(y-t)$. So if $P=0$ then we get $y(y-1)(y-t)=0$ 
so $y=z=0,1,t$, a case that has already been considered. 
Thus we may restrict to the case when $\norm{P}>0$. 
In this case $K_{x_0,m}$ is locally constant unless 
$f_t(x_0,y,y)=0$. On the other hand, if $f_t(x_0,y,y)=0$, assume that 
$\norm{P}=q^{-s}$. Then near this point $\norm{x+\frac{b}{2}}\norm{y-z}^{2}=q^{-s}$, 
i.e., $q^{-k}\norm{y-z}^{2}=q^{-s}$, which yields that $\norm{y-z}=q^{\frac{k-s}{2}}$. In particular, this implies that $k-s$ is even, so the parity of $m-k$ is the same as the parity of $m-s$ (a fixed number). Also 
in this case $1-\frac{D}{P^2}=0$, so near this point we are in Case 1b. 
Thus in this case $K_{x_0,m}$ is also locally constant. 
\end{proof} 

This completes the proof of Theorem \ref{molli}. 

\begin{remark} This proof shows that $K_{x_0,m}$ fails to be locally constant near $(y,z)$ 
if and only if $D=0$ and $\norm{P}\le q^{-m}\norm{x-y}^2$. As $m\to \infty$, 
this set shrinks to the following 8 points: 
$$
(0,\tfrac{t}{x_0}),\ (1,\tfrac{x_0-t}{x_0-1}),\ (t,\tfrac{t(x_0-1)}{x_0-t}),\ (\infty, x_0),\ 
(\tfrac{t}{x_0},0),\ (\tfrac{x_0-t}{x_0-1},1),\ (\tfrac{t(x_0-1)}{x_0-t},t),\ (x_0,\infty).
$$
\end{remark} 

\subsection{Computation of eigenvalues of Hecke operators} \label{nonar2}
We would now like to compute the first ``batch" of eigenvalues of 
the Hecke operators, namely the eigenvalues on the finite dimensional space generated under the Hecke operators by the characteristic functions of 
balls of radius $q^{-1}$. We will show that this space has dimension $q+5$ and relate the eigenvalues of the Hecke operators on this space to eigenvalues 
of the usual Hecke operators over the finite field $\Bbb F_q$. 

\subsubsection{Computation of $K_{x,1}(y,z)$.} 
Let $\bold 1_x$ be the indicator function of the ball of radius $q^{-1}$ around $x\in F$, and let us compute $H_z\bold 1_x$, where $\norm{z}=1$, 
$z\ne 0,1,t$ mod $\pi$ (where, as before, $\pi\in F$ is the uniformizer). 
So we may assume that $\norm{x}\ge 1$. We also recall that 
$$
(y-z)^2f_t(x,y,z)=4P^2-D. 
$$

We have 
$$
(H_z\bold 1_x)(y)=K_{x,1}(y,z).
$$
So, as shown in the previous subsection, generically we get 
$$
(H_z\bold 1_x)(y)=\frac{2J_{1-k}(-\tfrac{D}{4P^2})}{\norm{y-z}}.
$$
where $k=-\log_q\norm{x+\frac{b}{2}}$. 
Thus we have the following cases, corresponding to the 
cases with the same numbers considered in the previous subsection. 

{\bf Case 1.} $\norm{P_t(x,y,z)}\ge \norm{y-z}^2$.  

{\bf Case 1a.} $\norm{f_t(x,y,z)}\ge \norm{P_t(x,y,z)}$. Then
 we have 
 $$
 (H_z\bold 1_x)(y)=\frac{2q^{-1}\theta(f_t(x,y,z))}{\sqrt{\norm{f_t(x,y,z)}}}.
 $$
 
 {\bf Case 1b.} $\norm{f_t(x,y,z)}\le q^{-1} \norm{P_t(x,y,z)}$. Then we have 
 $$
  (H_z\bold 1_x)(y)=\frac{q^{-s}\norm{y-z}}{\sqrt{\norm{P_t(x,y,z)}}},
 $$ 
 where $s=1$ if $\norm{P}$ is an even power of $q$ and 
 $s=\frac{1}{2}$ if $\norm{P}$ is an odd power of $q$. 
  
{\bf Case 2.} $\norm{P_t(x,y,z)}\le q^{-1} \norm{y-z}^2$.  
In this case, we have the following cases. 

{\bf Case 2a.} If $\norm{D_t(y,z)}\ge q^{-1}\norm{y-z}^4$ (equivalently, $\norm{f_t(x,y,z)}\ge q^{-1}\norm{y-z}^2$) then we have 
$$
(H_z\bold 1_x)(y)=\frac{2q^{-1}\theta(f_t(x,y,z))}{\sqrt{\norm{f_t(x,y,z)}}}.
$$

{\bf Case 2b.} If $\norm{D_t(y,z)}\le q^{-2}\norm{y-z}^4$ then
$$
(H_z\bold 1_x)(y)=\frac{-(1-q^{-1})(\log_q\norm{\tfrac{D_t(y,z)}{(y-z)^4}}+1)}{\norm{y-z}}.
$$

\subsubsection{Diagonalization of Hecke operators} 
Let $\mathcal O=\mathcal O_{F}$ be the ring of integers of $F$. 
For $x\in \mathcal O$ denote by $x_0$ its reduction to the residue field $\Bbb F_q$. 
Note that the function $\bold 1_x$ 
depends only on $x_0$. 
Let $t\in \mathcal O$ be such that $t_0\ne 0,1$. Let 
$z\in \mathcal O$, $z_0\ne 0,1,t_0$, and $x\in \mathcal O$. 
Assume first that $\norm{y}\le 1$. We have $\norm{y-z}\le 1$ while $\norm{P_t(x,y,z)}\le 1$, $\norm{f_t(x,y,z)}\le 1$. 

So if $P_{t_0}(x_0,y_0,z_0)\ne 0$ and $f_{t_0}(x_0,y_0,z_0)\ne 0$ 
then we are in Case 1a and 
$$ 
(H_z\bold 1_x)(y)=2q^{-1}\theta(f_{t_0}(x_0,y_0,z_0)).
$$
However, if $P_{t_0}(x_0,y_0,z_0)\ne 0$ but $f_{t_0}(x_0,y_0,z_0)=0$ then we are in Case 1b and 
$$ 
(H_z\bold 1_x)(y)=q^{-1}.
$$

On the other hand, if $P_{t_0}(x_0,y_0,z_0)=0$ (which necessarily implies $y_0\ne z_0$) then 
we have $\norm{D_t(y,z)}\le 1$. So if $\norm{D_t(y,z)}=1$ (i.e., $y_0\ne 0,1,t_0$) then we are in Case 2a and we have 
$$ 
(H_z\bold 1_x)(y)=2q^{-1}\theta(f_{t_0}(x_0,y_0,z_0))=2q^{-1}\theta(-D_{t_0}(y_0,z_0)).
$$
If $\norm{D_t(y,z)}\le q^{-1}$ (i.e., $y_0=0,1,t_0$) then we are in Case 2a ($\norm{D_t(y,z)}=q^{-1}$) or Case 2b ($\norm{D_t(y,z)}\le q^{-2}$) but in both cases we have 
$$
(H_z\bold 1_x)(y)=-(1-q^{-1})(\log_q\norm{y-r}+1).
$$
where $r=0,1,t$ respectively. 

Now assume $\norm{y}\ge q$. 
Then $\norm{y-z}=\norm{y}$, $\norm{P_t(x,y,z)}\le \norm{y}^2$, $\norm{D_t(x,y,z)}=\norm{y}^3$, 
so if $\norm{P_t(x,y,z)}=\norm{y}^2$ (i.e., $x_0\ne z_0$)
we are in Case 1a and 
$$
(H_z\bold 1_x)(y)=\frac{2q^{-1}}{\norm{y}},
$$
while if $\norm{P_t(x,y,z)}\le q^{-1}\norm{y}^2$ (i.e., $x_0=z_0$) then 
if $\norm{y}=q$ we are in Case 2a and if $\norm{y}\ge q^2$ then we are in Case 2b, but in both cases we have 
 $$
(H_z\bold 1_x)(y)=\frac{-(1-q^{-1})(\log_q\norm{y^{-1}}+1)}{\norm{y}}.
$$

For $r=0,1,t$ let us introduce the functions 
$$
\phi_r:=-(1-q^{-1})(\log_q\norm{y-r}+1), \norm{y-r}\le q^{-1}
$$
and vanishing otherwise. We also have a similar function $\phi_\infty$ which differs by dividing by $\norm{y}$ (to account for the fact that we have half-forms):
$$
\phi_\infty:=\frac{-(1-q^{-1})(\log_q\norm{y^{-1}}+1)}{\norm{y}}, \norm{y}\ge q
$$
and vanishing otherwise. Also we introduce the function 
$$
\bold 1_\infty(y)=\frac{1}{\norm{y}}\text{ if }\norm{y}\ge q
$$ 
and zero otherwise. 
Then we have 
$$
H_z\bold 1_x=\sum_{j\in \Bbb P^1(\Bbb F_q)} a_{x_0j}\bold 1_y+\sum_{r=0,1,t,\infty}b_{x_0r}\phi_r, 
$$
where the coefficients $a_{ij},b_{ir}$ are as follows. 

{\bf Case A.} $i\ne 0,1,t_0,\infty$.  

{\bf Case A1.} If $j\ne 0,1,t_0,\infty$ then if $f_{t_0}(i,j,z_0)\ne 0$, we have 
$$
a_{ij}=2q^{-1}\theta(f_{t_0}(i,j,z_0))
$$
while if $f_{t_0}(i,j,z_0)= 0$ then 
$$
a_{ij}=q^{-1}.
$$

{\bf Case A2.} If $j=0$ then if $i\ne t_0/z_0$, we have 
$$
a_{ij}=2q^{-1}
$$
while if $i= t_0/z_0$, we have 
$$
a_{ij}=0.
$$
The same holds for $j=1,t_0,\infty$, where the corresponding equations have the form 
$i=\frac{t_0-z_0}{1-z_0}$, $i=\frac{t_0(z_0-1)}{z_0-t_0}$, $i=z_0$ respectively. 

Also if $i\ne t_0/z_0$, we have 
$$
b_{i0}=0
$$
while if $i=t_0/z_0$ then 
$$
b_{i0}=q.
$$
The same holds for $r=1,t,\infty$, where the corresponding equations have the form 
$i=\frac{t_0-z_0}{1-z_0}$, $i=\frac{t_0(z_0-1)}{z_0-t_0}$, $i=z_0$ respectively. 

{\bf Case B.} $i=0,1,t,\infty$. Consider first the case $i=0$. 

{\bf Case B1.} If $j\ne 0,1,t,\infty$ then if $f_{t_0}(0,j,z_0)=(t_0-jz_0)^2\ne 0$ (i.e., $j\ne t_0/z_0$)  
then  
$$
a_{0j}=2q^{-1},
$$
while if $j=t_0/z_0$ then 
$$
a_{0j}=q^{-1}.
$$

{\bf Case B2.} $j=0$. We have $f_{t_0}(0,0,z_0)=t_0^2$ so 
$$
a_{0j}=2q^{-1}.
$$

{\bf Case B3.} $j=1$. We have $f_{t_0}(0,1,z_0)=(z_0-t_0)^2$ so 
$$
a_{0j}=2q^{-1}.
$$

{\bf Case B4.} $j=t_0$. We have $f_{t_0}(0,t_0,z_0)=t_0^2(z_0-1)^2$ so 
$$
a_{0j}=2q^{-1}.
$$

{\bf Case B5.} $j=\infty$. Since $z_0\ne 0$, we have 
$$
a_{0j}=2q^{-1}.
$$

In the cases $i=1,t,\infty$ the answer is the same. 

Also it is easy to see that 
$$
b_{ir}=0
$$
for all $i,r=0,1,t,\infty$. 

Now note that the functions $\bold 1_x$ are orthogonal with $\norm{\bold 1_x}^2=q^{-1}$. 
We would like to correct the functions $\phi_r$ so that 
they become orthogonal to each other and to $\bold 1_x$. 
So let us introduce 
$$
\psi_r:=\phi_r-\beta \bold 1_r
$$
and choose $\beta$ so this is orthogonal to $\bold 1_r$. 
Then
$$
\beta=q(\psi_r,\bold 1_r).
$$
So for $r=0,1,t$ we have 
$$
\beta=-q(1-q^{-1})\int_{\norm{y}\le q^{-1}}(\log_q\norm{y-r}+1)\norm{dy}=
q^{-1}.
$$  
Thus for $r=0,1,t$ 
$$
\psi_r(y)=-(1-q^{-1})\log_q\norm{y-r}-1,\ \norm{y-r}\le q^{-1}
$$
and zero otherwise, and 
$$
\psi_\infty(y)=\frac{-(1-q^{-1})\log_q\norm{y^{-1}}-1}{\norm{y}},\ \norm{y}\ge q
$$
and zero otherwise. 

Let us now normalize these vectors. We have 
$$
\norm{\psi_0}^2=\int_{\norm{y}\le q^{-1}}|\psi_0(y)|^2\cdot \norm{dy}=q^{-2}.
$$
So the normalized basis vectors are 
$$
\widehat{\bold 1}_x:=q^{1/2}\bold 1_x,\ \widehat{\psi}_r:=q\psi_r.
$$
In this basis the matrix of the operator $H_z$ is symmetric, and has the following form. 

\begin{proposition} 1. Let $x_0,y_0\ne 0,1,t_0$. If $f_{t_0}(x_0,y_0,z_0)$ is not a square then 
$$
(\widehat{\bold 1}_x, H_z\widehat{\bold 1}_y)=0.
$$
If $f_{t_0}(x_0,y_0,z_0)=0$ then 
$$
(\widehat{\bold 1}_x, H_z\widehat{\bold 1}_y)=q^{-1}.
$$

2. We have 
$$
(\widehat{\bold 1}_x,H_z\widehat{\bold 1}_y)=q^{-1}
$$
at the following 8 positions: 
\begin{equation}\label{8pos} 
(0,\tfrac{t_0}{z_0}),\ (1,\tfrac{z_0-t_0}{z_0-1}),\ (t,\tfrac{t_0(z_0-1)}{z_0-t_0}),\ (\infty, z_0),\ 
(\tfrac{t_0}{z_0},0),\ (\tfrac{z_0-t_0}{z_0-1},1),\ (\tfrac{t_0(z_0-1)}{z_0-t_0},t_0),\ (z_0,\infty).
\end{equation}

3. In all other cases 
$$
(\widehat{\bold 1}_x,H_z\widehat{\bold 1}_y)=2q^{-1}.
$$

4. We have 
$$
(\widehat{\bold 1}_x,H_z\widehat\psi_r)=(\widehat\psi_r,H_z\widehat{\bold 1}_x)=q^{-\frac{1}{2}}
$$
at the 8 positions of \eqref{8pos}. Otherwise 
$$
(\widehat{\bold 1}_x,H_z\widehat\psi_r)=(\widehat\psi_r,H_z\widehat{\bold 1}_x)=0.
$$

5. We have 
$$
(\widehat\psi_s,H_z\widehat\psi_r)=0.
$$
\end{proposition} 

\begin{proof} (1)-(4) follow from the formulas for $a_{ij},b_{ij}$ and the self-adjointness of $H_z$. 
(5) is checked by an easy direct computation, using that the function $f_t(x,y,z)$ is a square and has 
norm $1$ when the distances from $x$ to $r$ and from $y$ to $s$ are $\le q^{-1}$. 
\end{proof} 

In particular, we see that in the orthogonal (but not orthonormal) basis $\bold 1_x,\psi_r$ the matrix 
$qH_z$ (of size $q+5$) has integer entries $0,1,2,q$ (even though it is non-symmetric in this basis). 

\begin{proposition}\label{5eig}  (i) The Perron-Frobenius eigenvalue of $H_z$ equals 
$$
\lambda_+=(1+q^{-1})\frac{1+\sqrt{1+\frac{32q}{(q+1)^2}}}{2}=1+9q^{-1}+O(q^{-2}). 
$$
The corresponding eigenvector is 
$$
v_+=\lambda_+\sum_{i\in \Bbb P^1(\Bbb F_q)} \bold 1_i+\sum_{r=0,1,t,\infty} (\bold 1_r+\psi_r).
$$
The operator $H_z$ also has the eigenvalue 
$$
\lambda_-=(1+q^{-1})\frac{1-\sqrt{1+\frac{32q}{(q+1)^2}}}{2}=-8q^{-1}+O(q^{-2}).
$$
with eigenvector 
$$
v_-=\lambda_-\sum_{i\in \Bbb P^1(\Bbb F_q)} \bold 1_i+\sum_{r=0,1,t,\infty} (\bold 1_r+\psi_r).
$$
These eigenvalues are the roots of the equation 
$$
q\lambda^2-(q+1)\lambda-8=0.
$$ 

(ii) $H_z$ has eigenvalue $0$ with multiplicity at least $3$. Namely, the 
vectors
$$
\sum_{r=0,1,t,\infty} a_r(\bold 1_r+\psi_r)
$$
with $a_1+a_2+a_3+a_4=0$ are null vectors for $H_z$. 
\end{proposition}

\subsection{Relation to Hecke operators over a finite field} \label{nonar3} 
The eigenvectors of Proposition \ref{5eig} span the $5$-dimensional space 
$V_0$ with basis 
$\bold 1:=\sum_{i\in \Bbb P^1(\Bbb F_q)} \bold 1_i$ and $e_r:=\bold 1_r+\psi_r$, 
$r=0,1,t,\infty$. Note that the matrix of the Hecke operator $H_z$ on $V_0$ in this basis 
does not depend on $t$. 

Let us now consider the orthogonal complement $V=V_1$ of this space, which has dimension $q$. 
Namely, $V$ is the space of vectors 
$$
\sum_{i\in \Bbb P^1(\Bbb F_q)} a_i\bold 1_i-q\sum_r a_r\psi_r
$$
where $\sum_{i\in \Bbb P^1(\Bbb F_q)}a_i=0$. Its basis is formed by the 
vectors 
$v_k:={\bf 1}_k-{\bf 1}_\infty+q\psi_\infty$ if $k\ne 0,1,t$ and $v_k:={\bf 1}_k-{\bf 1}_\infty-q\psi_k+q\psi_\infty$ 
if $k=0,1,t$. 

Thus the matrix of the operator $qH_z$ on $V$ in this basis is given by the formula
$$
[qH_z|_V]_{x,y}=N(t,z,x,y)-2+(q+1)\delta_{x,z}-q\delta_{0,y}\delta_{x,\frac{t}{z}}
-q\delta_{1,y}\delta_{x,\frac{z-t}{z-1}}-q\delta_{t,y}\delta_{x,\frac{t(z-1)}{z-t}}.
$$
where $N(t,z,x,y)$ is the number of solutions of 
the equation $w^2=f_t(x,y,z)$ in $\Bbb F_q$ (which is $2,1$ or $0$). 
This means that $H_z|_V=-q^{-1}T_z$, where $T_z$ 
is the cuspidal component of the Hecke operator over $\Bbb F_q$ (see \cite{K}, p.4). 
Namely, the formula for the matrix $T_z$ (also valid if $z=0,1,t$) is 
$$
(T_z)_{x,y}=2-N(t,z,x,y)-c_z\delta_{x,z} +(1-\delta_{z(z-1)(z-t),0})(q\delta_{0,y}\delta_{x,\frac{t}{z}}
+q\delta_{1,y}\delta_{x,\frac{z-t}{z-1}}+q\delta_{t,y}\delta_{x,\frac{t(z-1)}{z-t}}),
$$
where $c_z=1$ if $z=0,1,t$ and $c_z=q+1$ if $z\ne 0,1,t$.\footnote{Here we correct a misprint in the formula on the top of p.4 in \cite{K}, where it is written that $c_z=q+1$ if $z=0,1,t$ and $c_z=1$ if $z\ne 0,1,t$. So these should be switched.} 

\begin{example}Let $q=5$, $t=4$, $z=2$.\footnote{We thank M. Vlasenko for sharing her computations in this case.} 
 Then the 10-by-10 matrix of 
$qH_z$ in the basis $\bold 1_x,\psi_r$ is 
$$
\begin{pmatrix} 
2& 2& 1& 2& 2& 2&  &  &  &   \\ 2& 2& 2& 1& 2& 2&  &  &  &   \\ 1& 2&  &  
& 2& 1& 1&  &  & 1\\ 2& 1&  &  & 1& 2&  & 1& 1&   \\ 2& 2& 2& 1& 2& 2&  & 
 &  &   \\ 2& 2& 1& 2& 2& 2&  &  &  &   \\  &  & 5&  &  &  &  &  &  &   \\  &  &  & 5&  &  &  &  &  &   \\  &  &  & 5&  &  &  &  &  &   \\  &  & 5&  &  &  &  &  &  & \end{pmatrix}
$$
(where empty positions stand for zeros, for better viewing). The eigenvalues of this matrix 
are $0$ (with multiplicity $5$), $10, -4, 2, 2\sqrt{3},-2\sqrt{3}$. The Perron-Frobenius eigenvector (with eigenvalue $10$) has coordinates  
$(3,3,2,2,3,3,1,1,1,1)$.

The matrix $-qH_z|_V=T_z$ has the form 
$$
\begin{pmatrix} 
0& 0& 1& 0& 0\\
0&  0& 0& 1  & 0 \\ 
0& -6& -4& -4& -6\\
 0& 6 & 2 & 2  & 6  \\
 0 & 0&  0& 1 & 0& 
 \end{pmatrix}
$$
and has eigenvalues $-2,2\sqrt{3},-2\sqrt{3},0,0$. 

Note that this describes all the cases for $q=5$. Indeed, we have the group $\Bbb S_3$ acting 
on the values of $t$ (generated by $t\mapsto 1-t$ and $t\mapsto 1/t$) which acts transitively 
on the possible values $t=2,3,4$, and the stabilizer of $t=4$ is generated by $t\mapsto 1/t$ 
which exchanges the possible values $z=2,3$.  
\end{example} 

\setcounter{MaxMatrixCols}{20}

\begin{example} Let $q=7$, $t=6$, $z=2$. Then the 12-by-12 matrix of 
$qH_z$ in the basis $\bold 1_x,\psi_r$ is 
$$
\begin{pmatrix}
2& 2& 2& 1& 2& 2& 2& 2&   &   &   &   \\ 
2& 2& 2& 1& 2& 2& 2& 2&   &   &   &   \\ 
2& 2&   & 2&   &   & 1& 1&   &   & 1& 1\\ 
1& 1& 2&   &   &   & 2& 2& 1& 1&   &   \\ 
2& 2&   &   &   &   & 2& 2&   &   &   &    \\ 
2& 2&   &   &   &   & 2& 2&   &   &   &    \\ 
 2& 2& 1& 2& 2& 2& 2& 2&  &   &   &    \\ 
 2& 2& 1& 2& 2& 2& 2& 2&  &   &   &    \\ 
   &   &   & 7&   &   &   &   &   &   &  &    \\ 
   &   &   & 7&   &   &   &   &   &   &  &    \\ 
   &   &  7&  &   &   &   &   &   &   &  &    \\ 
   &   &  7&  &   &   &   &   &   &   &  & 
\end{pmatrix} 
$$
The eigenvalues of this matrix are $0$ (multiplicity $6$), 
$4,-2,-1\pm \sqrt{17},4\pm 6\sqrt{2}$. The Perron-Frobenius eigenvalue is 
$4+6\sqrt{2}$. 

Similarly, if instead $z=3$ (with the same $q,t$) then the eigenvalues are 
$0$ (multiplicity $6$), 
$4,-2,1\pm \sqrt{17},4\pm 6\sqrt{2}$. The Perron-Frobenius eigenvalue is still $4+6\sqrt{2}$. 

This covers all the cases with $q=7$, $t=2,4,6$ due to the $\Bbb S_3$-symmetry. 
\end{example} 

\begin{example} Let $q=7$, $t=3$, $z=2$.
Then the 12-by-12 matrix of 
$qH_z$ in the basis $\bold 1_x,\psi_r$ is 
$$
\begin{pmatrix}
2& 2& 2& 2& 2& 1& 2& 2&   &   &   &   \\ 
2& 2& 2& 2& 2& 2& 1& 2&   &   &   &   \\ 
2& 2&   & 2&   &   & 1& 1&   &   &  & 1\\ 
2& 2& 2& 2&  1&2& 2& 2&  &   &   &   \\ 
2& 2&   &  1&   & 1  & & 2&   &1   &   &    \\ 
1& 2&   &   2&  1 &   & & 2&   &   & 1  &    \\ 
 2& 1& 1& 2& & & & 2&       1&   &   &    \\ 
 2& 2& 1& 2& 2& 2& 2& 2&  &   &   &    \\ 
   &   &   & &   &   &  7 &   &   &   &  &    \\ 
   &   &   & &  7 &   &   &   &   &   &  &    \\ 
   &   &  &  &   &  7 &   &   &   &   &  &    \\ 
   &   &  7&  &   &   &   &   &   &   &  & 
\end{pmatrix} 
$$
The eigenvalues of this matrix are 
$0$ (multiplicity $3$), $\frac{-1\pm \sqrt{33}}{2}$ (each with multiplicity $2$), 
$\frac{1\pm \sqrt{33}}{2}$, $1$, $4\pm 6\sqrt{2}$. The Perron-Frobenius eigenvalue is again 
$4+6\sqrt{2}$. 

Via the $\Bbb S_3$-symmetry this covers all cases with $t=3$ except $z=6$. 
In this case the eigenvalues 
are $0$ (multiplicity $3$), $\frac{1\pm \sqrt{33}}{2}$ (each with multiplicity $3$), 
$1$, $4\pm 6\sqrt{2}$. The Perron-Frobenius eigenvalue is again 
$4+6\sqrt{2}$. 

This also covers the case $t=5$ via the $\Bbb S_3$-symmetry. Thus we have described all the cases with $q=7$. 
\end{example} 

\begin{remark} By a theorem of Drinfeld (\cite{Dr2}), the eigenvalues of $H_z$ on $V=V_1$ have absolute value 
$\le 2q^{-\frac{1}{2}}$ and have the form $2q^{-\frac{1}{2}}{\rm Re}\lambda$, where $|\lambda|=1$ 
and $\lambda$ are asymptotically uniformly distributed on the circle for large $q$. 
\end{remark} 

\section{Singularities of eigenfunctions}\label{sing}

In this section we study the singularities of eigenfunctions of Hecke operators for $G=PGL_2$ in genus zero (with parabolic points). We will use 
the Gaudin operators, so this analysis only applies to the archimedian case. However, we expect that the behavior of eigenfunctions near singularities is essentially the same over all local fields. 

\subsection{Singularities for $N=m+2$ parabolic points} \label{sing1}

Consider the (modified) Gaudin operators 
$$
G_i:=\sum_{0\le j\le m, j\ne i}\frac{1}{t_i-t_j}\left(-(y_i-y_j)^2\partial_i\partial_j+(y_i-y_j)(\partial_i-\partial_j)\right). 
$$
To study the behavior of eigenfunctions, we will quotient out the symmetry under the group $\Bbb G_m\ltimes \Bbb G_a$ of affine transformations of 
$\Bbb A^1$ and write these operators 
in the coordinates $y_1,...,y_{m-1}$ on $Bun_0^\circ$. 

First we quotient out the translation group $\Bbb G_a$. Set $t_0=0$ and consider the action of $G_i$ on shift-invariant functions, i.e., such that $\sum_i \partial_if=0$. Thus 
$$
\partial_0 f=-\sum_{1\le i\le m}\partial_i f.
$$ 
Substituting this and setting $y_0=0$, we have, for $1\le i\le m$:
$$
G_i=\frac{1}{t_i}\left(y_i^2\partial_i(\sum_{1\le j\le m} \partial_j)+y_i(2\partial_i+\sum_{j\ne i,1\le j\le m}\partial_j)\right)+
$$
$$
\sum_{j\ne i, 1\le j\le m}\frac{1}{t_i-t_j}\left(-(y_i-y_j)^2\partial_i\partial_j+(y_i-y_j)(\partial_i-\partial_j)\right). 
$$

Now we quotient out the dilation group $\Bbb G_m$. To this end consider the action of the Gaudin operators on functions of homogeneity degree $-m/2$. 
On such functions 
$$
\partial_mf=-y_m^{-1}(\tfrac{m}{2}+\sum_{i=1}^{m-1}y_i\partial_i)f.
$$
Substituting this and setting $y_m=1$, $t_m=1$ we get the following proposition

\begin{proposition} For $1\le i\le m-1$ we have 
$$
G_i=\frac{1}{t_i}\left(y_i^2\partial_i\left(\sum_{1\le j\le m-1} (1-y_j)\partial_j-\tfrac{m}{2}\right)+y_i\left((2-y_i)\partial_i+\sum_{j\ne i,1\le j\le m-1}(1-y_j)\partial_j-\tfrac{m}{2}\right)\right)+
$$
$$
\sum_{j\ne i,1\le j\le m-1}\frac{1}{t_i-t_j}\left(-(y_i-y_j)^2\partial_i\partial_j+(y_i-y_j)(\partial_i-\partial_j)\right)+
$$
$$
\frac{1}{t_i-1}\left((y_i-1)^2\partial_i(\sum_{1\le j\le m-1} y_j\partial_j+\tfrac{m}{2})+(y_i-1)((1+y_i)\partial_i+\sum_{1\le j\le m-1,j\ne i} y_j\partial_j+\tfrac{m}{2})\right).
$$
\end{proposition} 

Consider now the quantum Gaudin system 
$$
G_i\psi=\mu_i\psi. 
$$
As explained in Subsection \ref{hit}, generically on $Bun_0^\circ$, this is a holonomic system of rank $2^{N-3}$
whose singularities are located on the wobbly divisor. Let us consider its solutions near a generic point $\bold y^0$ of
the divisor $y_{m-1}=0$, which is a component of the wobbly divisor. 
So we will set $y_{m-1}^0=0$, $y_i^0=a_i$
for $i=1,...,m-2$ with $a_i$ generic, and $z=y_{m-1}$. 
We want to find a solution in the form 
$$
\psi(y_1,...,y_{m-2},z)=\sum_{n\ge 0} a_n(y_1,...,y_{m-2})z^{\lambda+n}. 
$$
To this end, we have to compute the leading 
term $G_i^0$ of $G_i$ with respect to $z$. For $i<m-1$ we get:  
$$
G_i^0=\frac{t_{m-1}}{t_{m-1}-t_i}y_i(y_i\partial_i+1)\partial_z,
$$
and for $i=m-1$ we get 
$$
G_{m-1}^0=\sum_{1\le j\le m-2}\frac{1}{t_j-t_{m-1}}y_j(y_j\partial_j+1)\partial_z+
\frac{1}{t_{m-1}-1}\left(\partial_z (z\partial_z+\sum_{j\le m-2}y_j\partial_j+\tfrac{m-2}{2})\right),
$$
both of degree $-1$. Thus we have 
$$
G_i^0( a_0(y_1,...,y_{m-2})z^\lambda)=0,\ 1\le i\le m-1.
$$
This yields 
\begin{equation}\label{a0eq}
\lambda(y_i\partial_i+1)a_0=0,
\end{equation} 
and
$$
(\lambda^2-\tfrac{m-2}{2}\lambda)a_0=0.
$$
We thus obtain $\lambda=0$ or $\lambda=\frac{m-2}{2}$, and for $m>2$ the space of solutions 
with $\lambda=\frac{m-2}{2}$ is $1$-dimensional. Also replacing $\lambda$ with $\varepsilon$ 
where $\varepsilon^2=0$ (over the base ring $\Bbb C[\varepsilon]/\varepsilon^2$), we 
see that if $m>2$ then there are no solutions of 
the form $ a_0(y_1,...,y_{m-2})\log(z)+O(1)$ with $a_0\ne 0$. 
Thus, applying permutations of indices, we arrive at the following proposition. 

\begin{proposition} \label{bounded} 
For $N\ge 5$ points all solutions of the Gaudin system are bounded
near a generic point $\bold y$  of the divisor $y_i=y_j$.
Hence single-valued eigenfunctions of Gaudin operators and in particular 
eigenfunctions of the Hecke operators are continuous (albeit not $C^\infty$) 
at $\bold y$. 
\end{proposition} 

\begin{remark} 1. More precisely, our analysis shows that eigenfunctions belong to the local H\"older space
${\bf H}^{\frac{m-2}{2}}$ for $F=\Bbb R$ and to ${\bf H}^{m-2}$ for 
$F=\Bbb C$ near $\bold y$.\footnote{Here for a nonnegative integer $k$ and 
$0<\alpha\le 1$, we write ${\bf H}^{k+\alpha}$ for the local H\"older space $C^{k,\alpha}$. 
In particular, for a positive integer $k$, we have ${\bf H}^{k+\delta}\subset C^k$ for any $\delta>0$, but $C^k$ is a proper subspace of ${\bf H}^k$.}  

2. On the contrary, we have seen that for $4$ points the eigenfunctions grow logarithmically 
at the singularities. 
\end{remark}

\subsection{The monodromy of the Gaudin system}

\begin{proposition} \label{bounded1} 
For $N\ge 5$ points all solutions of the Gaudin system are bounded
near a generic point $\bold y$  of the wobbly divisor $D$.
Hence single-valued eigenfunctions of the Gaudin operators and in particular 
eigenfunctions of the Hecke operators are continuous (although not $C^\infty$) at $\bold y$. 
\end{proposition} 

\begin{proof} This follows from Proposition \ref{dpan} and Proposition \ref{bounded}.
\end{proof}  

We also obtain the following results on the monodromy of Gaudin  systems for $F=\Bbb C$. 

\begin{proposition} The monodromy operator $\gamma$ of the Gaudin system 
around a component of $D$ is a {\bf reflection} for odd number of points $N$ (diagonalizable, 
one eigenvalue $-1$, other eigenvalues $1$) and a {\bf transvection} 
for even $N$ (unipotent, $\gamma-1$ has rank $1$). 
\end{proposition} 

\begin{proof} It follows from our description of the local behavior of solutions 
that for even $N$ all eigenvalues of $\gamma$ are $1$, and for odd $N$ all of them but one are $1$ 
and one eigenvalue is $-1$, and that $\gamma-1$ has rank $\le 1$. But if $\gamma=1$ then 
the corresponding $D$-module on an open set of $\Bbb P^{N-3}$ is monodromy-free, hence a multiple of $\mathcal O$. Thus it 
cannot be irreducible, a contradiction with Proposition \ref{irredu}. 
\end{proof} 

\begin{corollary} The eigenfunctions of Gaudin (or Hecke) operators 
near a generic point $\bold y$ of the wobbly divisor $D$ have the form 
$$
f=f_0+f_1|z|^{m-2}
$$
for odd $m$ and 
$$
f=f_0+f_1|z|^{m-2}\log|z|
$$
for even $m$, where $f_0$ and $f_1$ are real analytic functions near $\bold y$
and $z$ is a complex coordinate on $Bun_0^\circ$ such that $D$ is locally 
near $\bold y$ 
defined by the equation $z=0$. 
\end{corollary} 

\begin{proof} Let $\bold y\in D$. The above analysis implies that 
for odd $m$ there is a basis of solutions of the Gaudin system in which all but one element are holomorphic  
 at $\bold y$ and one element is 
 of the form $hz^{\frac{m-2}{2}}$ where $h$ is holomorphic at $\bold y$, and for even $m$ 
there is a basis with all but two elements holomorphic at $\bold y$, and the remaining two elements 
have the form $hz^{\frac{m-2}{2}}$ and $g+h z^{\frac{m-2}{2}}\log z$, where $g,h$ are holomorphic at $\bold y$. This implies the statement. 
\end{proof} 

\begin{corollary} The monodromy group of the quantum Gaudin system with any eigenvalues $\mu_i$ is generated by reflections for odd $N$ and by 
transvections for even $N$. 
\end{corollary} 

\begin{corollary} The monodromy groups of the quantum Gaudin systems corresponding to real opers 
(i.e., satisfied by eigenfunctions of Hecke operators) are real, i.e., 
contained in $GL(2^{N-3},\Bbb R)$ up to conjugation. 
\end{corollary} 

\begin{proof} Consider the corresponding monodromy group in 
$GL(2^{N-3},\Bbb C)$. 
By the argument of \cite{EFK2}, Remark 1.8, this monodromy group is
contained in some inner real form of $GL(2^{N-3},\Bbb C)$. There are only two such forms -- the split one with the group of real points $GL(2^{N-3},\Bbb R)$ and the quaternionic one with group of real points $GL(2^{N-4},\Bbb H)$. But the quaternionic form does not contain images of transvections or reflections, since the space of invariants of such an element has odd complex dimension, hence cannot be a quaternionic vector space. 
This implies the statement. 
\end{proof} 

\subsection{The case of 5 points} \label{The case of 5 points} 

Consider now the case of $N=5$ points, i.e., $m=3$. We let $t_1=s$, 
$t_2=t$, $y_1=y$, $y_2=z$. 

\subsubsection{Behavior of eigenfunctions}
 Recall (see e.g. \cite{C},\cite{DP}) that $Bun_0^\circ$ is $\Bbb P^2$ blown up at points $(0,0,1)$, $(0,1,0)$, $(1,0,0)$, $(1,1,1)$, $(s,t,1)$. 
According to \cite{DP}, Subsection 5.4 or Proposition \ref{dpan}, the components of the divisor $D$ 
are as follows: 

1) exceptional fibers over these $5$ points; 

2) 10 lines 
$$
y=0,\ y=1,\ y=s,\ z=0,\ z=1,z=s, y=z,\ ty=sz,\ (t-1)(y-1)=(s-1)(z-1),\ \text{line at }\infty;
$$

3) quadric $st(y-z)+(t-s)yz+sz-ty=0$. 

So there are 16 components, permuted transitively by the Weyl group 
$W(D_5)=\Bbb S_5\ltimes \Bbb V$, where $\Bbb V$ is the 4-dimensional reflection representation over 
$\Bbb F_2$. Namely, the set of components is $W(D_5)/W(A_4)=\Bbb S_5\ltimes \Bbb V/\Bbb S_5=\Bbb V$, where $\Bbb V$ acts by translations and $\Bbb S_5$ by reflections. The origin in $\Bbb V$ corresponds to the component over $(s,t,1)$ which is the component of bundles isomorphic to $O(1)\oplus O(-1)$. 

The components of the divisor $D$ can be subdivided according to the invariant $r=|S|/2$. Namely, 

$\bullet$ $r=0$ corresponds to the exceptional fiber over $(t,s,1)$ (bundles $O(-1)\oplus O(1)$; $\binom{5}{0}=1$ component); 

$\bullet$ $r=1$ corresponds to lines $y=0,y=1,z=0,z=1,z=y$, the line at infinity, and the remaining $4$ exceptional fibers at $(0,0,1)$, $(0,1,0)$, $(1,0,0)$, $(1,1,1)$ ($\binom{5}{2}=10$ components);

$\bullet$ $r=2$ corresponds to lines $y=s$, $z=t$, 
$ty=sz$, $(t-1)(y-1)=(s-1)(z-1)$, and the quadric 
($\binom{5}{4}=5$ components). 

The divisor has normal crossings, and the crossings define the {\bf Clebsch graph}, which is the regular 16-vertex, 40-edge graph obtained by the 5-dimensional hypercube graph by quotienting by the central symmetry. 

Thus using the standard theory of holonomic systems with regular singularities on a normal crossing divisor, we get the following theorem for $F=\Bbb C$. 

\begin{theorem}\label{main5} (i) Near a generic point of $D$ there is a basis $f_1,f_2,f_3,f_4$ of 
solutions of the Gaudin system such that $f_1,f_2,f_3$ are holomorphic 
and $f_4=h\sqrt{z}$ where $h$ is holomorphic, where $D$ is locally defined by the equation $z=0$. 

(ii) Near the intersection point of two components of $D$ there is a basis $f_1,f_2,f_3,f_4$ of 
solutions of the Gaudin system such that $f_1,f_2$ are holomorphic, $f_3=h_1\sqrt{w}$ and $f_4=h_2\sqrt{z}$ where $h_1,h_2$ are holomorphic and $D$ is locally defined by the equation $zw=0$. 

(iii) Single-valued eigenfunctions of the Gaudin operators 
are continuous, near a generic point of the divisor $D$ are of the form $\psi_0+\psi_1|z|$, 
where $\psi_0,\psi_1$ are real analytic, and near the intersection of two 
components 
of $D$ are of the form $\psi_0+\psi_1|z|+\psi_2|w|$ where $\psi_0,\psi_1,\psi_2$ are real analytic. 
Thus all of them are in $L^2$ and satisfy the Gaudin equations as distributions, i.e., coincide with eigenfunctions of the Hecke operators. Thus eigenvalues of Hecke operators are in bijection with real opers (i.e., the inclusion $\Sigma\hookrightarrow \mathcal R$ is a bijection). 

(iv) The monodromy of the Gaudin system can be conjugated into $GL_4(\Bbb 
R)$. Moreover, the monodromy 
around components of $D$ is conjugate to ${\rm diag}(1,1,1,-1)$
and near intersection of two components to $({\rm diag}(1,1,1,-1),{\rm diag}(1,1,-1,1))$. 
\end{theorem}  

Thus in the case of $G=PGL_2$ and $X=\Bbb P^1$ with $5$ parabolic points we have proved all the conjectures from \cite{EFK2} (for $F=\Bbb C$). 

\subsubsection{The Schwartz space} 
We can now geometrically describe the Schwartz space $\mathcal S$ in the case of $5$ points for $F=\Bbb C$.\footnote{We are grateful to D. Gaiotto and E. Witten for helping us formulate this description.} (The geometric description of the Schwartz space for $4$ points was given in \cite{EFK}). 

Namely, consider an eigenfunction $\psi$ near a generic point $\bold y$ of one of the components of $D$, say $D_\emptyset$, the exceptional line over $(s,t,1)\in \Bbb P^2$. By Theorem \ref{main5}(iii), we have $\psi=\psi_0+\psi_1|z|$ for smooth $\psi_0,\psi_1$ near $\bold y$. A more careful analysis shows that $\psi_1$ is not only smooth but has an additional property. To fomulate this property, recall that $\psi$ is not a function but rather a half-density. It turns out that the condition on the half-density 
$\psi_1|z|$ is that it has to be the pullback of a smooth half-density from a neighborhood of $(s,t,1)\in \Bbb P^2$ under the map $\pi: Bun_0^\circ\to \Bbb P^2$. Note that in suitable local coordinates 
$\pi$ is given by the formula 
$$
\pi(z,w)=(z,zw),
$$
so a pullback of a smooth half-density $f(z,u)|dz\wedge du|$ looks like 
$$
f(z,zw)|dz\wedge d(zw)|=f(z,zw)|z|\cdot |dz\wedge dw|.
$$
(note that it is {\bf not} smooth!). In other words, we find that in such coordinates
$$
\psi_1(z,w)=f(z,zw)
$$
for a smooth function $f$. This can be deduced from the Gaudin system for $\psi$.
For example, the fact that $\psi_1(0,w)=f(0,0)$ is constant follows from equation \eqref{a0eq}
after passing to the coordinates $z,w$. 

This leads to the following description of the Schwartz space $\mathcal{S}\subset \mathcal{H}$ defined 
in Subsection \ref{Schwartz space}.

\begin{theorem} For $N=5$ the Schwartz space $\mathcal S$ consists of continuous half-forms on $Bun_0^\circ$ which are 

(1) smooth outside $D$;

(2) Near a generic point of a component $D_S\subset D$ have the form 
$\psi_0+\psi_1|z|$, where $\psi_0$ is smooth and $\psi_1|z|$ 
is the pullback of a smooth half-density from the blow-down of $Bun_0^\circ$ along $D_S$; 

(3)  Near the intersection of two components $D_{S_1},D_{S_2}\subset D$ have the form 
$\psi_0+\psi_1|z|+\psi_2|w|$, where $\psi_0$ is smooth and $\psi_1|z|$, $\psi_2|w|$  
are pullbacks of smooth half-densities from the blow-downs of $Bun_0^\circ$ along $D_{S_1}$ and $D_{S_2}$, respectively.  
\end{theorem} 

\section{Appendix: auxiliary results} 
 
\subsection{Lemmas on integrals over local fields} 

Let $F$ be a local field. 

\begin{lemma}\label{l2} We have 
$$
\int_{u\in F: \norm{u}\le 1}\frac{\norm{du}}{\sqrt{\norm{u(u+\varepsilon)}}}\sim \log\norm{\varepsilon^{-1}},\ \varepsilon\to 0. 
$$ 
\end{lemma} 

\begin{proof} This is easily obtained from the change of variable $u=\varepsilon v$.  
\end{proof} 

Now for $x\in F$, $x\ne 0,1$ consider the distribution on $\Bbb P^1(F)$ given by  
$$
E_x(\phi):=\int_{F}\phi(s)\frac{\norm{ds}}{\sqrt{\norm{s(s-1)(s-x)}}},
$$
where $\phi$ is a smooth function on $\Bbb P^1(F)$. 

\begin{lemma}\label{l2ell} We have 
$$
E_x\sim \norm{x}^{-\frac{1}{2}}\log\norm{x}\delta_\infty, \ x\to \infty,
$$
where $\delta_\infty$ is the delta distribution at $\infty$. 
In other words, for any test function $\phi$, 
$$
\lim_{x\to \infty}\frac{E_x(\phi)}{\norm{x}^{-\frac{1}{2}}\log\norm{x}}=\phi(\infty). 
$$
\end{lemma} 

\begin{proof} This follows easily from Lemma \ref{l2} by making a change of variable $s\mapsto s^{-1}$. 
\end{proof} 

Let $\theta: F\to \Bbb C$ be the function defined by the formula 
$\theta(a)=1$ if $a$ is a square and $\theta(a)=0$ otherwise.

\begin{lemma}\label{l33} For $b\in F^\times$ one has 
$$
\int_F\left(\frac{2\theta(x^2-4b)}{\sqrt{\norm{x^2-4b}}}-\frac{1+{\rm sign}(\log\norm{x})}{\norm{x}}\right)\norm{dx}=-\log\norm{b}.
$$
\end{lemma} 

\begin{proof} We have 
$$
\int_{\norm{x}\le R}\frac{2\theta(x^2-4b)\norm{dx}}{\sqrt{\norm{x^2-4b}}}=
\int_{\norm{x}\le R, x^2-y^2=4b}\frac{\norm{dx}}{\norm{y}}.
$$
Making the change of variables $u=\frac{x-y}{2},v=\frac{x+y}{2}$, we get $v=bu^{-1}$, 
$x=u+bu^{-1}$, $y=u-bu^{-1}$. Thus we get $\frac{dx}{y}=\frac{du}{u}$, hence
$$
\int_{\norm{x}\le R, x^2-y^2=4b}\frac{\norm{dx}}{\norm{y}}=\int_{\norm{u+bu^{-1}}\le R} 
\norm{\frac{du}{u}}. 
$$
For large $R$, the latter integral is close to 
$$
\int_{\norm{u},\norm{bu^{-1}}\le R} 
\norm{\frac{du}{u}}=\int_{R^{-1}\norm{b}\le \norm{u}\le R}\norm{\frac{du}{u}}=
\int_{\norm{x}\le R}\frac{1+{\rm sign}(\log\norm{x})}{\norm{x}}\norm{dx}-\log \norm{b},
$$
in the sense that the difference goes to zero as $R\to \infty$. 
This implies the statement. 
\end{proof} 

\subsection{Elliptic integrals}\label{ellint} In particular, the asymptotics of Lemma \ref{l2ell} applies to the function 
$$
E(x):=E_x(1)=\int_{F}\frac{\norm{ds}}{\sqrt{\norm{s(s-1)(s-x)}}}.
$$
If $F=\Bbb R$, this function is a classical elliptic integral, and it is also 
expressible in terms of elliptic integrals for $F=\Bbb C$. Thus we will 
call $E(x)$ 
{\bf the elliptic integral of $F$.} 

\begin{example}\label{ellintex} Let us compute the elliptic integral $E(x)$ over a non-archimedian local field $F$ with residue field $\Bbb F_q$. 

\begin{proposition} For $x\in F$ with $0<\norm{x}\le 1,\norm{1-x}=1$, we have 
$$
E(x)=\frac{1+4q^{-\frac{1}{2}}+q^{-1}}{1-q^{-1}}\log q-\log\norm{x},
$$
For other $x\ne 0,1$, the value of $E(x)$ 
is computed from this formula using the equalities
$$
E(x)=E(1-x),\ 
E(x^{-1})=\norm{x}^{\frac{1}{2}}E(x).
$$ 
\end{proposition} 

\begin{proof} 
Let $\pi\in F$ be a uniformizer, then $\norm{\pi}=q^{-1}$. 
Since $\norm{x}\le 1$, we have $\norm{x}=q^{-k}$ for $k\ge 0$. 

We will compute $E_*(x):=\frac{1-q^{-1}}{\log q}E(x)$. We have 
$$
E_*(x)=\sum_{n\in \Bbb Z}E_n(x),
$$
where 
$$
E_n(x):=q^{-\frac{n}{2}}\int_{s: \norm{s}=q^n}\frac{\norm{ds}}{\sqrt{\norm{(s-1)(s-x)}}},
$$
with the usual normalization of $\norm{ds}$. Note that the change of variable 
$s'=xs^{-1}$ yields the identity 
$$
E_n(x)=E_{-k-n}(x). 
$$

For $s\in F$ with $\norm{s}=q^n$ let $\bar s\in \Bbb F_q^\times$ be the 
image of $\pi^ns$. Thus 
$$
E_n(x)=\sum_{a\in \Bbb F_q^\times}E_{n,a}(x), 
$$
where 
$$
E_{n,a}(x):=q^{-\frac{n}{2}}\int_{s: \norm{s}=q^n,\bar s=a}\frac{\norm{ds}}{\sqrt{\norm{(s-1)(s-x)}}}.
$$
Thus if $n>0$, we get 
$$
E_{n,a}(x):=q^{-\frac{n}{2}-1},\ E_n(x)=(q-1)q^{-\frac{n}{2}-1}.
$$
So 
$$
\sum_{n>0}E_n(x)=\sum_{n<-k}E_n(x)=q^{-\frac{3}{2}}\frac{q-1}{1-q^{-\frac{1}{2}}}=q^{-\frac{1}{2}}+q^{-1}.
$$

Assume now that $k=0$ (which can happen only if $q\ge 3$). Then it remains to compute $E_0(x)$. Here we will use that $\norm{1-x}=1$. 
For $a\ne 1,\bar x$ we have 
$$
E_{0,a}(x)=q^{-1},
$$
and 
$$
E_{0,1}(x)=E_{0,\bar x}(x)=q^{-\frac{1}{2}}\int_{\norm{u}\le 1}\frac{\norm{du}}{\norm{u}^{\frac{1}{2}}}=
q^{-\frac{1}{2}}\sum_{n=0}^\infty (1-q^{-1})q^{-\frac{n}{2}}=q^{-\frac{1}{2}}+q^{-1}.
$$
So altogether we get 
$$
E_*(x)=4(q^{-\frac{1}{2}}+q^{-1})+(q-3)q^{-1}=1+4q^{-\frac{1}{2}}+q^{-1}.
$$
Now assume that $k>0$. Then similarly for $0<j<k$ we have 
$$
E_j(x)=(q-1)q^{-1},
$$
and 
$$
E_0(x)=E_k(x)=(q-2)q^{-1}+q^{-\frac{1}{2}}+q^{-1}=(q-1)q^{-1}+q^{-\frac{1}{2}}. 
$$
So altogether we get 
$$
E_*(x)=2(q^{-\frac{1}{2}}+q^{-1})+(k+1)(q-1)q^{-1}+2q^{-\frac{1}{2}}=1+4q^{-\frac{1}{2}}+q^{-1}+k(1-q^{-1}). 
$$
Thus 
$$
E(x)=\left(\frac{1+4q^{-\frac{1}{2}}+q^{-1}}{1-q^{-1}}+k\right)\log q=\frac{1+4q^{-\frac{1}{2}}+q^{-1}}{1-q^{-1}}\log q-\log\norm{x},
$$
as claimed. 
\end{proof} 
\end{example} 

It is also useful to consider (over any local field $F$) the {\bf modified elliptic integral} 
$$
E_+(x):=2\int_{F}\frac{\theta(s(s-1)(s-x))\norm{ds}}{\sqrt{\norm{s(s-1)(s-x)}}},
$$ 
where we recall that $\theta(a)=1$ if $a$ is a square and $\theta(a)=0$ 
otherwise. This function computes the volume of the elliptic curve 
$C_x$ given by the equation $$v^2=u(u-1)(u-x)$$ over $F$ 
with respect to the Haar measure $\norm{\frac{du}{v}}$. 

\subsection{The Harish-Chandra theorem} 

Let $\bold G$ be a reductive algebraic group defined over a local field $F$, and $\rho: \bold G(F)\to {\rm Aut}(V)$ be an irreducible unitary representation of $\bold G(F)$. Let $dg$ be a Haar measure on $\bold G(F)$. 

\begin{theorem}\label{hcthm} (see \cite{HC} for the archimedian case and \cite{J} for the non-archimedian case) 
For any smooth compactly supported function $\phi$ on $\bold G(F)$,\footnote{As usual, over a non-archimedian field smooth means locally constant.}  the operator 
$$
A_\phi:=\int_{\bold G(F)} \rho(g)\phi(g)dg
$$
on $V$ is trace class. 
\end{theorem} 

\begin{remark} In the non-archimedian case, the operator $A_\phi$ is actually of finite rank.  
\end{remark} 

This allows us to define the {\bf Harish-Chandra character} of $\rho$ to be the distribution 
$\chi=\chi_\rho$ on $G$ such that 
$$
{\rm Tr}A_\phi=(\chi,\phi). 
$$ 

\subsection{The principal series representation $W$ and its Harish-Chandra character}\label{prinse} 

Let $F$ be a local field. Then the group $G=PGL_2(F)$ acts by unitary operators 
on the space of square-integrable half-densities $W:=L^2(\Bbb P^1(F))$. 

Explicitly, this space can be realized as the space $L^2(F)$ 
with the norm defined by the Haar measure and right action of 
$G$ given by the formula 
$$
\left(\rho\left(\begin{matrix} a& b\\ c& d\end{matrix}\right)f\right)(z)=\frac{\norm{ad-bc}^{\frac{1}{2}}}{\norm{cz+d}}f\left(\frac{az+b}{cz+d}\right).
$$
This is the principal series representation with parameter $\nu=0$. 

The following proposition is classical. 

\begin{proposition}\label{charsl2} \cite{GenFun}
(i) The Harish-Chandra character $\chi_W$ of the representation $W$ 
is a locally integrable function on $G$;

(ii) $\chi_W$ is given by the formula 
$$
\chi_W(g)=\frac{2}{\norm{\frac{x}{y}-\frac{y}{x}}}
$$
if $g$ is hyperbolic with a lift to $GL_2(F)$ having eigenvalues $x,y\in F$, and $\chi_W(g)=0$ otherwise. 
\end{proposition}    

\subsection{A lemma on algebraic groups} 

Let $\bold k$ be a field, $C$ an irreducible algebraic curve over $\bold k$. 
Let $\bold G$ be a connected algebraic group over 
$\bold k$ of dimension $d$ and $\xi: C\to \bold G$ a regular map.  
Suppose that $\bold G$ is generated by the image of $\xi$. 
Let $\xi_n: C^n\to \bold G$ be the map given by
$\xi_n(s_1,...,s_n)=\xi(s_1)...\xi(s_n)$, and $Z_n$ be the closure
of the image of $\xi_n$. 

\begin{lemma}\label{gt} $\dim Z_n=n$ for $n\le d$. 
In particular, $Z_n=\bold G$ and $\xi_n$ is dominant 
for any $n\ge d$. 
\end{lemma} 

\begin{proof} We have $Z_n\subset Z_{n+1}$ and $Z_n$ are irreducible closed subvarieties of $\bold G$. 
So if $Z_n\ne Z_{n+1}$ then $\dim Z_{n+1}>\dim Z_n$. 
On the other hand, if $Z_n=Z_{n+1}$ then $\xi(s)Z_n\subset Z_n$ for all 
$s$, hence 
$Z_n=\bold G$. This implies the statement. 
\end{proof} 

\subsection{Irreducibility of the character variety} 

Let $c_0,...,c_{m+1}\in \Bbb C^\times$ and ${\rm Loc}_m(c_0,...,c_{m+1})$ be the variety of irreducible rank $2$ local systems on $\Bbb C\Bbb P^1\setminus \lbrace t_0,...,t_{m+1}\rbrace$ with local monodromies at $t_i$ 
regular with eigenvalues $c_i,c_i^{-1}$. It is easy to show that 
this is a smooth variety of pure dimension $2(m-1)$.

The following lemma is well known, but we provide a proof for reader's convenience. 

\begin{lemma}\label{irreduci} For $m\ge 1$, the variety ${\rm Loc}_m(c_0,...,c_{m+1})$ is irreducible (equivalently, connected), and for $m\ge 2$ it is nonempty.
\end{lemma} 

\begin{example}\label{m=1} A simple computation shows that the variety ${\rm Loc}_1(c_0,c_1,c_2)$ consists of one point or is empty, and the latter happens iff $c_0^{\pm 1}c_1^{\pm 1}c_2^{\pm 1}=1$ for some choice of signs. 
\end{example} 

\begin{proof} Let $\widetilde{\rm Loc}_m(c_0,...,c_{m+1})$ be the variety of 
irreducible $m+2$-tuples of non-scalar $2$-by-$2$ 
matrices $(C_0,...,C_{m+1})$ such that $\prod_{i=0}^{m+1}C_i=1$ and the eigenvalues of $C_i$ are $c_i,c_i^{-1}$. Then $PGL_2$ acts freely on $\widetilde{\rm Loc}_m(c_0,...,c_{m+1})$ and ${\rm Loc}_m(c_0,...,c_{m+1})=\widetilde{\rm Loc}_m(c_0,...,c_{m+1})/PGL_2$. Thus $\widetilde{\rm Loc}_m(c_0,...,c_{m+1})$ is a smooth variety of pure dimension $2m+1$, and 
it suffices to show that it is irreducible.   

First consider the variety $\widetilde{\rm Loc}_m(c_0,...,c_{m+1})^{\rm red}$ of {\it reducible} tuples satisfying the same conditions (which is non-empty
only for special eigenvalues). Such a tuple  
preserves a line in the 2-dimensional space, and once such line is chosen, 
is determined by an extension between two 1-dimensional representations 
of the free group $F_{m+1}$ in which the generators map to $c_i^{\pm1}$. Thus 
\begin{equation}\label{reduci}
\dim \widetilde{\rm Loc}_m(c_0,...,c_{m+1})^{\rm red}\le m+2.  
\end{equation} 

We will prove the lemma by induction in $m$. 
The base case $m=1$ is discussed in Example \ref{m=1}. So let $m\ge 2$ and suppose the statement is known for smaller values. Let $C=C_{m}C_{m+1}$. 
It is easy to see that the matrix $C$ can be any non-scalar matrix. 
So we have a dominant map 
$$
\pi: {\rm Loc}_m(c_0,...,c_{m+1})\to SL_2
$$
which sends $(C_0,...,C_{m+1})$ to $C_{m}C_{m+1}$. 

Let us compute the fiber $\pi^{-1}(C)$. If $C$ is not a scalar and has eigenvalues $c,c^{-1}$ then $\pi^{-1}(C)$ consists of four parts: 

(i) $(C_0,...,C_{m-1},C)$ and $(C^{-1},C_{m},C_{m+1})$ are both irreducible. 
By the induction assumption the first one runs over a variety of dimension $2m-3$ and the second over a variety of dimension $1$ (as $C$ is fixed), so this piece 
has dimension $(2m-3)+1=2m-2$. Moreover, it is non-empty for generic $C$. So the union over all $C$ is non-empty of dimension $2m+1$, and it is irreducible by the induction assumption. 

(ii) $(C_0,...,C_m,C)$ is irreducible, but $(C^{-1},C_{m+1},C_m)$ is reducible. This can happen only for special eigenvalues of $C$ and gives dimension 
$(2m-3)+1=2m-2$ as well. So the union over all $C$ is of dimension $\le 2m$. 

(iii) $(C_0,...,C_{m-1},C)$ is reducible, but $(C^{-1},C_{m+1},C_m)$ is irreducible. 
Then by \eqref{reduci} 
the first tuple runs over a variety of dimension $\le m-1$ (as $C$ is fixed), while the second 
one runs over a variety of dimension $1$, so this piece has dimension $\le (m-1)+1=m$. 
As this occurs only for special eigenvalues of $C$, the union over all $C$ has dimension $\le m+2$. 

It remains to consider the case when $C=\pm 1$ is a scalar. These two cases are equivalent so we only consider $C=1$. Thus we have two options: 

(i) $(C_0,...,C_{m-1})$ is irreducible. Then we get a piece of dimension 
$(2m-3)+2=2m-1$ (the summand $2$ accounts for the choice of $C_m$ in its conjugacy class).  

(ii) $(C_0,...,C_{m-1})$ is reducible. By \eqref{reduci} 
this gives a piece of dimension 
$\le m+2$.

Since $m+2\le 2m$ and the variety 
$\widetilde{\rm Loc}_m(c_0,...,c_{m+1})$ has pure dimension, 
we obtain that this variety is irreducible, as claimed. 
\end{proof} 

\subsection{Interpolation of rational functions}\label{interpo} 

 Let $x_0,...,x_{2n}$ 
 be a fixed collection of distinct points on $\Bbb P^1$ over a field $\bold k$. Let ${\mathcal U}_n$ be the variety of rational functions $f(x)$ of degree $n$, i.e., defining 
 a degree $n$ map $\Bbb P^1\to \Bbb P^1$. We have a regular map $\iota_{\bold x}: {\mathcal U}_n\to (\Bbb P^1)^{2n+1}$ given by 
 $$
 \iota_{\bold x}(f):=(f(x_0),...,f(x_{2n})).
 $$
 
\begin{lemma}\label{interp} (i) $\iota_{\bold x}$ is an open embedding. 

(ii) If $x_i\in \bold k$ and $\bold s=(s_0,...,s_{2n})$ 
then $\iota_{\bold x}^{-1}(\bold s)=\frac{\det A(\bold x,\bold s,w)}{\det B(\bold x,\bold s,w)}$, 
where $A=(a_{ij}), B=(b_{ij})$ are the following square matrices of size $2n+1$: 
$$
a_{i0}=\tfrac{s_i}{w-x_i},\ b_{i0}=\tfrac{1}{w-x_i};\quad  a_{ij}=b_{ij}=s_ix_i^{j-1},\ 1\le j\le n;\quad a_{ij}=b_{ij}=x_i^{j-n-1},\ n+1\le j\le 2n.
$$
\end{lemma}  

\begin{proof} The lemma is classical but we give a proof for reader's convenience.  
If $\iota_{\bold x}(f_1)=\iota_{\bold x}(f_2)$ and $f_k=p_k/q_k$, $k=1,2$, where $p_k,q_k$ are polynomials of degree $n$, then $p_1q_2-q_1p_2$ 
is a polynomial of degree $\le 2n$
which vanishes at $2n+1$ points $x_0,x_1,...,x_{2n}$, so it vanishes 
identically and $f_1=f_2$. Thus $\iota_{\bold x}$ is injective. Since $\dim {\mathcal U}_n=2n+1$, 
this implies that $\iota_{\bold x}$ is dominant. Thus $\iota_{\bold x}$ is birational, and moreover 
Zariski's main theorem implies that it is actually an open embedding, proving (i). 
Part (ii) (due to Cauchy and Jacobi) follows from Cramer's rule
(see e.g. \cite{GM}). 
\end{proof} 

\subsection{The exponent of a sequence} \label{exponent} 

Let $\bold a:=\lbrace{a_i, i\ge 1\rbrace}$ be a sequence of nonnegative 
numbers such that $a_i\to 0$. Let ${\bf n}_{\bold a}(\varepsilon)$ be the 
number of $i$ such that $|a_i|\ge \varepsilon$. 
Then we have the following well known lemma from calculus. 

\begin{lemma}\label{seq} (i) Suppose that ${\bf n}_{\bold a}(\varepsilon)=O(\varepsilon^{-b})$ as $\varepsilon\to 0$, for some $b\ge 0$. Then for any $p>b$ we have $\sum_{i\ge 1} a_i^p<\infty$. 

(ii) If $\sum_{i\ge 1} a_i^p<\infty$ for some $p>0$ then ${\bf n}_{\bold a}(\varepsilon)=o(\varepsilon^{-p})$ as $\varepsilon\to 0$. 
\end{lemma} 

\begin{proof} By rescaling we may assume that $a_i<1$ for all $i$.  
Then 
$$
\sum_{i\ge 1}a_i^p\le \sum_{m\ge 1}\left(\frac{1}{m^p}-\frac{1}{(m+1)^p}\right){\bf n}_{\bold a}\left(\frac{1}{m+1}\right)\le C\sum_{m\ge 1}\frac{1}{m^{p+1-b}}
$$
for some $C>0$, which implies the statement. 

(ii) Assume the contrary, then there is a sequence $\varepsilon_m\to 0$ 
such that ${\bf n}_{\bold a}(\varepsilon_m)\ge K\varepsilon_m^{-p}$ for some $K>0$. 
We may choose this sequence in such a way that $\varepsilon_{m+1}\le \frac{\varepsilon_{m}}{2}$. Then 
$$
\sum_{m\ge 1}(\varepsilon_m^p-\varepsilon_{m+1}^p){\bf n}_{\bold a}(\varepsilon_m)\le \sum_{i\ge 1}a_i^p<\infty. 
$$
Thus 
$$
\sum_{m\ge 1}(\varepsilon_m^p-\varepsilon_{m+1}^p)\varepsilon_m^{-p}<\infty,
$$
which is a contradiction since $(\varepsilon_m^p-\varepsilon_{m+1}^p)\varepsilon_m^{-p}>1-2^{-p}$. 
\end{proof} 

Thus if ${\bf n}_a(\varepsilon)$ has at most polynomial growth in $\varepsilon^{-1}$ then 
the number $b=b(\bold a):={\rm inf}\lbrace{p: \sum_i a_i^p<\infty\rbrace}$ may also be described as ${\rm limsup} \frac{\log {\bf n}_{\bold a}(\varepsilon)}{\log(\varepsilon^{-1})}$. 

\begin{definition} We will call the {\bf reciprocal} $1/b$ 
of this number the {\bf exponent} of $\bold a$. 
\end{definition} 

For example, the exponent of the sequence $a_m=\frac{1}{m^p}$ is $p$. 

\subsection{Lemmas on compact operators} 

Let $A$ be a compact self-adjoint operator on a separable Hilbert space 
$\mathcal H$. Denote by ${\bf n}(A,\varepsilon)$ the number of eigenvalues of $A$ of magnitude $\ge \varepsilon$ counted with multiplicities. 

\begin{lemma}\label{compalem} (i) Let $A=B+E$, where $B$ has rank $r$ and $E$ has norm $<\varepsilon$. Then ${\bf n}(A,\varepsilon)\le r$. 

(ii) Let $D$ be another compact self-adjoint operator such $\norm{A-D}\le 
\delta$. Then for any $\varepsilon>0$, 
$$
{\bf n}(A,\delta+\varepsilon)\le {\bf n}(D,\varepsilon).
$$
\end{lemma} 

\begin{proof} (i) Let $V$ be the span of all eigenvectors of $A$ with eigenvalues 
of magnitude $\ge \varepsilon$. Then for any nonzero vector
$v\in V$, we have $|(Av,v)|\ge \varepsilon (v,v)$. 
On the other hand, $(Av,v)=(Bv,v)+(Ev,v)$ and $|(Ev,v)|<\varepsilon (v,v)$. 
Thus $Bv\ne 0$. Hence $\dim V\le r$. 

(ii) Let $\Pi$ be the orthogonal projection to the sum of eigenspaces of $D$ with eigenvalues of magnitude $\ge \varepsilon$. 
Then $D=\Pi D+(1-\Pi)D$, so $A=B+E$, where $B=\Pi D$ and 
$E=(1-\Pi)D+(A-D)$. Since $\norm{(1-\Pi)D}<\varepsilon$, we have  $\norm{E}< \delta+\varepsilon$. Also, the rank of $\Pi D$ is ${\bf n}(D,\varepsilon)$. Thus the result follows from (i). 
\end{proof} 

\begin{lemma}\label{Schatt} Let $T: L^2(M)\to L^2(M)$ be a bounded integral operator on an $n$-dimensional analytic $F$-manifold $M$ whose Schwartz kernel is supported on a submanifold of $M\times M$ of dimension $n+k$, where $k<n$. If ${\rm Tr}|T|^{2p}<\infty$ then $p\ge \frac{n}{k}$.\footnote{Here $|T|:=\sqrt{T^\dagger T}$.}
\end{lemma} 

\begin{proof} First suppose that $k+1<n$ and $p>1$ is such that ${\rm Tr}|T|^{2p}<\infty$. 
Pick an integral operator $A$ on $L^2(M)$ with Schwartz kernel supported
on an $n+1$-dimensional submanifold such that ${\rm Tr}|A|^{2n}<\infty$ and $A^{n-k-1}T^\dagger\ne 0$.  
Such $A$ exists; namely, this holds if the Schwartz kernel of $A$ is a smooth compactly supported half-density concentrated on an $n+1$-dimensional analytic submanifold of $M\times M$, both chosen in a sufficiently generic way. Then the operator $A^{n-k-1}T^\dagger$ has Schwartz kernel 
supported in codimension $1$, so it is not Hilbert-Schmidt. Hence 
$$
{\rm Tr}(T(A^{n-k-1})^\dagger A^{n-k-1}T^\dagger)={\rm Tr}(|T|^2|A^{n-k-1}|^2)=\infty.
$$
Thus by the H\"older inequality for Schatten norms
$$
({\rm Tr}|T|^{2p})^{\frac{1}{p}}({\rm Tr}|A^{n-k-1}|^{\frac{2p}{p-1}})^{\frac{p-1}{p}}=\infty.
$$ 
Hence ${\rm Tr}|A^{n-k-1}|^{\frac{2p}{p-1}}=\infty$. Using the H\"older 
inequality again, this yields
${\rm Tr}|A|^{\frac{2(n-k-1)p}{p-1}}=\infty$. This implies that 
$$
n>\frac{(n-k-1)p}{p-1},
$$
i.e. 
\begin{equation}\label{pbound} 
p>\frac{n}{k+1}. 
\end{equation} 
Now consider the operator $T^{\otimes N}$. It is an integral operator on 
the manifold $M^N$ of dimension $Nn$ with Schwartz kernel supported 
on a submanifold of dimension $Nn+Nk$. Also ${\rm Tr}|T^{\otimes N}|^{2p}=
({\rm Tr}|T|^{2p})^N<\infty$. So applying the bound \eqref{pbound} to $T^{\otimes N}$, 
we obtain 
$$
p>\frac{Nn}{Nk+1}.
$$
 Now sending $N\to \infty$, we get $p\ge \frac{n}{k}$, as claimed. 
\end{proof}

\subsection{Self-adjoint second order differential operators on the circle}\label{selfadjo} 
Let $L:=\partial b\partial+U$ be a second order differential operator on $\Bbb R/\Bbb Z$, where $b,U$ are smooth real functions such that $b$ has simple zeros (note that the number of zeros of $b$ is necessarily even). We'd like to study self-adjoint extensions of $L$ initially defined on the domain $C^\infty(\Bbb R/\Bbb Z)$, on which $L$ is symmetric. Our main example will be the Lam\'e operator $L=\partial z(z-1)(z-t) \partial+z$, $t\in \Bbb R$, $t\ne 0,1$, acting on half-densities on $\Bbb R$, which in the coordinate $x=\pi^{-1}{\rm arccot}(z)$ takes the form $L=\partial b\partial+U$ for 
$$
b(x)=\pi^{-2}\cos \pi x\sin \pi x (\sin \pi x-\cos \pi x)(t\sin\pi x-\cos\pi x), 
$$  
which has four simple zeros on $\Bbb R/\Bbb Z$, namely $0,\frac{1}{4},\frac{1}{2},\pi^{-1}{\rm arccot}(t)$. 

Consider the closure of $L$ on $C^\infty(\Bbb R/\Bbb Z)$, and let $V\subset L^2(\Bbb R/\Bbb Z)$ be the domain of the adjoint operator $L^\dagger$. 
Then $L^\dagger$ is not symmetric on $V$ if $b$ has zeros. So we may define a skew-hermitian form on $V$ given by $\omega(v,u)=(L^\dagger v,u)-(v,L^\dagger u)$. 

\begin{proposition} $\omega$ has rank $2r$, where $r$ is the number of zeros of $b$, and signature $(r,r)$. Moreover, the domain of the closure of 
$L$ is $V^\perp\subset V$, so $V/V^\perp$ has dimension $2r$. 
\end{proposition} 

\begin{proof} 
Let $x_i$, $i=1,...,r$, be the zeros of $b$. Note that $V$ is the space 
of $L^2$ functions $f$ on $\Bbb R/\Bbb Z$ such that $(bf')'\in L^2$ (thus 
$f$ is in $C^1$ outside $x_i$). Hence $bf'$ belongs to the Sobolev space $W^{1,2}$, so it is continuous. Let $bf'(x_i)=c_i$. 
Let $h$ be a smooth function on $\Bbb R/\Bbb Z$ except at $x_i$ such that 
$h=c_i\log|x-x_i|$ near $x_i$. 
Then  $h\in V$ and $bh'(x_i)=c_i$. Thus we have a subspace $V_0\subset V$ of codimension 
$r$  which is the space of $f\in V$ such that $c_i=0$. 

Let $f\in V_0$. Then $b(x)f'(x)=\int_{x_i}^x\phi(t)dt$, for each $i$, where $\phi\in L^2(\Bbb R/\Bbb Z)$. 
Thus 
$$
f(x)=\int \frac{\int_{x_i}^x\phi(t)dt}{b(x)}dx+\sum_i a_i\theta(x-x_i)+C
$$
near $x=x_i$, where $\theta$ is the Heaviside function.
By the Cauchy-Schwarz inequality we have 
$$
\left|\int_{x_i}^x\phi(t)dt\right|^2\le |x-x_i|\int_{x_i}^x|\phi(t)|^2dt,
$$
which means that 
$$
\int_{x_i}^x\phi(t)dt=o(|x-x_i|^{\frac{1}{2}}),\ x\to x_i.
$$
Thus $f$ is continuous on each side near $x_i$ (but may jump at $x_i$). 

Now, we have 
$$
\omega(v,\overline u)=\int_0^1 (bv')'udx-\int_0^1 v(bu')'dx.
$$
Integrating this by parts, we get 
$$
\omega(v,\overline u)=\lim_{\varepsilon\to 0}\sum_i (b(v'u-u'v)(x_i+\varepsilon)-b(v'u-u'v)(x_i-\varepsilon)).
$$
Assume that $v,u\in V_0$. Then, since $bu'(x_i)=bv'(x_i)$, and $u$ and $v$ are continuous near $x_i$ on each side, we get 
$$
\omega(v,\overline u)=0.
$$
Thus $V_0$ is an isotropic subspace of $V$ with respect to $\omega$, which shows that ${\rm rank}(\omega)\le 2r$ and $\dim V/V^\perp\le 2r$. 

It remains to show that this dimension is exactly $2r$. For this purpose for each $i$ let $v=\log|x-x_i|$ and $u=\theta(x-x_i)$ near $x=x_i$ 
(and smooth at other $x_j$). Then 
$$
\omega(v_i,\overline u_i)=\lim_{x\to x_i}b(x-x_i)(x-x_i)^{-1}=b'(x_i)\ne 0.
$$
On the other hand, 
$$
\omega(v_i,\overline v_j)=\omega(u_i,\overline u_j)=\omega(v_i,\overline u_j)=0
$$
for $i\ne j$, and $\omega(u_i,\overline u_i)=\omega(v_i,\overline v_i)=0$ for all $i$. 
This implies the statement. 
\end{proof} 

From the proof we also obtain the following corollary. 

\begin{corollary} (i) $V^\perp=V\cap C(\Bbb R/\Bbb Z)$, the space of all continuous functions in $V$. 

(ii) Near each point $x_i$ every element $f\in V$ can be uniquely written 
as 
$$
f(x)=c_i\log|x-x_i|+a_i\theta(x-x_i)+f_i(x),
$$
where $f_i$ is continuous. Moreover, we have 
$$
\omega(v,u)=\sum_{i=1}^r(c_i(v)\overline{a}_i(u)-a_i(v)\overline{c}_i(u)).
$$

(iii) The subspaces $W\subset V$ containing $V^\perp$ on which $L^\dagger$ is closed symmetric (respectively, self-adjoint)  are in natural bijection with isotropic (respectively, Lagrangian) subspaces in $V/V^\perp=\Bbb C^{2r}$ with respect to the skew-hermitian form $\sum_i(c_i\overline{a}_i'-a_i\overline{c}_i')$, via $W\mapsto W/V^\perp$. 

(iii) Let $C^\infty_W$ be the space of functions in $W$ 
which are smooth outside the points $x_i$ and of the form (ii) with $f_i$ 
smooth at each $x_i$. Then for any 
$V^\perp\subset W\subset V$, the operator $L$ maps $C^\infty_W$ 
to itself, and $W$ is the domain of the closure of $L$ initially defined on $C^\infty_W$. Moreover, $L$ is symmetric (respectively, essentially self-adjoint) on $C^\infty_W$ if and only if $W$ is isotropic (respectively, Lagrangian). 
\end{corollary}  

We have seen that one example of a space $W$ with Lagrangian $W/V^\perp$ is $W=V_0$, which may be characterized as the space of bounded functions in $V$. Another example is the space $W=V_1$ of functions which near $x_i$ look like 
$$
f(x)=c_i\log|x-x_i|+f_i(x),
$$
where $f_i$ is continuous. 

\begin{corollary} The operator $L$ is essentially self-adjoint on the spaces $C^\infty_{V_0}$ 
and $C^\infty_{V_1}$. 
\end{corollary} 

\begin{remark} 1. When $L$ is essentially self-adjoint, it has a discrete 
spectrum, determined from the corresponding singular Sturm-Liouville problem.  

2. If we want to have an invariant domain that contains the eigenfunctions of $L$, 
we should slightly enlarge the spaces $C^\infty_{V_0}$, $C^\infty_{V_1}$. 

Namely, they can be replaced with the space 
$\widetilde C^\infty_{V_0}$ of functions in $V$ which are smooth outside $x_i$ and near $x_i$ are of the form 
$$
f(x)=g_i(x)\theta(x-x_i)+f_i(x),
$$
where $f_i,g_i$ are smooth, and the space 
$\widetilde C^\infty_{V_1}$ of functions in $V$ which are smooth outside $x_i$ and near $x_i$ are of the form 
$$
f(x)=g_i(x)\log|x-x_i|+f_i(x),
$$
where $f_i,g_i$ are smooth. 
\end{remark}

\begin{remark} 1. The space $V$ is a module over $C^\infty(\Bbb R/\Bbb Z)$. 
Indeed, if $\phi\in C^\infty(\Bbb R/\Bbb Z)$ then 
$L(\phi f)=fL\phi+\phi Lf+2\phi'bf'$. It is clear that the first two summands are in $L^2$, and so is the last 
summand, since $bf'\in W^{1,2}$, hence in $L^2$. 

2. $V$ is invariant under changes of variable $x$ which do not move the points $x_i$. 
Indeed, let $x=g(u)$ be a change of variable. 
Then $L$ transforms to $L'=g'(u)^{-1}\partial_u b(g(u))g'(u)^{-1}\partial_u$. Thus we need to check that 
for $f\in V$ we have $\partial_u b(g(u))g'(u)^{-1}\partial_u f\in L^2$. Let $h(u)=\frac{b(g(u))}{b(u)g'(u)}$. 
So we need to check that $(bhf')'\in L^2$. We have 
$$
(bhf')'=hLf+h'bf', 
$$ 
which is in $L^2$ since $bf'\in W^{1,2}$, hence in $L^2$. 

3. Let $c\in C^\infty(\Bbb R/\Bbb Z)$, and consider the operator 
$$
L(b,c,U)=\partial b\partial+i(bc\partial+\partial bc)+U
$$ 
(so that the previous setting is recovered by putting $c=0$). 
This operator is symmetric on smooth functions 
if $b,c,U$ are real-valued. Then the space $V=V(b,c,U)$ of functions 
$f\in L^2(\Bbb R/\Bbb Z)$ and $L(b,c,U)f\in L^2(\Bbb R/\Bbb Z)$ is independent on 
$b,c,U$ (as long as the points $x_i$ are fixed). 
Indeed, by the $C^\infty$-version of 
\cite{EFK}, Lemma 13.2, locally near $x_i$ by changes of variable and left and right 
multiplication by invertible functions, $L(b,c,U)$ can be uniquely 
brought to the form $\partial x \partial$, and we have seen 
that such transformations preserve the spaces $V$. 

4. Our arguments show that the space $V^\perp$ (for any $b,c,U$) has the following explicit description. 
It is the space of functions of the form 
$f=\int \frac{gdx}{b}+C$, where $g\in W^{1,2}$ is such that $g(x_i)=0$ and 
$\int_{\Bbb R/\Bbb Z}\frac{gdx}{b}=0$. Namely, $g=bf'$. Note that the 
integral is (absolutely) convergent 
since any function in $W^{1,2}$ is $1/2$-H\"older continuous (Morrey inequality). 
For the same reason (continuity of $g$) this description is independent on the choice of $b$. 
\end{remark} 

Thus, we get the following proposition. 

\begin{proposition} The space $V(b,c,U)$ is independent on $b,c,U$ and can be described as the space of functions 
$$
f(x)=\sum_i c_i\theta(x-x_i)+\int\frac{g(x)}{b(x)}dx+C,\ c_i\in \Bbb C, 
$$
where $g\in W^{1,2}$ is such that $\int_{\Bbb R/\Bbb Z}\frac{g(x)}{b(x)}dx=-\sum_i c_i$ (where the integral is understood in the sense of principal value). In particular, 
the continuous part $f(x)-c_i\theta(x-x_i)$ of $f$ at $x_i$ is $\frac{1}{2}$-H\"older continuous at $x_i$ (so is ${\rm const}+O(|x-x_i|^{\frac{1}{2}})$ as $x\to x_i$). 
\end{proposition}


\begin{thebibliography}{EFK2}  

\bibitem[Be]{Be} F. Beukers, {\em Unitary Monodromy of Lam\'e
  Differential Operators}, Regular and Chaotic Dynamics {\bf 12},
  No. 6 (2007) 630--641.

\bibitem[BD1]{BD} A.~Beilinson and V.~Drinfeld, {\em Quantization of
    Hitchin's integrable system and Hecke eigensheaves}, Preprint
     \url{http://www.math.uchicago.edu/~drinfeld/langlands/QuantizationHitchin.pdf}

\bibitem[BD3]{BD:opers} A.~Beilinson and V.~Drinfeld, {\em Opers},
 arXiv:math/0501398.

\bibitem[C]{C} C. Casagrande, {\em Rank 2 quasiparabolic vector
  bundles on $\Bbb P^1$ and the variety of linear subspaces contained
  in two odd-dimensional quadrics}, Math. Z. {\bf 280} (2015) 981--988.

\bibitem[CC]{CC} D.V. Chudnovsky and G. V. Chudnovsky, {\em
  Computational problems in arithmetic of linear differential
  equations. Some Diophantine applications}, in Number theory,
  eds. D.V. Chudnovsky, e.a. pp. 12--49, Lecture Notes in Math. {\bf
    1383}, Springer, Berlin, 1989.

\bibitem[De]{De} P. Deligne, {\em Un th\'eor\'eme de finitude pour la
  monodromie}, in Discrete groups in geometry and analysis,
  ed. R. Howe, pp. 1--19, Prog. in Math. {\bf 67}, Birkh\"auser,
  1987.

\bibitem[Dr1]{Dr} V. G. Drinfeld, {\em Two-dimensional $l$-adic
  representations of the fundamental group of a curve over a finite
  field and automorphic forms on $GL(2)$}, Amer. J. Math.
  {\bf 105} (1983) 85--114.
  
\bibitem[Dr2]{Dr2} V. Drinfeld, {\em Langlands conjecture for $GL(2)$
  over functional fields}, Proc. of Int. Congress of Mathematicians
  (Helsinki, 1978), pp. 565---574, Acad. Sci. Fennica, Helsinki, 1980.

\bibitem[DP1]{DP} R. Donagi and T. Pantev, {\em Parabolic Hecke
  eigensheaves}, arXiv:1910.02357.
  
\bibitem[DP2]{DP2} R. Donagi and T. Pantev, private communication.
  
\bibitem[EFK1]{EFK} P. Etingof, E. Frenkel, and D. Kazhdan, {\em An
  analytic version of the Langlands correspondence for complex
  curves}, in Integrability, Quantization, and Geometry, dedicated to
  Boris Dubrovin, Vol. II, eds. S. Novikov, e.a., pp. 137--202,
  Proc. Symp. Pure Math. {\bf 103.2}, AMS, 2021 (arXiv:1908.09677).

\bibitem[EFK2]{EFK2} P. Etingof, E. Frenkel, and D. Kazhdan, {\em
  Hecke operators and analytic Langlands correspondence for curves
  over local fields}, arXiv:2103.01509.

\bibitem[EK]{EK} P. Etingof and D. Kazhdan, {\em Characteristic
  functions of $p$-adic integral operators}, arXiv:2101.05185.

\bibitem[Fa]{F} G. Faltings, {\em Real projective structures on
    Riemann surfaces}, Compositio Math. {\bf 48} (1983)
223--269.
  
\bibitem[Fr1]{F:icmp} E. Frenkel, {\em Affine algebras, Langlands
  duality and Bethe ansatz}, in Proc. of Int. Congress of Math.
  Phys. (Paris, 1994), ed. D. Iagolnitzer, pp. 606--642,
  International Press, 1995 (arXiv:qalg/9506003).
  
  \bibitem[FS]{FS} E. Frenkel and A. Szenes, {\em Thermodynamic Bethe
    ansatz and dilogarithm identities. I}, Math.Res. Lett. {\bf 2}
    (1995) 677--693.

\bibitem[Ga]{Gai:outline} D. Gaitsgory, {\em Outline of the proof of
    the geometric Langlands conjecture for $GL_2$}, Ast\'erisque {\bf
       370} (2015) 1--112.

\bibitem[GGP]{GenFun} I. M. Gelfand, M. I. Graev, and
  I. I. Piatetski-Shapiro, Generalized functions, Vol. 6,
  Representation theory and automorphic functions, AMS Chelsea
  Publishing {\bf 382}, 1969.

\bibitem[GT]{GT} F. Gliozzi and R. Tateo, {\em Thermodynamic Bethe
  ansatz and three-fold triangulations}, Int. J. Mod. Phys. {\bf A11}
  (1996) 4051--4064.

\bibitem[Go]{Go} W. Goldman, {\em Projective structures with Fuchsian
    holonomy}, J. Diff. Geom. {\bf 25} (1987) 297--326. 

\bibitem[GM]{GM} P. R. Graves-Morris, {\em Symmetrical formulas for
  rational interpolants}, J. Comp. and Appl.  Math. {\bf 10} (1984)
  107--111.

\bibitem[HS]{HS} P. R. Halmos and V. S. Sunder, Bounded integral
  operators on $L^2$-spaces, Springer-Verlag, Berlin, 1978.

\bibitem[HC]{HC} Harish-Chandra, {\em Representations of semisimple Lie
  groups III}, Trans. Amer. Math. Soc. {\bf 76} (1954) 234--253.

\bibitem[Hi]{Hi} N. Hitchin, {\em Stable bundles and integrable
  systems}, Duke Math. J. {\bf 54} (1987) 91--114.

\bibitem[J]{J} H. Jacquet, {\em Sur les repr\'esentations des groupes
  r\'eductifs p-adiques}, C. R. Acad. Sci. Paris S\'er. A-B {\bf 280}
  (1975) A1271--A1272.

\bibitem[K]{K} M. Kontsevich, {\em Notes on motives in finite
  characteristic}, in Algebra, Arithmetic, and Geometry, in honor of
  Yu.I. Manin, Vol. II, eds. Yuri Tschinkel and Yuri Zarhin,
  pp. 213--247, Prog. in Math. {\bf 270}, Birkh\"auser, 2010
  (arXiv:math/0702206).

\bibitem[KR]{KR} M. G. Krein and A. G. Rutman, {\em Linear operators
  leaving invariant a cone in a Banach space}, Uspekhi Mat. Nauk {\bf
  3}, No. 1 (1948) 3--95 (Amer. Math. Soc. Transl. Ser. I, {\bf 10}
  (1962) 199--325).

\bibitem[KNS]{KNS} A. Kuniba, T. Nakanishi, and J. Suzuki, {\em
  T-systems and Y-systems in integrable systems}, J.Phys. {\bf A44}
  (2011) 103001.

\bibitem[L]{La} R.P.~Langlands, {\em On analytic form of
    geometric theory of automorphic forms} (in Russian), Preprint
  \url{http://publications.ias.edu/rpl/paper/2678}

\bibitem[Mu]{Mu} S. Mukai, An introduction to invariants and moduli,
  Cambridge Stud. in Adv. Math. {\bf 81}, Cambridge University Press,
  2003.

\bibitem[MS]{MS} V. B. Mehta and C. S. Seshadri, {\em Moduli of vector
  bundles on curves with parabolic structures}, Math. Ann. {\bf 248}
  (1980) 205--239.

\bibitem[NR]{NR} M.S. Narasimhan and S. Ramanan, {\em Moduli of vector
  bundles on a compact Riemann surface}, Ann. of Math. {\bf 89} (1969)
  14--51.

\bibitem[Ru]{Ru} S. Ruijsenaars, {\em Hilbert-Schmidt Operators
  vs. Integrable Systems of Elliptic Calogero-Moser Type III. The Heun
  Case}, SIGMA {\bf 5} (2009), 049 (arXiv:0904.3250).

\bibitem[S]{S} C. S. Seshadri, Moduli of vector bundles on curves with
  parabolic structures, Bull. Amer. Math. Soc. {\bf 83} (1977)
  124--126.

\bibitem[Sk]{Skl} E.K. Sklyanin, {\em Separation of variables in the
  Gaudin model}, J. Sov. Math. {\bf 47} (1989) 2473--2488.

\bibitem[Ta]{Ta} L. Takhtajan, {\em On real projective connections,
  V.I. Smirnov's approach, and black hole type solutions of the
  Liouville equation}, Theor. Math. Phys. {\bf 181} (2014) 1307--1316.

\bibitem[Te]{Te} J. Teschner, {\em Quantisation conditions of the
  quantum Hitchin system and the real geometric Langlands
  correspondence}, in Geometry and Physics, in honour of Nigel
  Hitchin, Vol. I, eds.  Dancer, e.a., pp. 347--375, Oxford University
  Press, 2018 (arXiv:1707.07873).

\bibitem[udB]{udB} Niels uit de Bos, {\em An explicit geometric
  Langlands correspondence for the projective line minus four points},
  arXiv:1906.03240.

\bibitem[We1]{We1} A. Weil, {\em L'Int\'egration dans les groupes
  topologiques et ses applications}, Actualit\'es Sci. et Ind. {\bf
  1145}, Hermann, 1965.

\bibitem[We2]{We} A. Weil, {\em Ad\`eles et groupes alg\'ebriques},
  S\'eminaire Bourbaki, {\bf 5}, Exp. 186, pp. 249--257, 1959.

\bibitem[Z]{Z} D. Zagier, {\em The dilogarithm function}, in Frontiers
  in Number Theory, Physics, and Geometry II, eds. P. Cartier, e.a.,
  pp. 3--65, Springer, Berlin, 2007.

\end{thebibliography}
\end{document}